\documentclass[12pt, reqno]{amsart}

\usepackage{
amsmath,
amssymb,
amsthm,
mathrsfs,
amsfonts,
ytableau,
enumitem,
comment,
upgreek,
soul,
yhmath,
mathtools,
shuffle,
kotex,
array,
caption,
tabularx,
}

\usepackage{hyperref}

\hypersetup{
    colorlinks=true,
    linkcolor = blue,
    citecolor=magenta,
    urlcolor=cyan,
}

\usepackage[capitalise, nameinlink]{cleveref}

\crefname{equation}{}{}
\crefname{figure}{{\sc Figure}}{{\sc Figure}}

\usepackage[euler]{textgreek}

\usepackage{tikz,tikz-cd}

\usepackage[colorinlistoftodos, textwidth = 2.3cm]{todonotes}

\ytableausetup{mathmode, boxsize=1.2em}

\newtheorem{theorem}{Theorem}[section]
\newtheorem{proposition}[theorem]{Proposition}
\newtheorem{lemma}[theorem]{Lemma}

\newtheorem{conjecture}[theorem]{Conjecture}

\newtheorem*{claim*}{Claim}

\theoremstyle{definition}
\newtheorem{algorithm}[theorem]{Algorithm}
\newtheorem{example}[theorem]{Example}
\newtheorem{definition}[theorem]{Definition}
\newtheorem{remark}[theorem]{Remark}

\numberwithin{equation}{section} 
\numberwithin{figure}{section}
\numberwithin{table}{section}

\def\Z{\mathbb Z}
\def\C{\mathbb C}

\newcommand{\nc}{\newcommand}

\nc{\SG}{\mathfrak{S}}
\nc{\SRT}{\mathrm{SRT}}
\nc{\comp}{\mathrm{comp}}
\nc{\Des}{\mathrm{Des}}
\nc{\set}{\mathrm{set}}
\nc{\ch}{\mathrm{ch}}
\nc{\id}{\mathrm{id}}
\nc{\Sym}{\mathrm{Sym}}
\nc{\QSym}{\mathrm{QSym}}
\nc{\Ext}{\mathrm{Ext}}
\nc{\source}{\mathsf{source}}
\nc{\sink}{\mathsf{sink}}
\nc{\len}{\mathsf{len}}
\nc{\ind}{\mathtt{index}}
\nc{\col}{\mathsf{col}}
\nc{\row}{\mathsf{row}}
\nc{\rad}{\mathrm{rad}}
\nc{\IGLT}{\mathrm{IGLT}}
\nc{\Top}{\mathsf{Top}}
\nc{\Bot}{\mathsf{Bot}}
\nc{\pr}{\mathsf{pr}}
\nc{\Gr}{\mathrm{Gr}}
\nc{\RS}{\star}

\nc{\calS}{\mathcal{S}}
\nc{\calI}{\mathcal{I}}
\nc{\calR}{\mathcal{R}}
\nc{\calG}{\mathcal{G}}
\nc{\calP}{\mathcal{P}}
\nc{\calE}{\mathcal{E}}

\nc{\bfS}{\mathbf{S}}
\nc{\bfF}{\mathbf{F}}
\nc{\bfP}{\mathbf{P}}
\nc{\bfG}{\mathbf{G}}

\nc{\rmr}{\mathrm{r}}
\nc{\rmc}{\mathrm{c}}
\nc{\rmt}{\mathrm{t}}
\nc{\rmw}{\mathrm{w}}
\nc{\rmv}{\mathrm{v}}

\nc{\sfB}{\mathsf{B}}
\nc{\sfb}{\mathsf{b}}
\nc{\sft}{\mathsf{t}}
\nc{\sfp}{\mathsf{p}}
\nc{\sfq}{\mathsf{q}}
\nc{\sfA}{\mathsf{A}}
\nc{\sfD}{\mathsf{D}}

\nc{\scrT}{\mathscr{T}}

\nc{\tH}{\mathtt{H}}
\nc{\tC}{\mathtt{C}}
\nc{\tyd}{\mathtt{yd}}
\nc{\trd}{\mathtt{rd}}
\nc{\tHL}{\mathtt{HL}}

\nc{\sfhA}{\widehat{\mathsf{A}}}
\nc{\sfread}{\mathsf{read}}
\nc{\sfLread}{\Phi}
\nc{\sfhD}{\widehat{\mathsf{D}}}

\nc{\itread}{\mathit{read}}

\nc{\ra}{\rightarrow}
\nc{\opi}{\overline{\pi}}
\nc{\bal}{{\boldsymbol{\upalpha}}}
\nc{\bfpi}{\boldsymbol{\uppi}}
\nc{\tGam}{\widetilde{\Gamma}}
\nc{\hGam}{\widehat{\Gamma}}

\nc{\GS}[1]{U_{#1}}
\nc{\GSm}[2]{U_{#1;#2}}
\nc{\grco}[2]{(\underline{#1,#2})}
\nc{\IGLTm}[2]{\mathrm{IGLT}(#1)_{#2}}

\nc{\yh}[1]{\todo[size=\tiny,color=cyan!10]{#1 \\ \hfill --- Young-Hun}}
\nc{\YH}[1]{\todo[size=\tiny,inline,color=cyan!10]{#1
		\\ \hfill --- Young-Hun}}
		
\nc{\sm}[1]{\todo[size=\tiny,color=magenta!10]{#1 \\ \hfill --- Semin}}
\nc{\SM}[1]{\todo[size=\tiny,inline,color=magenta!10]{#1
		\\ \hfill --- Semin}}

\nc{\nt}[1]{\todo[size=\tiny,color=exgreen!10]{#1 \\ \hfill --- Note}}
\nc{\NT}[1]{\todo[size=\tiny,inline,color=exgreen!10]{#1
		\\ \hfill --- Note}}

\definecolor{purple}{rgb}{0.44, 0.0, 1.0}
\definecolor{exgreen}{cmyk}{0.8, 0.1, 1, 0}

\definecolor{yhblue}{rgb}{0,0,0.6}

\newenvironment{red}{\relax\color{red}}{\hspace*{.5ex}\relax}
\newenvironment{blue}{\relax\color{yhblue}}{\hspace*{.5ex}\relax}
\newenvironment{green}{\relax\color{wsgreen}}{\hspace*{.5ex}\relax}
\newenvironment{magenta}{\relax\color{magenta}}{\hspace*{.5ex}\relax}

\nc{\ber}{\begin{red}}
\nc{\er}{\end{red}}
\nc{\beb}{\begin{blue}}
\nc{\eb}{\end{blue}}
\nc{\bema}{\begin{magenta}}
\nc{\ema}{\end{magenta}}
\nc{\begr}{\begin{green}}
\nc{\egr}{\end{green}}

\title[Weak Bruhat interval modules for genomic Schur functions]{Weak Bruhat interval modules \\ for genomic Schur functions}

\author[Y.-H. Kim]{Young-Hun Kim}
\address{Center for quantum structures in modules and spaces, Seoul National University, Seoul 08826, Republic of Korea}
\email{ykim.math@gmail.com}

\author[S. Yoo]{Semin Yoo}
\address{Discrete Mathematics Group, Institute for Basic Science, 55 Expo-ro Yuseong-gu, Daejeon 34126, South Korea}
\email{syoo19@ibs.re.kr}

\thanks{The first author was supported by the National Research Foundation of Korea(NRF) grant funded by the Korea government (NRF-2020R1A5A1016126 and NRF-2022R1A2C1004045) and Basic Science Research Program through NRF funded by the Ministry of Education (RS-2023-00240377). 
The second author was supported by the KIAS Individual Grant (CG082701) at
Korea Institute for Advanced Study and the Institute for Basic Science (IBS-R029-C1).}


\date{\today}

\begin{document}

\maketitle

\begin{abstract}
Let $\lambda$ be a partition of a positive integer $n$. The genomic Schur function $U_\lambda$ was introduced by Pechenik--Yong in the context of the $K$-theory of Grassmannians. Recently, Pechenik provided a positive combinatorial formula for the fundamental quasisymmetric expansion of $U_\lambda$ in terms of increasing gapless tableaux. In this paper, for each $1 \le m \le n$, we construct an $H_m(0)$-module $\mathbf{G}_{\lambda;m}$ whose image under the quasisymmetric characteristic is the $m$th degree homogeneous component of $U_\lambda$ by defining an $H_m(0)$-action on increasing gapless tableaux. We provide a method to assign a permutation to each increasing gapless tableau, and use this assignment to decompose $\mathbf{G}_{\lambda;m}$ into a direct sum of weak Bruhat interval modules. Furthermore, we determine the projective cover of each summand of the direct sum decomposition.
\end{abstract}

\tableofcontents

\section{Introduction}
Let $X = \Gr_k(\C^n)$ be the Grassmannian of $k$-dimensional subspaces of $\C^n$.
The Schur functions play a central role to understand the structure of the cohomology ring $H^*(X,\Z)$.
For example, the structure constants of Schur functions, called the \emph{Littlewood-Richardson coefficients}, are equal to that of \emph{Schubert classes} $[X_\lambda]$ ($\lambda \in \mathrm{Rec}_{k,n-k}$) which form a $\Z$-basis of $H^*(X,\Z)$.
Here, $\mathrm{Rec}_{k,n-k}$ is the set of partitions whose Young diagrams are contained in a $k \times (n-k)$ rectangle and $X_\lambda$ is the \emph{Schubert variety} associated to $\lambda$.
For more information, see \cite[Part III]{97Ful}.

Since the early 2000s, several combinatorial interpretations for the \emph{$K$-theoretic Littlewood-Richardson rule} have been introduced.
We briefly introduce the relevant results.
Let $K(X)$ be the Grothendieck ring of algebraic vector bundles over $X$.
It is well known that the classes of structure sheaves $[\mathcal{O}_{X_\lambda}]$ ($\lambda \in \mathrm{Rec}_{k,n-k}$) of $X_\lambda$ form a $\Z$-basis of $K(X)$ (for instance, see \cite[Remark 3.4.2]{05Brion}).
For $\lambda, \mu, \nu \in \mathrm{Rec}_{k,n-k}$, let $c^\nu_{\lambda,\mu}$ be the integer defined by 
\[
[\mathcal{O}_{X_\lambda}] \cdot[\mathcal{O}_{X_\mu}] = \sum_{\nu \in \mathrm{Rec}_{k,n-k}} c^\nu_{\lambda,\mu} [\mathcal{O}_{X_\nu}].
\]
The coefficient $c^\nu_{\lambda,\mu}$ is called a \emph{$K$-theoretic Schubert structure constant}, or sometimes a \emph{$K$-theoretic Littlewood-Richardson coefficient}.
In \cite{02Buch}, Buch provided the first combinatorial description for $(-1)^{|\nu| - |\lambda| - |\mu|} c^\nu_{\lambda,\mu}$ by using set valued tableaux.
Afterwards, several combinatorial models for the $K$-theoretic Schubert structure constants have been constructed.
For instance, see \cite{04KY, 17PY2, 09TY, 06Vakil}.
In particular, Pechenik and Yong \cite{17PY2} gave a combinatorial description for $(-1)^{|\nu| - |\lambda| - |\mu|} c^\nu_{\lambda,\mu}$ by using \emph{genomic tableaux}.

Genomic tableaux were first introduced by Pechenik and Yong \cite{17PY} in the context of the torus-equivariant $K$-theoretic Schubert calculus.
These were defined as edge-labeled tableaux with certain conditions and used as the key object in the first proof of a conjecture of Thomas and Yong \cite{18TY} on the torus-equivariant $K$-theoretic Schubert structure coefficients $K^\nu_{\lambda,\mu}$, where $\lambda, \mu, \nu \in \mathrm{Rec}_{k,n-k}$.
Soon after, in \cite{17PY3}, genomic tableaux were also used to prove a mild modification of a conjecture of Knutson and Vakil \cite{09CV} on $K^\nu_{\lambda,\mu}$.
In a sequel paper \cite{17PY2}, Pechenik and Yong studied combinatorial theory of non-edge-labeled genomic tableaux, providing a combinatorial description for $(-1)^{|\nu| - |\lambda| - |\mu|} c^\nu_{\lambda,\mu}$ in terms of genomic tableaux as mentioned above.
Therein, they defined a symmetric function $U_\lambda$, called the \emph{genomic Schur function}, as a generating function for genomic tableaux of shape $\lambda$ for all partitions $\lambda$.
Further, they proved that $\{U_\lambda \mid \text{$\lambda$ is a partition}\}$ is a basis for the ring of symmetric functions and pointed out that genomic Schur functions are not Schur-positive in general.
As an alternative positivity, Pechenik~\cite{20Pechenik} showed that genomic Schur functions are fundamental positive.
Specifically, for any partition $\lambda$,
\begin{align}\label{eq: GS F expansion}
\GS{\lambda} = \sum_{T \in \IGLT(\lambda)} F_{\comp(T)},
\end{align}
where $\IGLT(\lambda)$ is the set of \emph{increasing gapless tableaux} of shape $\lambda$, $\comp(T)$ is the composition associated to $T$, and $F_{\comp(T)}$ is the \emph{fundamental quasisymmetric function} associated to $\comp(T)$.
For the precise definitions, see \cref{subsec: Genomic Schur}. 
In addition, Pechenik left remarks on interpretations of \cref{eq: GS F expansion} in terms of representation theory of the $0$-Hecke algebras.
Before discussing these remarks, we review the representation theory of $0$-Hecke algebras.

The \emph{$0$-Hecke algebra $H_n(0)$} is the $\C$-algebra obtained from the Hecke algebra $H_n(q)$ by specializing $q$ to $0$.
In \cite{79Norton}, Norton classified all irreducible $H_n(0)$-modules up to isomorphism.
These modules correspond in a natural way to compositions $\alpha$ of $n$.
We denote by $\bfF_\alpha$ the irreducible module corresponding to $\alpha$. 
Duchamp, Krob, Leclerc, and Thibon \cite{96DKLT} revealed a deep connection between the representation theory of the $0$-Hecke algebras and the ring $\QSym$ of quasisymmetric functions by introducing the \emph{quasisymmetric characteristic}, which is a ring isomorphism 
$$
\ch: \bigoplus_{n \ge 0} \calG_{0}(H_n(0)\text{\bf -mod}) \ra \QSym, \quad [\bfF_\alpha] \mapsto F_\alpha.
$$
Here, $H_n(0)\text{\bf -mod}$ is the category of the finite dimensional $H_n(0)$-modules, $\calG_{0}(H_n(0)\text{\bf -mod})$ is the Grothendieck group of $H_n(0)\text{\bf -mod}$, and $\bigoplus_{n \ge 0} \calG_{0}(H_n(0)\text{\bf -mod})$ is considered as the ring equipped with the induction product.
In view of this correspondence, there have been considerable attempts to provide a representation theoretic interpretation of noteworthy quasisymmetric functions by constructing appropriate $0$-Hecke modules.
For instance, Berg {\it et al.} \cite{15BBSSZ} provided such interpretation for the \emph{dual immaculate functions}, Tewari and van Willigenburg \cite{15TW} for the \emph{quasisymmetric Schur functions}, and Searles \cite{19Searles} for the \emph{extended Schur functions}.
Readers interested in relevant results may also refer to \cite{22BS, 20CKNO, 22CKNO2, 20CKNO2, 22JKLO, 19TW}.

Very recently, Jung, Kim, Lee, and Oh \cite{22JKLO} introduced the \emph{weak Bruhat interval module $\sfB(\sigma, \rho)$ associated to $[\sigma,\rho]_L$} to provide a unified method to study the $H_n(0)$-modules introduced in \cite{22BS, 15BBSSZ, 20CKNO2, 19Searles, 15TW, 19TW}.
Here, $\sigma$ and $\rho$ are arbitrary permutations in the symmetric group $\SG_n$ and $[\sigma,\rho]_L$ is the left weak Bruhat interval from $\sigma$ to $\rho$.
Indeed, they proved that all indecomposable direct summands of these $H_n(0)$-modules are contained in the family of weak Bruhat interval modules up to isomorphism.
They also investigated several structural properties of weak Bruhat interval modules such as embeddings into the regular representation of $H_n(0)$, the induction product, restrictions, and (anti-)involution twists of weak Bruhat interval modules.
In addition, they implicitly remarked that $\bigoplus_{n \ge 0} \calG_{0}(\mathscr{B}_n)$ is isomorphic to $\QSym$, where $\mathscr{B}_n$ is the full subcategory of $H_n(0)\text{\bf -mod}$ whose objects are direct sums of weak Bruhat interval modules up to isomorphism, $\calG_{0}(\mathscr{B}_n)$ is the Grothendieck group of $\mathscr{B}_n$, and $\bigoplus_{n \ge 0} \calG_{0}(\mathscr{B}_n)$ is considered as the ring equipped with the induction product.
For more information, see \cite{22JKLO}.
Another uniform way to study various $0$-Hecke modules was also proposed in \cite{22Searles} by introducing \emph{diagram modules}.

The aforementioned remarks of Pechenik are concerned with the problem of finding an appropriate $0$-Hecke module whose image, under the quasisymmetric characteristic, is a homogeneous component of the genomic Schur function $U_{\lambda}$. 
It was shown that the fifth degree homogeneous component of $U_{(3,3)}$ cannot be the image of quasisymmetric characteristic of any projective $H_{5}(0)$-module. 
The problem for finding indecomposable $0$-Hecke modules for a homogeneous component of $U_{\lambda}$ was also considered. 
However, there was no successful answer for this problem. 
We show here that for some partition $\lambda$ of $n$ and $1\le m \le n$, it is impossible to construct an indecomposable $H_m(0)$-module $M$ such that $\ch([M])$ is the $m$th degree homogeneous component $\GSm{\lambda}{m}$ of $U_\lambda$.
Precisely, by considering the $\Ext$-group between irreducible $H_3(0)$-modules, we show that there does not exist any indecomposable $H_3(0)$-module $M$ such that $\ch([M])$ is the third degree homogeneous component of $U_{(2,1,1)}$ (\cref{rem: counterexample for indecomposable module}).

The purpose of this paper is to provide a nice representation theoretic interpretation of genomic Schur functions.
To achieve our purpose, we construct an $H_m(0)$-module $\mathbf{G}_{\lambda;m}$ by defining an $H_m(0)$-action on the $\C$-span of the set $\IGLTm{\lambda}{m}$ of increasing gapless tableaux of shape $\lambda$ with maximum entry $m$.
The $H_m(0)$-module $\mathbf{G}_{\lambda;m}$ is well suited for our purpose in that it satisfies the following properties:
\begin{enumerate}[label = $\bullet$]
\item The image of $\bfG_{\lambda;m}$ under the quasisymmetric characteristic is $\GSm{\lambda}{m}$. 
\item $\mathbf{G}_{\lambda;m}$ can be decomposed into a direct sum of weak Bruhat interval modules $\bfG_E$.
\item The projective cover of $\bfG_E$ is also a weak Bruhat interval module.
\end{enumerate}
While obtaining the decomposition of $\mathbf{G}_{\lambda;m}$, we introduce a method of assigning a permutation to each increasing gapless tableau.
What is fascinating is that this assignment allows us to study increasing gapless tableaux using the various properties of permutations. For more information, see \cref{Sec: Further avenues}(1).
In addition, we leave a conjecture on a representation theoretic interpretation of Pechenik's combinatorial formula for the Schur expansion of certain $U_\lambda$ (\cref{conj: repn interpretation}).

From now on, we describe our results in more detail.
For convenience, we fix a partition $\lambda$ of $n$ and a positive integer $m$ less than or equal to $n$ unless otherwise stated.

In \cref{Sec: 0-Hecke action}, we construct an $H_m(0)$-module $\bfG_{\lambda;m}$ by defining an $H_m(0)$-action on the $\C$-span of $\IGLTm{\lambda}{m}$ (\cref{thm: 0-Hecke action on IGLT}).
We then provide a direct sum decomposition of $\bfG_{\lambda;m}$ into $H_m(0)$-submodules which will turn out to be weak Bruhat interval modules.
To do this, we define an equivalence relation $\sim_{\lambda;m}$ on $\IGLTm{\lambda}{m}$ (\cref{def: equiv relation}).
Let $\calE_{\lambda;m}$ be the set of equivalence classes of $\IGLTm{\lambda}{m}$ with respect to $\sim_{\lambda;m}$.
We prove that the $\C$-span of each equivalence class $E \in \calE_{\lambda;m}$ is closed under the $H_m(0)$-action (\cref{thm: preserving}), and we thus obtain the direct sum decomposition
\[
\bfG_{\lambda;m} = \bigoplus_{E \in \calE_{\lambda;m}} \bfG_E,
\]
where $\bfG_E$ is the $H_m(0)$-submodule of $\bfG_{\lambda;m}$ whose underlying space is the $\C$-span of $E$.
Hereafter, we fix an equivalence class $E \in \calE_{\lambda;m}$.

In \cref{Sec: source and sink}, we prove the existence and uniqueness of the \emph{source tableau} and those of the \emph{sink tableau} in $E$, which play an important role in verifying that $\bfG_E$ is isomorphic to a weak Bruhat interval module.
For the precise definitions of source and sink tableaux, see \cref{def: source and sink tableaux}.
Let us briefly describe our strategy to prove the existence and uniqueness of source tableaux.
We first give a characterization for source tableaux (\cref{Lem: source and sink cond}).
Considering this characterization, we design an algorithm to construct a tableau $\source(T)$, where $T$ is an arbitrary tableau in $E$ (\cref{Alg: construct source of T}).
Then, we see that $\source(T)$ is a source tableau with $\source(T) \sim_{\lambda;m} T$, which proves the existence of source tableaux in $E$.
Next, as a key lemma to prove the uniqueness of source tableaux in $E$, we verify that $\source(T) = T$ for any source tableau $T \in E$ (\cref{lem: tS(T)}).
Finally, combining this lemma with an observation that $\source(T_1) = \source(T_2)$ for any $T_1, T_2 \in E$, we obtain our desired result (\cref{thm: uniqueness of source}).
In a similar manner, we also prove the existence and uniqueness of sink tableaux in $E$.
For details, see \cref{subsec: sink tableau}.
We denote by $T_E$ and $T'_E$ the unique source and sink tableaux, respectively.

In \cref{Sec: weak bruhat interval module}, we prove that $\bfG_E$ is isomorphic to a weak Bruhat interval module of $H_m(0)$.
To begin with, we introduce a relation $\preceq_{E}$ on $E$ defined by 
\[
T_1  \preceq_E  T_2 \quad \text{if $\pi_\sigma \cdot T_1 = T_2$ for some $\sigma \in \SG_m$}.
\]
Then, to each $T \in E$ we assign a permutation $\sfread(T) \in \SG_m$, called \emph{the standardized reading word}.
With these preparations, we prove that $T_E \preceq_E T \preceq_E T'_E$ for all $T \in E$ and $(E, \preceq_E)$ forms a poset which is isomorphic to the left weak Bruhat interval $([\sfread(T_E),\sfread(T_E')]_L, \preceq_L)$ (\cref{thm: poset isomorphism}).
Specifically, we prove that the map
$$
f: (E, \preceq_E) \ra ([\sfread(T_E),\sfread(T_E')]_L, \preceq_L), 
\quad T \mapsto \sfread(T) 
$$ 
is a poset isomorphism.
The structure of the posets $(E, \preceq_E)$ for $E \in \calE_{\lambda;m}$ enables us to show that 
\[
\ch([\bfG_{\lambda;m}]) = \GSm{\lambda}{m} 
\quad 
\text{for any $\lambda \vdash m$ and $1 \le m \le n$}
\] 
in \cref{Prop: characteristic image of G}.
Since $E$ is a basis for $\bfG_E$ and $[\sfread(T_E),\sfread(T_E')]_L$ is a basis for $\sfB(\sfread(T_E),\sfread(T_E'))$, 
it is natural to ask if the map $f: (E, \preceq_E) \ra [\sfread(T_E),\sfread(T_E')]_L$ can be lifted to an $H_{m}(0)$-module isomorphism 
$$
\widetilde{f}: \bfG_E \ra \sfB(\sfread(T_E),\sfread(T_E')),
\quad T \mapsto \sfread(T) \quad \text{for $T \in E$}.
$$
We give an affirmative answer for this question in \cref{thm: WBIM for G_E} by proving that $\widetilde{f}: \bfG_E \ra \sfB(\sfread(T_E),\sfread(T_E'))$ is an $H_m(0)$-module isomorphism.

\cref{Sec: proj cover} is devoted to finding a projective cover of $\bfG_E$. 
For a generalized composition $\bal$ of $m$, let $\bfP_{\bal}$ be the projective module whose underlying space is the $\C$-span of the set $\SRT(\bal)$ of standard ribbon tableaux of shape $\bal$.
For the precise definition of $\bfP_{\bal}$, see \cref{subsec: Proj module}.
It is well known that a surjective $H_m(0)$-module homomorphism $\psi : \bfP_{\bal} \ra \bfG_E$ is a projective cover of $\bfG_E$ if and only if $\ker(\psi) \subseteq \rad(\bfP_{\bal})$ (for instance, see \cite[Proposition 3.6]{95ARS}).
We provide a sufficient condition for $\scrT \in \SRT(\bal)$ to be contained in $\rad(\bfP_{\bal})$ (\cref{Lem: calT in Rad}).
Then, by considering the specific generalized composition $\bal_E$ defined in \cref{eq: generalized composition}, we construct a surjective $H_m(0)$-module homomorphism $\eta : \bfP_{\bal_E} \ra \bfG_E$ (\cref{lem: surj linear map}).
Finally, we prove that the set $\SRT(\bal_E) \cap \ker(\eta)$ is a basis for $\ker(\eta)$ and that every element in this basis satisfies the above sufficient condition.
Hence, we obtain $\ker(\eta) \subseteq \rad(\bfP_{\bal_E})$ as desired (\cref{thm: proj cover}).

In the last section, we discuss further avenues to pursue.

\section{Preliminaries}\label{Sec: prem}

Given any integers $m$ and $n$, define
$$
[m,n] := \begin{cases}
\{k \in \Z \mid m \le k \le n\} & \text{if $m \le n$},\\
\emptyset  & \text{otherwise.}
\end{cases}
$$ 
Throughout this section, we assume that $n$ is a nonnegative integer.

\subsection{Compositions}\label{subsec: comp and diag}

A \emph{composition} $\alpha$ of $n$, denoted by $\alpha \models n$, is a finite ordered list of positive integers $(\alpha_1, \alpha_2, \ldots, \alpha_k)$ satisfying $\sum_{i=1}^k \alpha_i = n$.
We call $k =: \ell(\alpha)$ the \emph{length} of $\alpha$ and $n =:|\alpha|$ the \emph{size} of $\alpha$. For convenience we define the empty composition $\emptyset$ to be the unique composition of size and length $0$.
A \emph{generalized composition} $\bal$ of $n$ is a formal expression $\alpha^{(1)} \RS \alpha^{(2)} \RS \cdots \RS  \alpha^{(k)}$, where $\alpha^{(i)} \models n_i$ for positive integers $n_i$'s with $n_1 + n_2 + \cdots + n_k = n$.
If $\lambda = (\lambda_1,\lambda_2,\ldots, \lambda_{\ell(\alpha)}) \models n$ satisfies that $\lambda_1 \ge \lambda_2 \ge \cdots \ge \lambda_{\ell(\alpha)}$, then we say that $\lambda$ is a \emph{partition} of $n$ and denote it by $\lambda \vdash n$.

Given $\alpha = (\alpha_1, \alpha_2, \ldots,\alpha_{\ell(\alpha)}) \models n$ and $I = \{i_1 < i_2 < \cdots < i_l\} \subset [1, n-1]$, 
let 
\begin{align*}
&\set(\alpha) := \{\alpha_1,\alpha_1+\alpha_2,\ldots, \alpha_1 + \alpha_2 + \cdots + \alpha_{\ell(\alpha)-1}\}, \\
&\comp(I) := (i_1,i_2 - i_1,\ldots,n-i_l).
\end{align*}
The set of compositions of $n$ is in bijection with the set of subsets of $[1, n-1]$ under the correspondence $\alpha \mapsto \set(\alpha)$ (or $I \mapsto \comp(I)$).
Let $\alpha^\rmc$ be the unique composition satisfying that $\set(\alpha^c) = [1, n-1] \setminus \set(\alpha)$.
For a generalized composition 
$\bal = \alpha^{(1)} \RS \alpha^{(2)} \RS \cdots \RS  \alpha^{(k)}$, let $\bal^\rmc := (\alpha^{(1)})^\rmc \RS (\alpha^{(2)})^\rmc \RS \cdots \RS (\alpha^{(k)})^\rmc$.

For compositions $\alpha = (\alpha_{1}, \alpha_{2}, \ldots, \alpha_{\ell(\alpha)})$ and $\beta= (\beta_{1}, \beta_{2}, \ldots, \beta_{\ell(\beta)})$,
let $\alpha \cdot \beta$ be the \emph{concatenation}
and $\alpha \odot \beta$ the \emph{near concatenation} of $\alpha$ and $\beta$.
In other words,
$ \alpha \cdot \beta = (\alpha_1, \alpha_2, \ldots, \alpha_{\ell(\alpha)}, \beta_1, \beta_2, \ldots, \beta_{\ell(\beta)})$ and 
$\alpha \odot \beta = (\alpha_1,\alpha_2,\ldots, \alpha_{k-1},\alpha_{\ell(\alpha)} + 
\beta_1,\beta_2, \ldots, \beta_{\ell(\beta)})$.
For a generalized composition 
$\bal = \alpha^{(1)} \RS \alpha^{(2)} \RS \cdots \RS  \alpha^{(k)}$,
we define
\begin{align*}
[\bal] & := \{
\alpha^{(1)} \  \square \  \alpha^{(2)} \ \square \ \cdots \  \square  \ \alpha^{(k)}
\mid
\square = \text{$\cdot$ or $\odot$}
\}.
\end{align*}
We also define
\begin{align}\label{eq: bal bullet and odot}
\bal_\bullet := \alpha^{(1)} \cdot \alpha^{(2)} \cdot \  \cdots \  \cdot  \alpha^{(k)}, \quad
\bal_\odot := \alpha^{(1)} \odot \alpha^{(2)} \odot \cdots \odot  \alpha^{(k)},
\end{align}
and
\begin{align}\label{eq: bal odot beta}
\bal \cdot \beta := \alpha^{(1)} \RS \alpha^{(2)} \RS \cdots \RS (\alpha^{(k)} \cdot \beta).
\end{align}

\subsection{Diagrams}

For $\alpha = (\alpha_1, \alpha_2, \ldots, \alpha_{\ell(\alpha)}) \models n$, we define the \emph{ribbon diagram $\trd(\alpha)$ of $\alpha$} by the connected skew diagram without $2\times 2$ boxes, such that the $i$th column from the left has $\alpha_i$ boxes. 
For a generalized composition $\bal = \alpha^{(1)} \RS \alpha^{(2)} \RS \cdots \RS \alpha^{(k)}$ of $n$, we define the {\em generalized ribbon diagram $\trd(\bal)$ of $\bal$} to be the skew diagram whose connected components are $\trd(\alpha^{(1)}), \trd(\alpha^{(2)}), \ldots, \trd(\alpha^{(k)})$ such that $\trd(\alpha^{(i+1)})$ is strictly to the northeast of $\trd(\alpha^{(i)})$ for $i = 1, 2, \ldots, k-1$.
For example, if $\bal = (2,1) \RS (1,1)$, then
\[
\trd(\bal) = \begin{array}{c}
\begin{ytableau}
\none & \none & ~ & ~ \\
~ & ~ \\
~
\end{ytableau}
\end{array}.
\]
A filling of $\trd(\bal)$ is a function $\scrT: \trd(\bal) \ra \Z_{>0}$.

For $\lambda = (\lambda_1, \lambda_2, \ldots, \lambda_{\ell(\lambda)}) \vdash n$, 
we define the \emph{Young diagram} $\tyd(\lambda)$ of $\lambda$ by a left-justified array of $n$ boxes where the $i$th row from the top has $\lambda_i$ boxes for $1 \le i \le \ell(\lambda)$.
We say that a box in $\tyd(\lambda)$ is \emph{in the $i$th row} if it is in the $i$th row from the top and \emph{in the $j$th column} if it is in the $j$th column from the left.
We denote by $(i,j)$ the box in the $i$th row and $j$th column.
For any box $(i,j)$, let $\row((i,j)) = i$ and $\col((i,j)) = j$.
Denoting $(i,j) \in \tyd(\lambda)$ means that $1 \le i \le \ell(\lambda)$ and $1 \le j \le \lambda_i$.
We also say that a lattice point on $\tyd(\lambda)$ is \emph{in the $i$th row} if it is in the $(i+1)$st horizontal line from the top and \emph{in the $j$th column} if it is in the $(j+1)$st vertical line from the left.
We denote by $\grco{i}{j}$ the lattice point in the $i$th row and $j$th column.
For example, if $\lambda = (3,2,2)$, then 
\begin{displaymath}
\def \hp {0.65}
\def \vp {0.65}
\begin{tikzpicture}
\node[] at (-1, 0.95) {$\tyd(\lambda) =$};
\draw (0*\hp,0*\vp) -- (2*\hp,0*\vp) -- (2*\hp,2*\vp) -- (3*\hp,2*\vp) -- (3*\hp,3*\vp) -- (0*\hp,3*\vp) -- (0*\hp,0*\vp);
\draw (0*\hp,2*\vp) -- (2*\hp,2*\vp);
\draw (0*\hp,1*\vp) -- (2*\hp,1*\vp);
\draw (1*\hp,0*\vp) -- (1*\hp,3*\vp);
\draw (2*\hp,2*\vp) -- (2*\hp,3*\vp);

\draw[fill = red!50] (2*\hp,2*\vp) -- (2*\hp,3*\vp) -- (3*\hp,3*\vp) -- (3*\hp,2*\vp) -- (2*\hp,2*\vp);
\node at (2.5*\hp,2.5*\vp) {\tiny $(1,3)$};

\node[color = blue] at (0,0) {$\bullet$}; 
\node at (0.5*\hp,-0.4*\vp) {\tiny $\grco{3}{0}$};
\node[] at (2.2, 0.75) {,};
\end{tikzpicture}
\end{displaymath}
the box $(1,3)$ is the box filled with red, and 
the lattice point $\grco{3}{0}$ is the point marked by the blue dot.
A filling of $\tyd(\lambda)$ is a function $T: \tyd(\lambda) \ra \Z_{>0}$. 
Throughout this paper, we assume that
\[
\begin{array}{ll}
T((i,j)) = \infty & \text{if $(i,j) \in (\Z_{>0} \times \Z_{>0}) \setminus \tyd(\lambda)$}
\quad \text{and} \\[0.5ex]
T((i,j)) = -\infty & \text{if $(i,j) \in (\Z \times \Z) \setminus (\Z_{>0} \times \Z_{>0})$.}
\end{array}
\]
For any filling $T$ of $\tyd(\lambda)$,
let 
$$
\max (T) := \max \{T((i,j)) \mid (i,j) \in \tyd(\lambda)\}.
$$

\subsection{The $0$-Hecke algebra and the quasisymmetric characteristic}\label{subsec: 0-Hecke alg}
To begin with, we recall that the symmetric group $\SG_n$ is generated by simple transpositions $s_i := (i,i+1)$ with $1 \le i \le n-1$.
An expression for $\sigma \in \SG_n$ of the form $s_{i_1} s_{i_2} \cdots s_{i_p}$ that uses the minimal number of simple transpositions is called a \emph{reduced expression} for $\sigma$. 
The number of simple transpositions in any reduced expression for $\sigma$, denoted by $\ell(\sigma)$, is called the \emph{length} of $\sigma$.
Let $w_0$ be the longest element in $\SG_n$, and $w_0(\alpha)$ the longest element in the parabolic subgroup of $\SG_n$ generated by $\{s_i \mid i \in \set(\alpha) \}$ for $\alpha \models n$.

The $0$-Hecke algebra $H_n(0)$ is the $\C$-algebra generated by $\pi_1, \pi_2, \ldots,\pi_{n-1}$ subject to the following relations:
\begin{align*}
\pi_i^2 &= \pi_i \quad \text{for $1\le i \le n-1$},\\
\pi_i \pi_{i+1} \pi_i &= \pi_{i+1} \pi_i \pi_{i+1}  \quad \text{for $1\le i \le n-2$},\\
\pi_i \pi_j &=\pi_j \pi_i \quad \text{if $|i-j| \ge 2$}.
\end{align*}
For each $1 \le i \le n-1$, let $\opi_i := \pi_i -1$.

Consider any reduced expression $s_{i_1} s_{i_2} \cdots s_{i_p}$ for a permutation $\sigma \in \SG_n$. 
We define the elements $\pi_{\sigma}$ and $\opi_{\sigma}$ of $H_n(0)$ by
\[
\pi_{\sigma} := \pi_{i_1} \pi_{i_2} \cdots \pi_{i_p} \quad \text{and} \quad \opi_{\sigma} := \opi_{i_1} \opi_{i_2} \cdots \opi_{i_p}.
\]
It is well known that these elements are independent of the choice of reduced expressions, and both $\{\pi_\sigma \mid \sigma \in \SG_n\}$ and $\{\opi_\sigma \mid \sigma \in \SG_n\}$ are bases for $H_n(0)$.

In \cite{79Norton}, Norton classified all irreducible modules and projective indecomposable modules of the $0$-Hecke algebras. 
It was shown that there are $2^{n-1}$ distinct irreducible $H_n(0)$-modules and $2^{n-1}$ distinct projective indecomposable $H_n(0)$-modules, which are naturally parametrized by compositions of $n$.
For each $\alpha \models n$,
the irreducible module $\bfF_\alpha$ corresponding to $\alpha$ is the $1$-dimensional $H_n(0)$-module spanned by a vector $v_\alpha$ whose $H_n(0)$-action is given by
\[
\pi_i \cdot v_\alpha = \begin{cases}
0 & i \in \set(\alpha),\\
v_\alpha & i \notin \set(\alpha),
\end{cases}
\qquad (1\le i \le n-1).
\]
Also, the projective indecomposable module corresponding to $\alpha$ is the submodule $\calP_\alpha := H_n(0) \pi_{w_0(\alpha^\rmc)} \opi_{w_0(\alpha)}$ of the regular representation of $H_n(0)$.
It is known that $\calP_\alpha / \rad \; \calP_\alpha \cong \bfF_\alpha$ for all $\alpha \models n$, where $\rad \;  \calP_\alpha$ is the radical of $\calP_\alpha$.
For instance, see \cite{16Huang, 79Norton}.

Let $\calR(H_n(0))$ denote the $\Z$-span of the isomorphism classes of finite dimensional representations of $H_n(0)$. The isomorphism class corresponding to an $H_n(0)$-module $M$ will be denoted by $[M]$. The \emph{Grothendieck group} $\calG_0(H_n(0))$ is the quotient of $\calR(H_n(0))$ modulo the relations $[M] = [M'] + [M'']$ whenever there exists a short exact sequence $0 \ra M' \ra M \ra M'' \ra 0$. 
The set $\{[\bfF_\alpha] \mid \alpha \models n \}$ is a free $\Z$-basis for $\calG_0(H_n(0))$. 
Let
\[
\calG := \bigoplus_{n \ge 0} \calG_0(H_n(0)).
\]
be the ring equipped with the induction product.

Let us review the connection between $\calG$ and quasisymmetric functions.
Quasisymmetric functions are power series of bounded degree in variables $x_{1},x_{2},x_{3},\ldots$  with coefficients in $\Z$, which are shift invariant in the following sense: The coefficient of the monomial $x_{1}^{\alpha _{1}}x_{2}^{\alpha _{2}}\cdots x_{k}^{\alpha _{k}}$ is equal to the coefficient of the monomial $x_{i_{1}}^{\alpha _{1}}x_{i_{2}}^{\alpha _{2}}\cdots x_{i_{k}}^{\alpha _{k}}$ for any strictly increasing sequence of positive integers $i_{1}<i_{2}<\cdots <i_{k}$ indexing the variables and any positive integer sequence $(\alpha _{1},\alpha _{2},\ldots ,\alpha _{k})$ of exponents.

Given a composition $\alpha$, the \emph{fundamental quasisymmetric function} $F_\alpha$ is defined by $F_\emptyset = 1$ and
\begin{align*}
F_\alpha = \sum_{\substack{1 \le i_1 \le i_2 \le \cdots \le i_k \\ i_j < i_{j+1} \text{ if } j \in \set(\alpha)}} x_{i_1} x_{i_2} \cdots x_{i_k}.
\end{align*}
It is well known that $\{F_\alpha \mid \text{$\alpha$ is a composition}\}$ is a basis for the ring $\QSym$ of quasisymmetric functions.
For instance, see \cite{83Gessel} and \cite[Proposition 7.19.1]{99Stanley}.
In~\cite{96DKLT}, Duchamp, Krob, Leclerc, and Thibon showed that 
\begin{align*}
\ch : \calG \ra \QSym, \quad [\bfF_{\alpha}] \mapsto F_{\alpha},
\end{align*}
called \emph{quasisymmetric characteristic}, is a ring isomorphism.

\subsection{Projective modules of the $0$-Hecke algebra}
\label{subsec: Proj module}

In~\cite{16Huang}, Huang provided a combinatorial description of projective indecomposable modules and their induction products by using standard ribbon tableaux of generalized composition shape.
Here, we introduce the description briefly.

\begin{definition}\label{def: SRT}
For a generalized composition $\bal$ of $n$, a \emph{standard ribbon tableau} (SRT) of shape $\bal$ is a filling of $\trd(\bal)$ with $\{1,2,\ldots,n\}$ such that 
the entries are all distinct,
the entries in each row are increasing from left to right, and 
the entries in each column are increasing from top to bottom.
\end{definition}

We denote by $\SRT(\bal)$ the set of all $\SRT$s of shape $\bal$. 
Define an $H_n(0)$-action on the $\C$-span of $\SRT(\bal)$ by 
\begin{align}\label{eq: action on bfP}
\begin{aligned}
\pi_i \cdot \scrT = \begin{cases}
\scrT & \text{if $i$ appears strictly above $i+1$ in $\scrT$},\\
0 & \text{if $i$ and $i+1$ are in the same row of $\scrT$},\\
s_i \cdot \scrT & \text{if $i$ appears strictly below $i+1$ in $\scrT$}
\end{cases}
\end{aligned}
\end{align}
for $1\le i \le n-1$ and $\scrT \in \SRT(\bal)$.
Here, $s_i \cdot \scrT$ is obtained from $\scrT$ by swapping $i$ and $i+1$. 
Let $\bfP_\bal$ be the resulting module.

\begin{theorem}{\rm (\cite[Theorem 3.3]{16Huang})}
\label{thm: bfP isom to calP}
The following hold.
\begin{enumerate}[label = {\rm (\arabic*)}]
\item For any $\alpha \models n$, $\bfP_\alpha \cong \calP_\alpha$ as $H_n(0)$-modules.
\item Let $\bal = \alpha^{(1)} \RS \alpha^{(2)} \RS \cdots \RS \alpha^{(k)}$ be a generalized composition of $n$ such that $\alpha^{(i)} \models n_i$ for each $1 \le i \le k$.
Then,
\[
\bfP_\bal \cong (\bfP_{\alpha^{(1)}} \otimes \bfP_{\alpha^{(2)}} \otimes \cdots \otimes \bfP_{\alpha^{(k)}})\uparrow^{H_n(0)}_{H_{n_1}(0) \otimes H_{n_2}(0) \otimes \cdots \otimes H_{n_k}(0)}
\cong \bigoplus_{\beta \in [\bal]} \bfP_\beta 
\]
as $H_n(0)$-modules.
\end{enumerate}
\end{theorem}

For a generalized composition $\bal$, let $\scrT_\bal \in \SRT(\bal)$ be the standard ribbon tableau obtained by filling $\trd(\bal)$ with entries $1,2,\ldots, n$ from top to bottom starting from the left.
Then $\bfP_\bal$ is cyclically generated by $\scrT_\bal$.

\subsection{Weak Bruhat interval modules of the $0$-Hecke algebra}\label{subsec: WBIM}

Given $\sigma \in \SG_n$ and $i \in [1,n-1]$, $i$ is called a \emph{left descent of $\sigma$} if $\ell(s_i \sigma) < \ell(\sigma)$.
Let $\Des_L(\sigma)$ be the set of all left descents of $\sigma$.
The \emph{left weak Bruhat order $\preceq_L$} on $\SG_n$ is the partial order on $\SG_n$ whose covering relation $\preceq_L^c$ is defined as follows: $\sigma \preceq_L^c s_i \sigma$ if and only if $i \notin \Des_L(\sigma)$.
Given $\sigma, \rho \in \SG_n$, the closed interval $\{\gamma \in \SG_n \mid \sigma \preceq_L \gamma \preceq_L \rho\}$ is called the \emph{left weak Bruhat interval from $\sigma$ to $\rho$} and denoted by $[\sigma,\rho]_L$.

\begin{definition}{\rm (\cite{22JKLO})}\label{def: WBIM}
Let $\sigma, \rho \in \SG_n$.
The \emph{weak Bruhat interval module associated to $[\sigma,\rho]_L$}, denoted by $\sfB(\sigma,\rho)$, is the $H_n(0)$-module with the underlying space $\C[\sigma,\rho]_L$ and with the $H_n(0)$-action defined by
\begin{align*}
\pi_i \cdot \gamma := \begin{cases}
\gamma & \text{if $i \in \Des_L(\gamma)$}, \\
0 & \text{if $i \notin \Des_L(\gamma)$ and $s_i\gamma \notin [\sigma,\rho]_L$,} \\
s_i \gamma & \text{if $i \notin \Des_L(\gamma)$ and $s_i\gamma \in [\sigma,\rho]_L$}.
\end{cases} 
\end{align*}
\end{definition}

In~\cite{22JKLO}, it was shown that the family of weak Bruhat interval modules contains various $H_n(0)$-modules up to isomorphism. 
In particular, the family contains all projective indecomposable modules and their induction products.
Precisely, given $\alpha \models n$ and generalized composition $\bal = \alpha^{(1)} \RS \alpha^{(2)} \RS \cdots \RS  \alpha^{(k)}$ of $n$,
\begin{align*}
\bfP_\alpha \cong \sfB(w_0(\alpha^\rmc), w_0 w_0(\alpha))
\quad \text{and} \quad
\bfP_\bal \cong \sfB(w_0(\bal_\bullet^\rmc), w_0 w_0(\bal_\odot)).
\end{align*}
For the definitions of $\bal_\bullet$ and $\bal_\odot$, see \cref{eq: bal bullet and odot}.
For more information on weak Bruhat interval modules, see~\cite{22JKLO}.

\begin{remark}
Duchamp, Hivert, and Thibon \cite{02DHT} constructed an $H_n(0)$-module arising from a poset on $[n]$.
One may ask if there is a relationship between the $H_n(0)$-modules arising from posets and weak Bruhat interval modules. 
Very recently, Choi, Kim, and Oh \cite{23CKO} addressed this question, proving that every weak Bruhat interval module can be recovered as an $H_n(0)$-module arising from a regular poset.
\end{remark}

For later use, we state the following useful property that can be easily proved:
For any $\sigma, \rho \in \SG_n$ and $\rho' \in [\sigma,\rho]_L$, the linear map 
\begin{align}\label{eq: definition of pr}
\pr : \sfB(\sigma,\rho) \ra \sfB(\sigma,\rho'), 
\quad
\gamma \mapsto \begin{cases}
\gamma & \text{if $\gamma \in [\sigma,\rho']_L$,} \\
0 & \text{otherwise}
\end{cases}
\end{align}
is a surjective $H_n(0)$-module homomorphism.

\subsection{Genomic Schur functions}
\label{subsec: Genomic Schur}

In \cite{17PY2}, Pechenik and Yong introduced the genomic Schur function as a generating function for genomic tableaux to develop the combinatorial theory of genomic tableaux.
Recently, Pechenik \cite{20Pechenik} provided another way to define genomic Schur functions by using increasing gapless tableaux.

\begin{definition}
Given $\lambda \vdash n$, an \emph{increasing gapless tableau} of shape $\lambda$ is a filling of $\tyd(\lambda)$ such that
\begin{enumerate}[label = {\rm (\arabic*)}]
\item the entries in each row strictly increase from left to right,
\item the entries in each column strictly increase from top to bottom, and
\item the set $T^{-1}(k)$ is nonempty for all $1 \le k \le \max (T)$.
\end{enumerate}
\end{definition}
Let $\IGLT(\lambda)$ be the set of all increasing gapless tableaux of shape $\lambda$.
Given $T \in \IGLT(\lambda)$ and $i \in [1, \max (T)]$, let $\Top_i(T)$ (resp. $\Bot_i(T)$) be the unique box $B \in T^{-1}(i)$ such that $\row(B)$ is minimal (resp. maximal) among $\row(B')$'s for $B' \in T^{-1}(i)$.
In other words, $\mathsf{Top}_i(T)$ (resp. $\mathsf{Bot}_i(T)$) is the highest (resp. lowest) box in $T$ having entry $i$.
Also, let
\begin{align}\label{eq: rb cb rt ct}
(r_\sfb^{(i)}(T), c_\sfb^{(i)}(T)) := \Bot_i(T)
\quad \text{and} \quad 
(r_\sft^{(i)}(T), c_\sft^{(i)}(T)) := \Top_i(T).
\end{align}
If $T$ is clear in the context, we simply write $r_\sfb^{(i)}, c_\sfb^{(i)}, r_\sft^{(i)}$, and $c_\sft^{(i)}$ instead of $r_\sfb^{(i)}(T)$, $c_\sfb^{(i)}(T)$, $r_\sft^{(i)}(T)$, and $c_\sft^{(i)}(T)$, respectively.
We call $i$ a \emph{descent} of $T$ if $r_\sft^{(i)} < r_\sfb^{(i+1)}$, or equivalently, there is some instance of $i$ strictly above some instance of $i+1$ in $T$.
Denote by $\Des(T)$ the set of all descents of $T$ and set $\comp(T) := \comp(\Des(T))$.
Given $1 \le m \le n$, we define  
$$
\IGLTm{\lambda}{m} := \{T \in \IGLT(\lambda) \mid \max (T) = m\}.
$$

\begin{definition}{\rm (\cite{20Pechenik, 17PY2})}
For $\lambda \vdash n$, the \emph{genomic Schur function} $\GS{\lambda}$ is defined by
\[
\GS{\lambda} := \sum_{1 \le m \le n} \left( \sum_{T \in \IGLT(\lambda)_m} F_{\comp(T)} \right).
\]
\end{definition}

For $1 \le m \le n$, let $\GSm{\lambda}{m} := \sum_{T \in \IGLTm{\lambda}{m}} F_{\comp(T)}$.
From the definition, it immediately follows that $\GSm{\lambda}{m}$ is the $m$th degree homogeneous component of $\GS{\lambda}$.

\begin{example}
Note that
\[
\IGLT( (2,2) ) = \left\{
\begin{array}{l}
\begin{ytableau}
1 & 2 \\
2 & 3
\end{ytableau},~
\begin{ytableau}
1 & 2 \\
3 & 4
\end{ytableau},~
\begin{ytableau}
1 & 3 \\
2 & 4
\end{ytableau}
\end{array}
\right\}.
\]
One can see that
\[
\comp\left( 
\begin{array}{l}
\begin{ytableau}
1 & 2 \\
2 & 3
\end{ytableau}
\end{array}
\right)
= (1,1,1),
\quad
\comp\left( 
\begin{array}{l}
\begin{ytableau}
1 & 2 \\
3 & 4
\end{ytableau}
\end{array}
\right)
= (2,2),
\quad
\comp\left( 
\begin{array}{l}
\begin{ytableau}
1 & 3 \\
2 & 4
\end{ytableau}
\end{array}
\right)
= (1,2,1).
\]
Thus, 
$$
\GSm{(2,2)}{3} = F_{(1,1,1)}, \quad \GSm{(2,2)}{4} = F_{(2,2)} + F_{(1,2,1)}, \quad \text{and} \quad \GS{(2,2)} = F_{(1,1,1)} + F_{(2,2)} + F_{(1,2,1)}.
$$
\end{example}

Hereafter, we assume that $n$ is a positive integer, $m$ is a positive integer less than or equal to $n$, and $\lambda$ is a partition of $n$, unless otherwise stated.

\section{$0$-Hecke modules arising from increasing gapless tableaux}\label{Sec: 0-Hecke action}

In this section, we introduce an $H_m(0)$-module $\bfG_{\lambda;m}$ by defining an $H_n(0)$-action on the $\C$-span of $\IGLTm{\lambda}{m}$. 
Then, we decompose $\bfG_{\lambda;m}$ into a direct sum of $H_m(0)$-submodules which will turn out to be weak Bruhat interval modules in \cref{Sec: weak bruhat interval module}.

\subsection{An $H_{m}(0)$-action on $\C\IGLTm{\lambda}{m}$}\label{Subsec: 0-Hecke modules for genomic Shcur functions}

We start by introducing the necessary definitions.

\begin{definition}\label{def: attacking descent}
Given $T \in \IGLT(\lambda)$ and $1 \le i \le \max (T) -1$, we say that $i$ is an \emph{attacking descent} if $i \in \Des(T)$, and either
\begin{enumerate}[label = \textrm{(\alph*)}]
\item
there exists $(j,k) \in \tyd(\lambda)$ such that $T( (j,k) ) = i$ and $T( (j +1, k) ) = i + 1$, or

\item 
there exists a box $B \in T^{-1}(i+1)$ placed weakly above $\Bot_i(T)$.
\end{enumerate}
\end{definition}
We notice that if $i$ is a non-attacking descent of $T$, then all $(i+1)$'s lie strictly below and strictly left of $\Bot_i(T)$.

Take any $1 \le m \le n$.
For each $1 \le i \le m-1$, define a linear operator $\bfpi_i: \C \, \IGLTm{\lambda}{m} \ra \C \, \IGLTm{\lambda}{m}$ by letting
\begin{align*}
\bfpi_i (T) := 
\begin{cases}
T & \text{if $i$ is not a descent of $T$,}\\
0 & \text{if $i$ is an attacking descent of $T$,}\\ 
s_i \cdot T & \text{if $i$ is a non-attacking descent of $T$}
\end{cases}
\end{align*}
for $T \in \IGLTm{\lambda}{m}$ and extending it by linearity.
Here, $s_i \cdot T$ is the tableau obtained from $T$ by replacing $i$ and $i+1$ with $i+1$ and $i$, respectively.

The following is the main theorem of this subsection.

\begin{theorem}\label{thm: 0-Hecke action on IGLT}
For any $1 \le m \le n$,
the operators $\bfpi_1, \bfpi_2,\ldots, \bfpi_{m-1}$ satisfy the same relations as the generators $\pi_1, \pi_2, \ldots, \pi_{m-1}$ for $H_m(0)$.
In other words, $\bfpi_1, \bfpi_2,\ldots, \bfpi_{m-1}$ define an $H_m(0)$-action on $\C \, \IGLTm{\lambda}{m}$.
\end{theorem}

In order to prove this theorem, let us establish some necessary lemmas.

\begin{lemma}\label{lem: idempotent relations of bfpi}
For $1 \le i \le m-1$, $\bfpi_i^2 = \bfpi_i$. 
\end{lemma}

\begin{proof}
Let $T \in \IGLTm{\lambda}{m}$ and $1 \le i \le m-1$.
If $\bfpi_i(T) = T$ or $\bfpi_i(T) = 0$, then it is obvious that $\bfpi_i^2(T) = \bfpi_i(T)$.
If $\bfpi_i(T) = s_i \cdot T$, then, every $i+1$ is strictly below each $i$ in $T$.
This implies that $i \notin \Des(s_i \cdot T)$, thus $\bfpi_i^2(T) = \bfpi_i(T)$.
\end{proof}

\begin{lemma}\label{lem: commutation relations of bfpi}
For $1 \le i,j \le m-1$ with $|i-j|>1$, $\bfpi_i \bfpi_j  = \bfpi_j \bfpi_i$.
\end{lemma}

\begin{proof}
Let $T \in \IGLTm{\lambda}{m}$ and $1 \le i,j \le m-1$ with $|i-j|>1$.
Suppose that $\bfpi_i (T) = T$ or $\bfpi_i (T) = 0$.
If $\bfpi_j(T) = T$ or $\bfpi_j(T) = 0$, then it is obvious that $\bfpi_i \bfpi_j(T) = \bfpi_j \bfpi_i(T)$.
If $\bfpi_j(T) = s_j \cdot T$, then 
$$
T^{-1}(i) = (s_j \cdot T)^{-1}(i) 
\quad \text{and} \quad 
T^{-1}(i+1) = (s_j \cdot T)^{-1}(i+1).
$$
Combining this with the assumption that $\bfpi_i(T) = T$ or $\bfpi_i (T) = 0$, we have that $\bfpi_i(\bfpi_j(T)) = \bfpi_j(\bfpi_i(T))$.

For the remaining case, suppose that $\bfpi_i (T) = s_i \cdot T$ and $\bfpi_j (T) = s_j \cdot T$.
Since $|i - j| > 1$,
\[
T^{-1}(i) = (s_j \cdot T)^{-1}(i)
\quad \text{and} \quad 
T^{-1}(i+1) = (s_j \cdot T)^{-1}(i+1)
\]
and
\[
T^{-1}(j) = (s_i \cdot T)^{-1}(j)
\quad \text{and} \quad 
T^{-1}(j+1) = (s_i \cdot T)^{-1}(j+1).
\]
It follows that $\bfpi_i \cdot (s_j \cdot T) = s_i \cdot (s_j \cdot T)$ and $\bfpi_j \cdot (s_i \cdot T) = s_j \cdot (s_i \cdot T)$.
It is obvious that $s_i \cdot (s_j \cdot T) = s_j \cdot (s_i \cdot T)$, thus the assertion follows.
\end{proof}

\begin{lemma}\label{lem: braid relations of bfpi}
For $1 \le i \le m-2$, $\bfpi_i \bfpi_{i+1} \bfpi_{i} = \bfpi_{i+1} \bfpi_i \bfpi_{i+1}$.
\end{lemma}

\begin{proof}
Given $T \in \IGLTm{\lambda}{m}$ and $1 \le i \le m-2$, we have three cases.

\smallskip

\noindent {\it Case 1: $\bfpi_i (T) = T$.}
If $\bfpi_{i+1}(T) = T$ or 
$\bfpi_{i+1}(T) = 0$, the proof is straightforward.
Let us assume that $\bfpi_{i+1} \cdot T =s_{i+1}\cdot T$. 
Note that
\begin{align}\label{eq: pi_i braid}
\bfpi_i \bfpi_{i+1} \bfpi_{i}(T) = \bfpi_i (s_{i+1} \cdot T) \quad \text{and} \quad \bfpi_{i+1} \bfpi_i \bfpi_{i+1}(T) = \bfpi_{i+1} \bfpi_i (s_{i+1} \cdot T).
\end{align}
If $\bfpi_i (s_{i+1} \cdot T) = s_{i+1} \cdot T$ or $\bfpi_i (s_{i+1} \cdot T) = 0$, one can easily see that the right hand sides of the two equations in \cref{eq: pi_i braid} are the same.
For the remaining part of Case 1, suppose that $\bfpi_i (s_{i+1} \cdot T) = s_i \cdot (s_{i+1} \cdot T)$.
Since
$$
T^{-1}(i) = (s_{i}\cdot(s_{i+1}\cdot T))^{-1}(i+1)
\quad \text{and} \quad
T^{-1}(i+1) = (s_{i}\cdot(s_{i+1}\cdot T))^{-1}(i+2),
$$
the assumption $\bfpi_i (T) = T$ implies that $\bfpi_{i+1} (s_{i}\cdot(s_{i+1}\cdot T)) = s_{i}\cdot(s_{i+1}\cdot T)$.
Thus, we have that 
$$
\bfpi_i \bfpi_{i+1} \bfpi_{i}(T) = s_{i}\cdot(s_{i+1}\cdot T) = \bfpi_{i+1} \bfpi_i \bfpi_{i+1}(T).
$$
\smallskip

\noindent {\it Case 2: $\bfpi_i(T) = 0$.}
If $\bfpi_{i+1}(T) = T$ or $\bfpi_{i+1}(T) = 0$, then $\bfpi_i \bfpi_{i+1} \bfpi_i(T) = 0 =\bfpi_{i+1} \bfpi_i \bfpi_{i+1}(T)$.
Assume that $\bfpi_{i+1}(T) = s_{i+1}\cdot T$.
Let us consider the three subcases 
\[
\bfpi_{i}(s_{i+1}\cdot T)= s_{i+1}\cdot T,\quad
\bfpi_{i}(s_{i+1}\cdot T)= 0,
\quad \text{and} \quad
\bfpi_{i}(s_{i+1}\cdot T)= s_i \cdot (s_{i+1}\cdot T).
\]

In case where $\bfpi_{i}(s_{i+1}\cdot T) = s_{i+1}\cdot T$,
we have that $r_\sft^{(i)}(s_{i+1} \cdot T) \ge r_\sfb^{(i+1)}(s_{i+1} \cdot T)$.  
For the definitions of $r_\sft^{(i)}(s_{i+1} \cdot T)$ and $r_\sfb^{(i+1)}(s_{i+1} \cdot T)$, see~\cref{eq: rb cb rt ct}.
In addition, 
$r_\sft^{(i+2)}(T) > r_\sfb^{(i+1)}(T)$ since $\bfpi_{i+1}(T) = s_{i+1}\cdot T$. 
Therefore,
\begin{align*}
r_\sft^{(i)}(T)
= r_\sft^{(i)}(s_{i+1} \cdot T)
\ge r_\sfb^{(i+1)}(s_{i+1} \cdot T)
= r_\sfb^{(i+2)}(T)
> r_\sft^{(i+2)}(T)
> r_\sfb^{(i+1)}(T).
\end{align*}
But, this contradicts the assumption that $\bfpi_i(T) = 0$, thus $\bfpi_{i}(s_{i+1}\cdot T)$ cannot be $s_{i+1}\cdot T$.

In case where $\bfpi_{i}(s_{i+1}\cdot T)= 0$, we immediately have that $\bfpi_{i+1} \bfpi_i \bfpi_{i+1}(T) = 0$. 

In case where $\bfpi_{i}(s_{i+1}\cdot T)=s_{i}\cdot (s_{i+1}\cdot T)$,
since
$$
T^{-1}(i) = (s_{i}\cdot(s_{i+1}\cdot T))^{-1}(i+1) 
\quad \text{and} \quad 
T^{-1}(i+1) = (s_{i}\cdot(s_{i+1}\cdot T))^{-1}(i+2),
$$
the assumption $\bfpi_i (T) = 0$ implies that $\bfpi_{i+1} (s_{i}\cdot(s_{i+1}\cdot T)) = 0$.

\smallskip

\noindent {\it Case 3: $\bfpi_i(T) = s_i \cdot T$.} 
First, suppose that $\bfpi_{i+1}(T)=T$. 
Then
\[
\bfpi_i \bfpi_{i+1} \bfpi_{i}(T) = \bfpi_i (\bfpi_{i+1}(s_{i} \cdot T)) 
\quad \text{and} \quad 
\bfpi_{i+1} \bfpi_i \bfpi_{i+1}(T) = \bfpi_{i+1} (s_{i} \cdot T).
\]
If $\bfpi_{i+1} (s_{i} \cdot T)=s_{i}\cdot T$ or $\bfpi_{i+1} (s_{i} \cdot T)=0$, then one can easily show that $\bfpi_i (\bfpi_{i+1}(s_{i} \cdot T)) = \bfpi_{i+1} (s_{i} \cdot T)$.
In case where $\bfpi_{i+1} (s_{i} \cdot T)=s_{i+1}\cdot(s_{i}\cdot T)$, we have that
\[
\bfpi_i (\bfpi_{i+1}(s_{i} \cdot T)) = \bfpi_{i}(s_{i+1} \cdot (s_{i} \cdot T))
\quad \text{and} \quad
\bfpi_{i+1} (s_{i} \cdot T) = s_{i+1} \cdot (s_{i}\cdot T). 
\]
In addition, the assumption $\bfpi_{i+1}(T)=T$ implies that $r_\sft^{(i+1)}(T) \ge r_\sfb^{(i+2)}(T)$, equivalently,
$$
r_\sft^{(i)}(s_{i+1} \cdot (s_i \cdot T))
\ge r_\sfb^{(i+1)}(s_{i+1} \cdot (s_i \cdot T)).
$$
Therefore, $\bfpi_{i}(s_{i+1} \cdot (s_{i} \cdot T)) = s_{i+1} \cdot (s_{i}\cdot T)$.

Next, suppose that $\bfpi_{i+1}(T) = 0$. 
We claim that $\bfpi_{i}(\bfpi_{i+1}(s_{i} \cdot T))=0$. 
If $\bfpi_{i+1}(s_{i} \cdot T)=s_{i}\cdot T$, $r_\sft^{(i+1)}(s_i \cdot T) \ge 
r_\sfb^{(i+2)}(s_{i} \cdot T)$, equivalently, $r_\sft^{(i)}(T) \ge  r_\sfb^{(i+2)}(T)$.
Additionally, the assumption $\bfpi_{i}(T) = s_i \cdot T$ implies that $r_\sft^{(i+1)}(T) > r_\sfb^{(i)}(T)$.
Thus
\[
r_\sft^{(i+1)}(T) > r_\sfb^{(i)}(T) \ge r_\sft^{(i)}(T) \ge r_\sfb^{(i+2)}(T).
\]
It follows that $i+1 \notin \Des(T)$, which contradicts the assumption $\bfpi_{i+1}(T)=0$. 
Hence, $\bfpi_{i+1}(s_{i} \cdot T)$ cannot be $s_{i}\cdot T$.
If $\bfpi_{i+1}(s_{i}\cdot T)=0$, then it is clear that $\bfpi_{i}(\bfpi_{i+1}(s_{i} \cdot T))=0$. 
If $\bfpi_{i+1}(s_{i}\cdot T)=s_{i+1}\cdot (s_{i}\cdot T)$, then 
$$
T^{-1}(i + 1) =  (s_{i+1}\cdot(s_{i}\cdot T))^{-1}(i) 
\quad \text{and} \quad 
T^{-1}(i + 2) = (s_{i+1}\cdot(s_{i}\cdot T))^{-1}(i + 1).
$$
Thus, the assumption $\bfpi_{i+1} (T) = 0$ implies that $\bfpi_{i}(s_{i+1} \cdot (s_{i} \cdot T))=0$.

Finally, suppose that $\bfpi_{i+1}(T) = s_{i+1} \cdot T$. 
Then we have $r_\sft^{(i+2)}(T) > r_\sfb^{(i+1)}(T)$.
In addition, the assumption $\bfpi_i(T) = s_i \cdot T$ implies that
$r_\sft^{(i+1)}(T) > r_\sfb^{(i)}(T)$.
Thus,
$$
r_\sft^{(i+2)}(T) > r_\sfb^{(i+1)}(T) \ge r_\sft^{(i+1)}(T) > r_\sfb^{(i)}(T).
$$
Now, one can easily see that $\bfpi_i \bfpi_{i+1} \bfpi_{i}(T) = \bfpi_{i+1} \bfpi_i \bfpi_{i+1}(T)$.
\end{proof}

\cref{thm: 0-Hecke action on IGLT} now follows immediately from \cref{lem: idempotent relations of bfpi}, \cref{lem: commutation relations of bfpi}, and \cref{lem: braid relations of bfpi}.
Hereafter, for $1 \le m \le n$, we denote by $\bfG_{\lambda;m}$ the $H_m(0)$-module whose underlying space is $\C \, \IGLTm{\lambda}{m}$ and whose $H_m(0)$-action is given by \cref{thm: 0-Hecke action on IGLT}.

\begin{example}\label{eg: G 211} 
(1) When $T = \begin{array}{l}
$
\begin{ytableau}
\color{red} 1 & \large \color{red} 2 & \large \color{red} 3 & \large \color{red} 6 \\
\large \color{red} 2 & \large \color{red} 3 & \large \color{red} 5 & \large 7 \\
\large 4 & \large \color{red} 6
\end{ytableau}$
\end{array}$, we have
$$
\boldsymbol{\uppi}_3 (T) = s_3 \cdot T, \quad
\boldsymbol{\uppi}_4 (T) = T,
\quad \text{and} \quad \boldsymbol{\uppi}_i (T) = 0 
\quad \text{for $i = 1,2,5,6$}.
$$
Here, the indices in red are used to indicate the descents of the tableau.
\medskip

\noindent 
(2) Note that
\[
\IGLT( (2,1,1) ) = \left\{
\begin{array}{l}
\begin{ytableau}
1 & 2\\
3 \\
4
\end{ytableau}
\end{array},\ 
\begin{array}{l}
\begin{ytableau}
1 & 3\\
2 \\
4
\end{ytableau}
\end{array},\ 
\begin{array}{l}
\begin{ytableau}
1 & 4\\
2 \\
3
\end{ytableau}
\end{array},\ 
\begin{array}{l}
\begin{ytableau}
1 & 2\\
2 \\
3
\end{ytableau}
\end{array},\ 
\begin{array}{l}
\begin{ytableau}
1 & 3\\
2 \\
3
\end{ytableau}
\end{array}
\right\}.
\]
The descents of each tableau in $\IGLT((2,1,1))$ are given as follows:
\[
\arraycolsep=8pt
\def\arraystretch{2}
\begin{array}{c||c|c|c}
T           
&  
\scalebox{0.6}{
$
\begin{array}{l}
\begin{ytableau}
1 & 2\\
3 \\
4
\end{ytableau}
\end{array}$
}
&  
\scalebox{0.6}{
$\begin{array}{l}
\begin{ytableau}
1 & 3\\
2 \\
4
\end{ytableau}
\end{array}$
}
&  
\scalebox{0.6}{
$\begin{array}{l}
\begin{ytableau}
1 & 4\\
2 \\
3
\end{ytableau}
\end{array}$
}
\\ \hline\hline 
\Des(T) \subseteq [1,3]
& 
\{2, 3\}
&
\{1, 3\}
& 
\{1, 2\}
\end{array}
\quad \quad \quad
\begin{array}{c||c|c}
T
&  
\scalebox{0.6}{
$\begin{array}{l}
\begin{ytableau}
1 & 2\\
2 \\
3
\end{ytableau}
\end{array}$
}
&  
\scalebox{0.6}{
$\begin{array}{l}
\begin{ytableau}
1 & 3\\
2 \\
3
\end{ytableau}
\end{array}$
}     
\\ \hline \hline 
\Des(T) \subseteq [1,2]
& \{1, 2\}
& \{1, 2\}
\end{array}
\]
Therefore, $\GS{(2,1,1)} = (F_{(2,1,1)} + F_{(1,2,1)} + F_{(1,1,2)}) + 2F_{(1,1,1)}$.
The following figures illustrate the $H_m(0)$-action on $\bfG_{(2,1,1);m}$ for $m = 3,4$: 
\[
\begin{tikzpicture}
\def \hp {4}
\def \vp {2.5}


\node[] at (0 -0.5*\hp,2) {
\begin{ytableau}
1 & \color{exgreen} 2\\
\color{red} 3 \\
4
\end{ytableau}
};

\node at (0.5-0.5*\hp, 2.1) {} edge [out=40,in=320, loop] ();
\node[] at (1.4-0.5*\hp,2.1) {$\pi_1$};

\draw [->] (0-0.5*\hp, 1) -- (0-0.5*\hp,0.5); 
\node at (0.3-0.5*\hp, 0.75) {$\pi_2$};

\draw [->] (0.6-0.5*\hp, 1.15) -- (1.2-0.5*\hp, 0.75);
\node at (1.1-0.5*\hp, 1.1) {$\pi_3$};
\node at (1.5-0.5*\hp,0.6) {$0$};

\node[] at (0-0.5*\hp,-0.5) {
\begin{ytableau}
\color{red} 1 & \color{exgreen} 3 \\
2 \\
4
\end{ytableau}
};

\node at (0.5-0.5*\hp, -0.4) {} edge [out=40,in=320, loop] ();
\node[] at (1.4-0.5*\hp,-0.4) {$\pi_2$};

\draw [->] (0-0.5*\hp,-2.5 + 1.05) -- (0-0.5*\hp,-2.5 + 0.55); 
\node at (0.3-0.5*\hp, -2.5 + 0.9) {$\pi_3$};

\draw [->] (0.6-0.5*\hp, -2.5 + 1.15) -- (1.2-0.5*\hp, -2.5 + 0.75);
\node at (1.1-0.5*\hp, -2.5 + 1.1) {$\pi_1$};
\node at (1.5-0.5*\hp, -2.5 + 0.6) {$0$};

\node[] at (0-0.5*\hp, -2.5 -0.5) {
\begin{ytableau}
\color{red} 1 & 4 \\
\color{red} 2 \\
3
\end{ytableau}
};

\node at (0.5-0.5*\hp, -2.5 -0.4) {} edge [out=40,in=320, loop] ();
\node[] at (1.4-0.5*\hp, -2.5 -0.4) {$\pi_3$};

\draw [->] (0-0.5*\hp,-5 + 1.05) -- (0-0.5*\hp,-5 + 0.55); 
\node at (0.6-0.5*\hp, -5 + 0.9) {$\pi_1, \pi_2$};

\node[] at (0-0.5*\hp,-4.7) {$0$};


\node[] at (\hp, 2) {
\begin{ytableau}
\color{red} 1 & \color{red} 2\\
\color{red} 2 \\
3
\end{ytableau}
};

\draw [->] (0 + \hp, 1) -- (0 + \hp,0.5); 
\node at (0.7 + \hp, 0.75) {$\pi_1, \pi_2$};

\node[] at (\hp, 0.25) {$0$};


\node[] at (2*\hp, 2) {
\begin{ytableau}
\color{red} 1 & 3\\
\color{red} 2 \\
3
\end{ytableau}
};

\draw [->] (0 + 2*\hp, 1) -- (0 + 2*\hp,0.5); 
\node at (0.7 + 2*\hp, 0.75) {$\pi_1, \pi_2$};

\node[] at (2*\hp, 0.25) {$0$};

\node[] at (1.5*\hp, 2) {\large $\oplus$};

\node[] at (0-0.5*\hp, -5.75) {$\bfG_{(2,1,1);4}$};
\node[] at (0 + 1.5*\hp, -1) {$\bfG_{(2,1,1);3}$};
\end{tikzpicture}
\]
\end{example}

In \cref{Prop: characteristic image of G}, we will prove that 
\begin{align}\label{eq: char of bfG}
\ch([\bfG_{\lambda;m}]) = \GSm{\lambda}{m} \quad \text{for any $\lambda \vdash n$ and $1 \le m \le n$,}
\end{align}
which implies that $\sum_{1 \le m \le n} \ch([\bfG_{\lambda;m}]) = \GS{\lambda}$.

We close this subsection by providing a remark which tells us that for some $\lambda \vdash n$ and $1\le m \le n$, there is no indecomposable $H_m(0)$-module $M$ satisfying $\ch([M]) = \GSm{\lambda}{m}$.

\begin{remark}
\label{rem: counterexample for indecomposable module}
In \cite[Theorem 4.7]{02DHT}, Duchamp, Hivert, and Thibon described the \emph{Ext-quiver} of $H_m(0)$. 
For the definition of Ext-quivers, see \cite[Definition 2.7.5]{14Zimmerman}.
According to their result, for any $\alpha \models m$, we have $\Ext^1_{H_m(0)}(\bfF_\alpha, \bfF_\alpha) = 0$, equivalently, there is no indecomposable $H_m(0)$-module $M$ such that $\ch([M]) = 2 F_\alpha$.
On the other hand, in \cref{eg: G 211}, we see that $\GSm{(2,1,1)}{3} = 2 F_{(1,1,1)}$.
Thus, we conclude that there is no indecomposable $H_3(0)$-module $M$ satisfying $\ch([M]) = \GSm{(2,1,1)}{3}$.
\end{remark}

\subsection{A direct sum decomposition of $\bfG_{\lambda;m}$ into $H_m(0)$-submodules}
\label{Subsec: decomposition}

Let us start with necessary definitions and notation.
Given $T \in \IGLTm{\lambda}{m}$, let
\[
\calI(T) := \left\{ i \in [1,m] \; \middle| \; |T^{-1}(i)| > 1 \right\}.
\]
Recall that we let $(r_\sfb^{(k)}, c_\sfb^{(k)}) = \Bot_k(T)$ and $(r_\sft^{(k)}, c_\sft^{(k)}) = \Top_k(T)$ for $1 \le k \le \max(T)$.
Given $i \in \calI(T)$, let $\Gamma_i(T)$ be the lattice path from $\grco{r_\sfb^{(i)}}{c_\sfb^{(i)} - 1}$ to $\grco{r_\sft^{(i)} - 1}{c_\sft^{(i)}}$ satisfying the following two conditions:

\begin{enumerate}[label = {\rm (\roman*)}]
\item if the path passes through two boxes horizontally, then the entry at the above box is strictly smaller than $i$ and the entry at the below box is weakly greater than $i$, and 
\item if the path passes through two boxes vertically, then the entry at the left box is strictly smaller than $i$ and the entry at the right box is weakly greater than $i$.
\end{enumerate}
Pictorially,
\begin{equation*}
\begin{array}{l}
\begin{tikzpicture}
\node [] at (-1,0) {(i)};

\node [] at (0,0) {$\begin{ytableau}
a\\
b
\end{ytableau}$};

\draw[line width = 0.5mm] (-0.5,0) -- (0.5,0);

\node[] at (2,0) {$\implies a < i \le b$};
\end{tikzpicture}
\end{array}
\quad \text{and} \quad
\begin{array}{l}
\begin{tikzpicture}
\node [] at (-1.25,0) {(ii)};

\node [] at (0,0) {$\begin{ytableau}
a & b
\end{ytableau}$};

\draw[line width = 0.5mm] (0,-0.5) -- (0,0.5);

\node[] at (2,0) {$\implies a < i \le b$.};
\end{tikzpicture}
\end{array}
\end{equation*}

\begin{example}\label{Eg: lattice path}
Let 
\[
T = \begin{array}{l}\scalebox{0.75}{$
\begin{ytableau}
1  & 6  & 10 & 14 & 22 & 24 & 26 & \color{blue} 27 \\
2  & 7  & 11 & 15 & 23 & 25 & \color{brown} 29    \\
3  & 8  & 12 & 16 & 28 & \color{brown} 29         \\
4  & 9  & 13 & \color{red} 17                      \\
5  &\color{red} 17 & \color{blue} 27              \\
18 & 20                                \\
19 & \color{violet} 21                             \\
\color{violet} 21                                     
\end{ytableau}$}
\end{array}.
\]
Note that $\calI(T) = \{17,21,27,29\}$.
By following the way of defining lattice paths, we obtain the lattice paths $\Gamma_{17}(T),\Gamma_{21}(T),\Gamma_{27}(T)$, and $\Gamma_{29}(T)$ as follows:
\[
\def\pp {0.512}
\def\hp {5.5}
\scalebox{0.75}{$
\begin{tikzpicture}

\node at (0, 0) {$
\begin{ytableau}
1  & 6  & 10 & 14 & 22 & 24 & 26 & 27 \\
2  & 7  & 11 & 15 & 23 & 25 & 29    \\
3  & 8  & 12 & 16 & 28 & 29         \\
4  & 9  & 13 & \color{red} 17                      \\
5  & \color{red} 17 & 27                          \\
18 & 20                                \\
19 & 21                             \\
21                                 
\end{ytableau}
$};

\draw[line width = 0.5mm, color = red] 
(-3*\pp, -1*\pp) -- (-3*\pp, 0*\pp) -- (-1*\pp, 0*\pp) -- (-1*\pp, 1*\pp) -- (0*\pp, 1*\pp);

\node[] at (0, -2.5) {$\Gamma_{17}(T)$};

\node at (0 + \hp,0) {$
\begin{ytableau}
1  & 6  & 10 & 14 & 22 & 24 & 26 & 27 \\
2  & 7  & 11 & 15 & 23 & 25 & 29    \\
3  & 8  & 12 & 16 & 28 & 29         \\
4  & 9  & 13 & 17                      \\
5  & 17 & 27                          \\
18 & 20                                \\
19 & \color{violet} 21                             \\
\color{violet} 21                                     
\end{ytableau}
$};

\draw[line width = 0.5mm, color = red] (-4*\pp +\hp, -4*\pp) -- (-4*\pp + \hp, -3*\pp) -- (-3*\pp + \hp, -3*\pp) -- (-3*\pp + \hp, -2*\pp) -- (-2*\pp + \hp, -2*\pp);
\node[] at (0 + \hp, -2.5) {$\Gamma_{21}(T)$};

\node at (0 + 2*\hp,0) {$
\begin{ytableau}
1  & 6  & 10 & 14 & 22 & 24 & 26 & \color{blue} 27 \\
2  & 7  & 11 & 15 & 23 & 25 & 29    \\
3  & 8  & 12 & 16 & 28 & 29         \\
4  & 9  & 13 & 17                      \\
5  & 17 & \color{blue} 27                          \\
18 & 20                                \\
19 & 21                             \\
21                                 
\end{ytableau}
$};

\draw[line width = 0.5mm, color = red] 
(-2*\pp + 2*\hp, -1*\pp) -- (-2*\pp + 2*\hp, 0*\pp) -- (0*\pp + 2*\hp, 0*\pp) -- (0*\pp + 2*\hp, 2*\pp) -- (2*\pp + 2*\hp, 2*\pp) -- (2*\pp + 2*\hp, 3*\pp) -- (3*\pp + 2*\hp, 3*\pp) -- (3*\pp + 2*\hp, 4*\pp) -- (4*\pp + 2*\hp, 4*\pp);
\node[] at (2*\hp, -2.5) {$\Gamma_{27}(T)$};

\node at (0 + 3*\hp, 0) {$
\begin{ytableau}
1  & 6  & 10 & 14 & 22 & 24 & 26 & 27 \\
2  & 7  & 11 & 15 & 23 & 25 & \color{brown} 29    \\
3  & 8  & 12 & 16 & 28 & \color{brown} 29         \\
4  & 9  & 13 & 17                      \\
5  & 17 & 27                          \\
18 & 20                                \\
19 & 21                             \\
21  
\end{ytableau}
$};

\draw[line width = 0.5mm, color = red] 
(1*\pp + 3*\hp, 1*\pp) -- (1*\pp + 3*\hp, 2*\pp) -- (2*\pp + 3*\hp, 2*\pp) -- (2*\pp + 3*\hp, 3*\pp) -- (3*\pp + 3*\hp, 3*\pp);
\node[] at (3*\hp, -2.5) {$\Gamma_{29}(T)$};
\end{tikzpicture}
$}
\]
\end{example}

Given a lattice path $\Gamma$, let $V(\Gamma)$ be the set of lattice points through which $\Gamma$ passes.
For two lattice paths $\Gamma$ and $\Gamma'$, we write $\Gamma = \Gamma'$ if $V(\Gamma) = V(\Gamma')$.
Now, we define the following equivalence relation on $\IGLTm{\lambda}{m}$.

\begin{definition}\label{def: equiv relation}
Let $\lambda \vdash n$ and $T_{1},T_{2} \in \IGLTm{\lambda}{m}$. 
The equivalence relation $\sim_{\lambda;m}$ on $\IGLTm{\lambda}{m}$ is defined by $T_{1} \sim_{\lambda;m} T_{2}$ if and only if 
\[
\left\{ \left(\Gamma_i(T_1), T_1^{-1}(i) \right) \; \middle| \; i \in \calI(T_1) \right\} 
= \left\{ \left(\Gamma_i(T_2), T_2^{-1}(i) \right) \; \middle| \; i \in \calI(T_2) \right\}.
\]
\end{definition}
If $\lambda$ and $m$ are clear in the context, we will drop the subscript from $\sim_{\lambda;m}$.
Let $\calE_{\lambda;m}$ be the set of equivalence classes of $\IGLTm{\lambda}{m}$ with respect to $\sim$.

\begin{example}
Let 
$$
\def\pp {0.38}
\def\hp {5.5}
\begin{array}{l}
\begin{tikzpicture}
\node at (-1.5,0) {$T_1 = $};
\node at  (0,0) {\scalebox{0.75}{$
\begin{ytableau}
1 & \color{red} 2 & \color{blue} 3 & \color{brown} 5 \\
\color{red} 2 & 4 & \color{brown} 5 & 6 \\
\color{blue} 3 & \color{brown} 5 & 7
\end{ytableau}$}};
\draw[line width = 0.5mm, color = red]  (-2*\pp, -0.5*\pp) -- (-2*\pp, 0.5*\pp) -- (-1*\pp, 0.5*\pp) -- (-1*\pp, 1.5*\pp) -- (0*\pp, 1.5*\pp);
\draw[line width = 0.5mm, color = blue]  (-2*\pp, -1.5*\pp) -- (-2*\pp, -0.5*\pp) -- (-1*\pp, -0.5*\pp)  -- (-1*\pp, 0.5*\pp) -- (0*\pp, 0.5*\pp) -- (0*\pp, 1.5*\pp) -- (1*\pp, 1.5*\pp);
\draw[line width = 0.5mm, color = brown]  (-1*\pp, -1.5*\pp) -- (-1*\pp, -0.5*\pp) -- (0*\pp, -0.5*\pp)  -- (0*\pp, 0.5*\pp) -- (1*\pp, 0.5*\pp) -- (1*\pp, 1.5*\pp) -- (2*\pp, 1.5*\pp);
\end{tikzpicture}
\end{array},
\ \ 
\begin{array}{l}
\begin{tikzpicture}
\node at (-1.5,0) {$T_2 = $};
\node at  (0,0) {\scalebox{0.75}{$
\begin{ytableau}
1 & \color{red} 2 & \color{blue} 3 & \color{brown} 5 \\
\color{red} 2 & 4 & \color{brown} 5 & 7 \\
\color{blue} 3 & \color{brown} 5 & 6
\end{ytableau}$}};
\draw[line width = 0.5mm, color = red]  (-2*\pp, -0.5*\pp) -- (-2*\pp, 0.5*\pp) -- (-1*\pp, 0.5*\pp) -- (-1*\pp, 1.5*\pp) -- (0*\pp, 1.5*\pp);
\draw[line width = 0.5mm, color = blue]  (-2*\pp, -1.5*\pp) -- (-2*\pp, -0.5*\pp) -- (-1*\pp, -0.5*\pp)  -- (-1*\pp, 0.5*\pp) -- (0*\pp, 0.5*\pp) -- (0*\pp, 1.5*\pp) -- (1*\pp, 1.5*\pp);
\draw[line width = 0.5mm, color = brown]  (-1*\pp, -1.5*\pp) -- (-1*\pp, -0.5*\pp) -- (0*\pp, -0.5*\pp)  -- (0*\pp, 0.5*\pp) -- (1*\pp, 0.5*\pp) -- (1*\pp, 1.5*\pp) -- (2*\pp, 1.5*\pp);
\end{tikzpicture}
\end{array},
\ \ 
\begin{array}{l}
\begin{tikzpicture}
\node at (-1.5,0) {$T_3 = $};
\node at  (0,0) {\scalebox{0.75}{$
\begin{ytableau}
1 & \color{red} 2 & \color{blue} 3 & \color{brown} 4 \\
\color{red} 2 & \color{blue} 3 & 5 & 6 \\
\color{blue} 3 & \color{brown} 4 & 7
\end{ytableau}$}};
\draw[line width = 0.5mm, color = red]  (-2*\pp, -0.5*\pp) -- (-2*\pp, 0.5*\pp) -- (-1*\pp, 0.5*\pp) -- (-1*\pp, 1.5*\pp) -- (0*\pp, 1.5*\pp);
\draw[line width = 0.5mm, color = blue]  (-2*\pp, -1.5*\pp) -- (-2*\pp, -0.5*\pp) -- (-1*\pp, -0.5*\pp)  -- (-1*\pp, 0.5*\pp) -- (0*\pp, 0.5*\pp) -- (0*\pp, 1.5*\pp) -- (1*\pp, 1.5*\pp);
\draw[line width = 0.5mm, color = brown]  (-1*\pp, -1.5*\pp) -- (-1*\pp, -0.5*\pp) -- (0*\pp, -0.5*\pp)  -- (0*\pp, 0.5*\pp) -- (1*\pp, 0.5*\pp) -- (1*\pp, 1.5*\pp) -- (2*\pp, 1.5*\pp);
\end{tikzpicture}
\end{array},
\ \ 
\text{and}
\ \ 
\begin{array}{l}
\begin{tikzpicture}
\node at (-1.5,0) {$T_4 = $};
\node at  (0,0) {\scalebox{0.75}{$
\begin{ytableau}
1 & \color{red} 2 & \color{blue} 4 & \color{brown} 5 \\
\color{red} 2 & 3 & \color{brown} 5 & 6 \\
\color{blue} 4 & \color{brown} 5 & 7
\end{ytableau}$}};
\draw[line width = 0.5mm, color = red]  (-2*\pp, -0.5*\pp) -- (-2*\pp, 0.5*\pp) -- (-1*\pp, 0.5*\pp) -- (-1*\pp, 1.5*\pp) -- (0*\pp, 1.5*\pp);
\draw[line width = 0.5mm, color = blue]  (-2*\pp, -1.5*\pp) 
-- (-2*\pp, -0.5*\pp + 0.05*\pp) 
-- (-1*\pp - 0.05*\pp, -0.5*\pp + 0.05*\pp)
-- (0*\pp - 0.05*\pp, -0.5*\pp + 0.05*\pp)
-- (0*\pp - 0.05*\pp, 0.5*\pp + 0.05*\pp)
-- (0*\pp - 0.05*\pp, 1.5*\pp)
-- (1*\pp, 1.5*\pp);
\draw[line width = 0.5mm, color = brown]  (-1*\pp, -1.5*\pp)
-- (-1*\pp, -0.5*\pp - 0.05*\pp)
-- (0*\pp + 0.05*\pp, -0.5*\pp - 0.05*\pp)
-- (0*\pp + 0.05*\pp, 0.5*\pp - 0.05*\pp)
-- (1*\pp + 0.05*\pp, 0.5*\pp - 0.05*\pp)
-- (1*\pp + 0.05*\pp, 1.5*\pp)
-- (2*\pp, 1.5*\pp);
\end{tikzpicture}
\end{array}.
$$
Then, $T_1 \sim_{(4,3,2);5} T_2$, but
$T_1 \not \sim_{(4,3,2);5} T_k$ for $k = 3,4$. 
\end{example}

\begin{theorem}\label{thm: preserving} 
Let $m$ and $n$ be positive integers with $m \le n$ and let $\lambda \vdash n$.
For any $1 \le i \le m-1$ and $E \in \calE_{\lambda;m}$, $\pi_i \cdot \C E \subseteq \C E$.
\end{theorem}

\begin{proof}
Let $T \in \IGLTm{\lambda}{m}$.
If $\pi_{i} \cdot T$ is $T$ or $0$, then $\pi_{i} \cdot T$ is clearly contained in $E$.
Therefore, we may assume that $\pi_{i} \cdot T = s_{i} \cdot T$.
In this case, all $i$'s are strictly above all $(i+1)$'s in $T$.
This implies that acting $\pi_{i}$ on $T$ does not change any lattice paths.
In addition, from the definition of $\pi_{i}$-action on $T$, it follows that for any $j \in \calI(T)$,
\begin{align*}
\left(\Gamma_j(T), T^{-1}(j) \right) 
= \begin{cases}
\left(
\Gamma_j(\pi_i \cdot T), (\pi_i \cdot T)^{-1}(j) \right) 
& \text{if $j \in \calI(T) \setminus \{i, i+1\}$}, 
\\
\left(\Gamma_{s_i(j)}(\pi_i \cdot T), (\pi_i \cdot T)^{-1}(s_i(j)) \right) 
& \text{if $j \in \calI(T) \cap \{i, i+1\}$}.
\end{cases} 
\end{align*}
Thus, we have that
\[
\left\{ \left(\Gamma_j(T), T^{-1}(j) \right) \; \middle| \; j \in \calI(T) \right\} 
= \left\{ \left(\Gamma_j(\pi_i \cdot T), (\pi_i \cdot T)^{-1}(j) \right) \; \middle| \; j \in \calI(\pi_i \cdot T) \right\},
\]
which implies that $\pi_i \cdot T \in E$.
\end{proof}

For each $E \in \calE_{\lambda;m}$, let $\bfG_E$ be the $H_m(0)$-submodule of $\bfG_{\lambda;m}$ whose underlying space is the $\C$-span of $E$.
Then, we have the following direct sum decomposition
\begin{equation*}
\bfG_{\lambda;m} = \bigoplus_{E \in \calE_{\lambda;m}} \bfG_E.
\end{equation*}

\section{Source and sink tableaux}
\label{Sec: source and sink}

The goal of this section is to show that there are two distinguished tableaux, called \emph{source} and \emph{sink tableaux}, in each equivalence class $E \in \calE_{\lambda; m}$.
To achieve our goal, we first give a characterization for source and sink tableaux.
Then, we construct two tableaux $\source(T)$ and $\sink(T)$ for each $T \in E$.
Finally, we verify that $\source(T)$ (resp. $\sink(T)$) is the unique source tableau (resp. sink tableau) in $E$, where $T$ is an arbitrary chosen element in $E$.
Hereafter, $E$ denotes an equivalence class of $\IGLTm{\lambda}{m}$ with respect to $\sim$ and $T$ denotes a tableau contained in $\IGLTm{\lambda}{m}$ unless otherwise stated.

To begin with, we give definitions for source tableaux and sink tableaux in $\IGLTm{\lambda}{m}$.

\begin{definition}\label{def: source and sink tableaux}
Let $T \in \IGLTm{\lambda}{m}$.
\begin{enumerate}[label = {\rm (\arabic*)}, itemsep = 0.5ex]
\item $T$ is said to be a \emph{source tableau} if there does not exist $T' \in \IGLTm{\lambda}{m}$ and $1 \le i \le m-1$ such that $\pi_i \cdot T' = T$ and $T' \neq T$.
\item $T$ is said to be a \emph{sink tableau} if there does not exist $T' \in \IGLTm{\lambda}{m}$ and $1 \le i \le m-1$ such that $\pi_i \cdot T = T'$ and $T' \neq T$.
\end{enumerate}
\end{definition}

The following lemma characterizes source and sink tableaux.

\begin{lemma}\label{Lem: source and sink cond}
The following hold.
\hfill
\begin{enumerate}[label = {\rm (\arabic*)}]
\item $T$ is a source tableau if and only if for all $i \notin \Des(T)$, $T((r_\sft^{(i)}, c_\sft^{(i)} + 1)) = i+1$, that is,  the box right adjacent to $\Top_i(T)$ is filled with $i+1$.
\item $T$ is a sink tableau if and only if $i$ is an attacking descent for all $i \in \Des(T)$.
\end{enumerate}
\end{lemma}

\begin{proof}
(1) 
To prove the ``only if'' part, suppose that $T$ is a source tableau. 
Assume, on the contrary, there exists $i \notin \Des(T)$ such that $i+1$ does not appear in the box $(r_\sft^{(k)}, c_\sft^{(k)} + 1)$.
Then one can easily see that $s_{i} \cdot T$ is contained in $\IGLTm{\lambda}{m}$.
This contradicts the assumption that $T$ is a source tableau because $\pi_i \cdot (s_{i} \cdot T) = T$.

Next, let us prove the ``if'' part.
Suppose contrary that $T$ is not a source tableau.
Then, there exists $T' \in \IGLTm{\lambda}{m}$ and $1 \le i \le n-1$ such that $\pi_i \cdot T' = T$ and $T' \neq T$.
This implies that $i \notin \Des(T)$, that is, $r_\sft^{(i)}(T) > r_\sfb^{(i+1)}(T)$.
Thus, the box right adjacent to $\Top_i(T)$ cannot be filled with $i+1$.

(2) The assertion immediately follows from the definitions of sink tableaux and attacking descents.
\end{proof}

\subsection{Existence and uniqueness of source tableaux in $E$}
\label{subsec: source tableau}

In this subsection, we construct the desired tableau $\source(T)$ and show that it is the unique source tableau in $E$.
To do this, we need the following preparation.

Given two lattice points $P$ and $P'$ in the same row, we denote the horizontal line from $P$ to $P'$ by $\tHL(P,P')$.
For each $i \in \calI(T)$, we define a new lattice path
$\tGam_i(T)$ by extending $\Gamma_i(T)$ with the following algorithm.

\begin{algorithm}\label{alg: tilde Gamma}
Fix $i \in \calI(T)$.
\begin{enumerate}[label = {\it Step \arabic*.}]
\item For each $j \in \calI(T)$, set $\Gamma'_j$ to be the lattice path obtained by connecting the following three lattice paths: 
\[
\tHL(\grco{r_\sfb^{(j)}}{0}, \grco{r_\sfb^{(j)}}{c_\sfb^{(j)} - 1}),
\ \ 
\Gamma_j(T),
\ \  \text{and} \ \ 
\tHL(\grco{r_\sft^{(j)} -1}{c_\sft^{(j)}}, \grco{r_\sft^{(j)} - 1}{\lambda_{r_\sft^{(j)} - 1}}).
\]
Here, $\lambda_0 := \lambda_1$.
\item Set $\grco{r_\sft}{c_{\sft}}$ to be the lattice point in $V(\Gamma'_i)$ satisfying that
\[
r_{\sft} = \min\{ r \mid \grco{r}{c} \in V(\Gamma'_i) \}
\quad \text{and} \quad
c_{\sft} = \min\{ c \mid \grco{r_{\sft}}{c} \in V(\Gamma'_i) \}.
\]
\item If there exists $j \in \calI(T)$ such that
\begin{align}\label{Eq: top cross condition}
r' < r_\sft < r'' \quad \text{and} \quad c', c'' > c_\sft
\quad \text{for some $\grco{r'}{c'}, \grco{r''}{c''} \in V(\Gamma'_j)$,}
\end{align}
then go to {\it Step 4}. Otherwise, go to {\it Step 5}.
\item Let $j_0 = \min\{j \mid \text{$\Gamma'_j$ satisfies \cref{Eq: top cross condition}}\}$ and $c_{0} = \min \left\{ 
c \mid  \grco{r_{\sft}}{c} \in V(\Gamma'_{j_0})
\right\}$.
Then, let $\Gamma$ be the lattice path satisfying that
\[
V(\Gamma) = V(\Gamma'_i) \setminus \{\grco{r_\sft}{c} \mid c \ge c_0\} \cup \{ \grco{r}{c} \in V(\Gamma'_{j_0}) \mid r \le r_\sft \ \text{and} \  c \ge c_0 \}.
\]
Set $\Gamma'_i := \Gamma$. 
Go to {\it Step 2}.
\item Return $\tGam_i(T) := \Gamma'_i$ and terminate the algorithm.
\end{enumerate}
\end{algorithm}

If $T$ is clear in the context, we simply write the lattice path $\tGam_i(T)$ by $\tGam_i$ for $i \in \calI(T)$. 

\begin{example}\label{Eg: modified lattice path}
Let us revisit \cref{Eg: lattice path}.
By applying \cref{alg: tilde Gamma} to each $i \in \calI(T)$, we obtain $\tGam_{17}$, $\tGam_{21}$, $\tGam_{27}$, and $\tGam_{29}$ as follows:
\begin{displaymath}
\def\pp {0.512}
\def\hp {5.5}
\scalebox{0.75}{$
\begin{tikzpicture}
\node at (0 - \hp, 0) {$
\begin{ytableau}
1  & 6  & 10 & 14 & 22 & 24 & 26 & 27 \\
2  & 7  & 11 & 15 & 23 & 25 & 29    \\
3  & 8  & 12 & 16 & 28 & 29         \\
4  & 9  & 13 & \color{red} 17                      \\
5  & \color{red} 17 & 27                          \\
18 & 20                                \\
19 & 21                             \\
21   
\end{ytableau}
$};

\draw[line width = 0.5mm, color = red] 
(-4*\pp - \hp, -1*\pp) -- (-3*\pp - \hp, -1*\pp) -- (-3*\pp - \hp, 0*\pp) -- (-1*\pp - \hp, 0*\pp) -- (-1*\pp - \hp, 1*\pp) -- (0*\pp - \hp, 1*\pp);

\draw[line width = 0.5mm, color = red] 
(0*\pp - \hp, 1*\pp) -- (0*\pp - \hp, 2*\pp) -- (2*\pp - \hp, 2*\pp) -- (2*\pp - \hp, 3*\pp) -- (3*\pp - \hp, 3*\pp) -- (3*\pp - \hp, 4*\pp) -- (4*\pp - \hp, 4*\pp);

\node[] at (-\hp, -2.5) {\large $\tGam_{17}(T)$};

\node at (0,0) {$
\begin{ytableau}
1  & 6  & 10 & 14 & 22 & 24 & 26 & 27 \\
2  & 7  & 11 & 15 & 23 & 25 & 29    \\
3  & 8  & 12 & 16 & 28 & 29         \\
4  & 9  & 13 & 17                      \\
5 & 17 & 27                          \\
18 & 20                                \\
19 & \color{violet} 21                             \\
\color{violet} 21                                     
\end{ytableau}
$};

\draw[line width = 0.5mm, color = red] (-4*\pp, -4*\pp) -- (-4*\pp, -3*\pp) -- (-3*\pp, -3*\pp) -- (-3*\pp, -2*\pp) -- (-2*\pp, -2*\pp);
\node[] at (0, -2.5) {\large $\tGam_{21}(T)$};

\node at (0 + \hp,0) {$
\begin{ytableau}
1  & 6  & 10 & 14 & 22 & 24 & 26 & \color{blue} 27 \\
2  & 7  & 11 & 15 & 23 & 25 & 29    \\
3  & 8  & 12 & 16 & 28 & 29         \\
4 & 9  & 13 & 17                      \\
5 & 17 & \color{blue} 27                          \\
18 & 20                                \\
19 & 21                             \\
21                                 
\end{ytableau}
$};

\draw[line width = 0.5mm, color = red] 
(-4*\pp + \hp, -1*\pp) -- (-2*\pp + \hp, -1*\pp) -- (-2*\pp + \hp, 0*\pp) -- (0*\pp + \hp, 0*\pp) -- (0*\pp + \hp, 2*\pp) -- (2*\pp + \hp, 2*\pp) -- (2*\pp + \hp, 3*\pp) -- (3*\pp + \hp, 3*\pp) -- (3*\pp + \hp, 4*\pp) -- (4*\pp + \hp, 4*\pp);
\node[] at (\hp, -2.5) {\large $\tGam_{27}(T)$};

\node at (0 + 2*\hp, 0) {$
\begin{ytableau}
1  & 6  & 10 & 14 & 22 & 24 & 26 & 27 \\
2  & 7  & 11 & 15 & 23 & 25 & \color{brown} 29    \\
3  & 8  & 12 & 16 & 28 & \color{brown} 29         \\
4  & 9  & 13 & 17                      \\
5  & 17 & 27                          \\
18 & 20                                \\
19 & 21                             \\
21                                     
\end{ytableau}
$};

\draw[line width = 0.5mm, color = red] 
(-4*\pp + 2*\hp, 1*\pp) -- (1*\pp + 2*\hp, 1*\pp) -- (1*\pp + 2*\hp, 2*\pp) -- (2*\pp + 2*\hp, 2*\pp) -- (2*\pp + 2*\hp, 3*\pp) -- (4*\pp + 2*\hp, 3*\pp);
\node[] at (2*\hp, -2.5) {\large $\tGam_{29}(T)$};
\end{tikzpicture}
$}
\end{displaymath}
\end{example}

For convenience, we introduce some terminologies related to $\tGam_i$'s.
For any $(r,c) \in \tyd(\lambda)$ and $i \in \calI(T)$, we say that \emph{$(r,c)$ is below} $\tGam_i$ if there exists $0 \le r' < r$ such that $\grco{r'}{c-1}, \grco{r'}{c} \in V(\tGam_i)$.
Otherwise, we say that \emph{$(r,c)$ is above} $\tGam_i$.
For each $i \in \calI(T)$, we call the path
$\tHL(\grco{r_\sfb^{(i)}}{0}, \grco{r_\sfb^{(i)}}{c_\sfb^{(i)} - 1})$ the \emph{bottom path of $\tGam_i$}.
Given $i,j \in \calI(T)$, if there exist $\grco{r'}{c'}, \grco{r''}{c''} \in V(\tGam_{j})$ such that $r' < r^{(i)}_\sfb < r''$ and $c',c'' < c_\sfb^{(i)}$, then we say that \emph{$\tGam_{j}$ crosses the bottom path of $\tGam_{i}$}.

In order to enumerate the lattice paths $\tGam_i$'s in appropriate order, to each $i \in \calI(T)$, 
we will give a label $\sfp_T(i) \in \{1,2,\ldots, |\calI(T)| \}$.
To do this, for $i \in \calI(T)$, we set $\sfp'_i \in \{1,2,\ldots, |\calI(T)|\}$ satisfying the following:
Let $i,j \in \calI(T)$.
\begin{enumerate}[label = {\bf C\arabic*.}]
\item If $r_\sfb^{(i)} < r_\sfb^{(j)}$, then $\sfp'_i < \sfp'_j$.
\item If $r_\sfb^{(i)} > r_\sfb^{(j)}$, then $\sfp'_i > \sfp'_j$.
\item When $r_\sfb^{(i)} = r_\sfb^{(j)}$, consider the lowest lattice point $p \in V(\tGam_i) \cap V(\tGam_j)$ such that neither $p + \grco{-1}{0}$ nor $p + \grco{0}{1}$ are contained in $V(\tGam_i) \cap V(\tGam_j)$.
If $p + \grco{-1}{0} \in V(\tGam_i)$, then $\sfp'_i < \sfp'_j$. Otherwise, $\sfp'_i > \sfp'_j$.
\end{enumerate}
We notice that $\{\sfp'_{i} \mid i\in \calI(T)\} = \{1,2,\ldots, |\calI(T)|\}$. 
By rearranging $\sfp_i'$'s with the following algorithm,
we define a function $\sfp_T: \calI(T) \ra \{1,2,\ldots, |\calI(T)|\}$.

\begin{algorithm}\label{alg: tsfp} 
For each $i \in \calI(T)$, let $p_i := \sfp'_i$, where $\sfp'_i$ is the index defined above.
\begin{enumerate}[label = {\it Step \arabic*.}]
\item Let $k = 1$.
\item Take $i_k$ and $i_{k+1}$ in $\calI(T)$ such that $p_{i_k} = k$ and $p_{i_{k+1}} = k+1$.
\item If $\tGam_{i_{k+1}}$ crosses the bottom path of $\tGam_{i_{k}}$, then set $p_{i_k} := k + 1$ and $p_{i_{k+1}} := k$ and go to {\it Step 1}. Otherwise, go to {\it Step 4}.
\item If $k < |\calI(T)| - 1$, then set $k = k+1$ and go to {\it Step 2}. 
Otherwise, set $\sfp_T(i) := p_i$ for each $i\in \calI(T)$ and go to {\it Step 5}. 
\item Return $(\sfp_T(i))_{i \in \calI(T)}$ and terminate the algorithm.
\end{enumerate}
\end{algorithm}

By the construction of $\sfp_T$, it is clear that $\sfp_T$ is a bijection.

\begin{lemma}\label{lem: sfpT for source}
Given a source tableau $T$, enumerate the elements of $\calI(T)$ in increasing order $j_1 < j_2 < \cdots < j_{|\calI(T)|}$.
Then, $\sfp_T(j_u) = u$ for all $1 \le u \le |\calI(T)|$.
\end{lemma}

\begin{proof}
We claim that $\sfp_T(j_{u}) < \sfp_T(j_{u+1})$ for all $1 \le u < |\calI(T)|$.
Given $1 \le u < |\calI(T)|$, we have two cases
\[
r_\sfb^{(j_u)} \le r_\sfb^{(j_{u+1})} 
\quad \text{and} \quad
r_\sfb^{(j_u)} > r_\sfb^{(j_{u+1})}.
\]

In case where $r_\sfb^{(j_u)} \le r_\sfb^{(j_{u+1})}$, we have that $\Bot_{j_u}(T)$ is above $\tGam_{j_{u+1}}$.
This implies that $\sfp'_{j_u}$ and $\sfp'_{j_{u+1}}$, defined in {\bf C1}-{\bf C3}, satisfy the inequality $\sfp'_{j_u} < \sfp'_{j_{u+1}}$.
By the construction of $\tGam_{j_{u}}$ and $\tGam_{j_{u+1}}$, $\tGam_{j_{u+1}}$ does not cross the bottom path of $\tGam_{j_{u}}$.
Thus, when applying \cref{alg: tsfp}, $\sfp'_{j_u}$ and $\sfp'_{j_{u+1}}$ are never swapped in {\it Step~3}.
This shows that $\sfp_T(j_u) < \sfp_T(j_{u+1})$.

In case where $r_\sfb^{(j_u)} > r_\sfb^{(j_{u+1})}$, we have that $\sfp'_{j_u} > \sfp'_{j_{u+1}}$.
It follows from~\cref{Lem: source and sink cond}(1) that for each $j_u < i \le j_{u+1}$, there exists a box in $T^{-1}(i)$ which appears weakly below $\Top_{j_u}(T)$, therefore $r_\sft^{(j_u)} \le r_{\sfb}^{(j_{u+1})}$.
By the construction of $\tGam_{j_{u+1}}$, each box $(r,c)$ which appears below $\tGam_{j_{u+1}}$ and satisfies $r \le r_{\sfb}^{(j_{u+1})}$ is filled with an integer greater than $j_{u+1}$.
This implies that $\Top_{j_u}(T)$ is above $\tGam_{j_{u+1}}$.
Combining this with the assumption $r_\sfb^{(j_u)} > r_\sfb^{(j_{u+1})}$, we have that $\tGam_{j_{u}}$ crosses the bottom path of $\tGam_{j_{u+1}}$.
Let
\begin{align*}
\calI_1 &:= \{i \in \calI(T) \mid i < j_u \text{ and } \sfp'_{j_{u+1}} < \sfp'_i < \sfp'_{j_u} \}, \\
\calI_2 &:= \{i \in \calI(T) \mid i > j_{u+1} \text{ and } \sfp'_{j_{u+1}} < \sfp'_i < \sfp'_{j_u} \},
\quad \text{and}\\
\calI_3 &:= \{i \in \calI(T) \mid \text{$\tGam_{j_{u+1}}$ crosses the bottom path of $\tGam_{i}$} \}.
\end{align*}
Note that $\sfp'_{j_{u}} = \sfp'_{j_{u+1}} + |\calI_1| + |\calI_2| + 1$.
One can see that, when applying~\cref{alg: tsfp}, we encounter the situation that
\[
k = \sfp'_{j_{u+1}} + |\calI_1| + |\calI_3|, 
\quad
i_k = {j_{u+1}},
\quad \text{and}\quad
i_{k+1} = {j_{u}}
\]
in {\it Step 2}.
In this situation, after applying {\it Step 3}, we have $p_{j_u} = k < k+1 = p_{j_{u+1}}$.
Since $\tGam_{j_{u+1}}$ cannot cross the bottom path of $\tGam_{j_{u}}$, the relative order $p_{j_u} < p_{j_{u+1}}$ does not change until the algorithm terminates.
Thus, we have that $\sfp_T(j_u) < \sfp_T(j_{u+1})$.

Since we have shown that $\sfp_T(j_u) < \sfp_T(j_{u+1})$ for all $1 \le u < |\calI(T)|$, we immediately have that
$\sfp_T(j_u) = u$ for all $1 \le u \le |\calI(T)|$.
\end{proof}

For convenience, we simply write the lattice path $\tGam_{\sfp^{-1}_{T}(u)}$ by $\tGam^{(u)}$ for $u \in [1,|\calI(T)|\,]$.
Given $u \in [1,|\calI(T)|\,]$, let $\sfA_{u}$ be the subdiagram of $\tyd(\lambda)$ consisting of the boxes located above $\tGam^{(u)}$. 
Then, define 
\begin{align}\label{Eq: sfD(i)}
\sfD_{u}^{(1)}(T) :=
\sfA_{u} 
\setminus \Big(\bigcup_{1 \le v < u} \left(\sfA_{v} \cup T^{-1}(\sfp_T^{-1}(v)) \right)\Big)
\quad \text{and} \quad
\sfD_{u}^{(2)}(T):= T^{-1}(\sfp_T^{-1}(u)).
\end{align}

\begin{example}\label{Eg: sfp and sfD compute}
Let us revisit \cref{Eg: lattice path} and \cref{Eg: modified lattice path}.
One can easily see that
\[
\sfp'_{17} = 2, \quad 
\sfp'_{21} = 4, \quad 
\sfp'_{27} = 3, \quad \text{and} \quad
\sfp'_{29} = 1.
\]
By applying~\cref{alg: tsfp}, one can compute $\sfp_T(i)$'s as \cref{Table: sfp compute}, where $\cdot$'s in the third and fourth columns are used to omit unnecessary information.
Consequently, we have
\[
\sfp_T(17) = 1, \quad 
\sfp_T(21) = 4, \quad 
\sfp_T(27) = 2, \quad \text{and} \quad
\sfp_T(29) = 3.
\]
We draw $\sfD_{u}^{(1)}(T)$ and $\sfD_{u}^{(2)}(T)$ for $u = 1,2,3,4$ in~\cref{Fig: sfD compute}. 
Here, asterisks and colored bullets are used to indicate the boxes in $\sfD_{u}^{(1)}(T)$ and $\sfD_{u}^{(2)}(T)$, respectively.
\end{example}

\begin{table}[t]
\renewcommand*\arraystretch{1.1}
\setlength{\tabcolsep}{9pt}
\centering
\begin{tabular}{c||c|c|c|c}
{\it Steps} & $k$ & $i_k$ & $i_{k+1}$ & $(p_{17}, p_{21}, p_{27}, p_{29})$  \\ \hline \hline
{\it Step 1} & 1 & $\cdot$ & $\cdot$ & $(2,4,3,1)$ \\ \hline
{\it Steps 2, 3} & 1 & 29 & 17 & $(1,4,3,2)$ \\ \hline
{\it Step 1} & 1 & $\cdot$ & $\cdot$ & $(1,4,3,2)$ \\ \hline
{\it Steps 2, 3} & 1 & 17 & 29 & $(1,4,3,2)$ \\ \hline
{\it Step 4} & 2 & $\cdot$ & $\cdot$ & $(1,4,3,2)$ \\ \hline
{\it Steps 2, 3} & 2 & 29 & 27 & $(1,4,2,3)$ \\ \hline
{\it Step 1} & 1 & $\cdot$ & $\cdot$ & $(1,4,2,3)$ \\ \hline
{\it Steps 2, 3} & 1 & 17 & 27 & $(1,4,2,3)$ \\ \hline
{\it Step 4} & 2 & $\cdot$ & $\cdot$ & $(1,4,2,3)$ \\ \hline
{\it Steps 2, 3} & 2 & 27 & 29 & $(1,4,2,3)$ \\ \hline
{\it Step 4} & 3 & $\cdot$ & $\cdot$ & $(1,4,2,3)$ \\ \hline
{\it Steps 2, 3} & 3 & 29 & 21 & $(1,4,2,3)$ \\ \hline
{\it Steps 4, 5} & 3 & $\cdot$ & $\cdot$ & $(1,4,2,3)$
\end{tabular}
\medskip
\caption{The process of obtaining $\sfp_T(i)$'s in~\cref{Eg: sfp and sfD compute}}
\label{Table: sfp compute}
\end{table}

\begin{figure}[t]
\begin{displaymath}
\def\pp {0.512}
\def\hp {5}
\def\sp {0.035}
\scalebox{0.75}{$
\begin{tikzpicture}
\node at (0 - \hp, 0) {$
\begin{ytableau}
\ast   & \ast  & \ast  & \ast  & \ast  & \ast  & \ast  & ~ \\
\ast   & \ast  & \ast  & \ast  & \ast  & \ast  & ~      \\
\ast   & \ast  & \ast  & \ast  & ~  & ~           \\
\ast   & \ast  & \ast  & \color{red} \bullet                     \\
\ast   & \color{red} \bullet  &  ~                         \\
~   & ~                               \\
~   & ~                               \\
~
\end{ytableau}
$};

\node at (-4.5*\pp - \hp - 0.1, -1*\pp + 0.2) {\color{red} \tiny $\tGam^{(1)}$};

\draw[line width = 0.5mm, color = red] 
(-4*\pp - \hp, -1*\pp + \sp) -- 
(-3*\pp - \hp - \sp, -1*\pp + \sp) -- 
(-3*\pp - \hp - \sp, 0*\pp + \sp) -- 
(-1*\pp - \hp - \sp, 0*\pp + \sp) -- 
(-1*\pp - \hp - \sp, 1*\pp + \sp) -- 
(0*\pp - \hp - \sp, 1*\pp + \sp) --
(0*\pp - \hp - \sp, 2*\pp + \sp) -- 
(2*\pp - \hp - \sp, 2*\pp + \sp) -- 
(2*\pp - \hp - \sp, 3*\pp + \sp) -- 
(3*\pp - \hp - \sp, 3*\pp + \sp) -- 
(3*\pp - \hp - \sp, 4*\pp + \sp) -- 
(4*\pp - \hp, 4*\pp + \sp);

\node at (-4.5*\pp - \hp - 0.1, -4*\pp) {\color{violet} \tiny $\tGam^{(4)}$};

\draw[line width = 0.5mm, color = violet] 
(-4*\pp - \hp, -4*\pp) -- 
(-4*\pp - \hp, -3*\pp) -- 
(-3*\pp - \hp, -3*\pp) -- 
(-3*\pp - \hp, -2*\pp) -- 
(-2*\pp - \hp, -2*\pp);

\node at (-4.5*\pp - \hp - 0.1, -1*\pp - 0.2) {\color{blue} \tiny $\tGam^{(2)}$};

\draw[line width = 0.5mm, color = blue] 
(-4*\pp - \hp, -1*\pp - \sp) -- 
(-2*\pp - \hp + \sp, -1*\pp - \sp) -- 
(-2*\pp - \hp + \sp, 0*\pp - \sp) -- 
(0*\pp - \hp + \sp, 0*\pp - \sp) -- 
(0*\pp - \hp + \sp, 2*\pp - \sp) -- 
(2*\pp - \hp + \sp, 2*\pp - \sp) -- 
(2*\pp - \hp + \sp, 3*\pp - \sp) -- 
(3*\pp - \hp + \sp, 3*\pp - \sp) -- 
(3*\pp - \hp + \sp, 4*\pp - \sp) -- 
(4*\pp - \hp, 4*\pp - \sp);

\node at (-4.5*\pp - \hp - 0.1, 1*\pp) {\color{brown} \tiny $\tGam^{(3)}$};

\draw[line width = 0.5mm, color = brown] 
(-4*\pp - \hp, 1*\pp) -- 
(1*\pp - \hp, 1*\pp) -- 
(1*\pp - \hp, 2*\pp) -- 
(2*\pp - \hp, 2*\pp) -- 
(2*\pp - \hp, 3*\pp) -- 
(4*\pp - \hp, 3*\pp);

\node[] at (-\hp, -3) {$\sfD_{1}^{(1)}(T)$ and $\sfD_{1}^{(2)}(T)$};

\node at (0, 0) {$
\begin{ytableau}
*(gray!70)   & *(gray!70)  & *(gray!70)  & *(gray!70)  & *(gray!70)  & *(gray!70)  & *(gray!70)  & \color{blue} \bullet \\
*(gray!70)   & *(gray!70)  & *(gray!70)  & *(gray!70)  & *(gray!70)  & *(gray!70)  & ~      \\
*(gray!70)   & *(gray!70)  & *(gray!70)  & *(gray!70)  & ~  & ~           \\
*(gray!70)   & *(gray!70)  & *(gray!70)  & *(gray!70)                     \\
*(gray!70)   & *(gray!70)  & \color{blue} \bullet                       \\
~   & ~                               \\
~   & ~                               \\
~
\end{ytableau}
$};


\draw[line width = 0.5mm, color = red] 
(-4*\pp , -1*\pp + \sp) -- 
(-3*\pp  - \sp, -1*\pp + \sp) -- 
(-3*\pp  - \sp, 0*\pp + \sp) -- 
(-1*\pp  - \sp, 0*\pp + \sp) -- 
(-1*\pp  - \sp, 1*\pp + \sp) -- 
(0*\pp  - \sp, 1*\pp + \sp) --
(0*\pp  - \sp, 2*\pp + \sp) -- 
(2*\pp  - \sp, 2*\pp + \sp) -- 
(2*\pp  - \sp, 3*\pp + \sp) -- 
(3*\pp  - \sp, 3*\pp + \sp) -- 
(3*\pp  - \sp, 4*\pp + \sp) -- 
(4*\pp , 4*\pp + \sp);


\draw[line width = 0.5mm, color = violet] 
(-4*\pp , -4*\pp) -- 
(-4*\pp , -3*\pp) -- 
(-3*\pp , -3*\pp) -- 
(-3*\pp , -2*\pp) -- 
(-2*\pp , -2*\pp);


\draw[line width = 0.5mm, color = blue] 
(-4*\pp , -1*\pp - \sp) -- 
(-2*\pp  + \sp, -1*\pp - \sp) -- 
(-2*\pp  + \sp, 0*\pp - \sp) -- 
(0*\pp  + \sp, 0*\pp - \sp) -- 
(0*\pp  + \sp, 2*\pp - \sp) -- 
(2*\pp  + \sp, 2*\pp - \sp) -- 
(2*\pp  + \sp, 3*\pp - \sp) -- 
(3*\pp  + \sp, 3*\pp - \sp) -- 
(3*\pp  + \sp, 4*\pp - \sp) -- 
(4*\pp , 4*\pp - \sp);


\draw[line width = 0.5mm, color = brown] 
(-4*\pp , 1*\pp) -- 
(1*\pp , 1*\pp) -- 
(1*\pp , 2*\pp) -- 
(2*\pp , 2*\pp) -- 
(2*\pp , 3*\pp) -- 
(4*\pp , 3*\pp);

\node[] at (0, -3) {$\sfD_{2}^{(1)}(T)$ and $\sfD_{2}^{(2)}(T)$};

\node at (0 + \hp, 0) {$
\begin{ytableau}
*(gray!70)   & *(gray!70)  & *(gray!70)  & *(gray!70)  & *(gray!70)  & *(gray!70)  & *(gray!70)  & *(gray!70) \\
*(gray!70)   & *(gray!70)  & *(gray!70)  & *(gray!70)  & *(gray!70)  & *(gray!70)  & \color{brown} \bullet      \\
*(gray!70)   & *(gray!70)  & *(gray!70)  & *(gray!70)  & \ast  & \color{brown} \bullet           \\
*(gray!70)   & *(gray!70)  & *(gray!70)  & *(gray!70)                     \\
*(gray!70)   & *(gray!70)  & *(gray!70)                        \\
~   & ~                               \\
~   & ~                               \\
~
\end{ytableau}
$};


\draw[line width = 0.5mm, color = red] 
(-4*\pp + \hp, -1*\pp + \sp) -- 
(-3*\pp + \hp - \sp, -1*\pp + \sp) -- 
(-3*\pp + \hp - \sp, 0*\pp + \sp) -- 
(-1*\pp + \hp - \sp, 0*\pp + \sp) -- 
(-1*\pp + \hp - \sp, 1*\pp + \sp) -- 
(0*\pp + \hp - \sp, 1*\pp + \sp) --
(0*\pp + \hp - \sp, 2*\pp + \sp) -- 
(2*\pp + \hp - \sp, 2*\pp + \sp) -- 
(2*\pp + \hp - \sp, 3*\pp + \sp) -- 
(3*\pp + \hp - \sp, 3*\pp + \sp) -- 
(3*\pp + \hp - \sp, 4*\pp + \sp) -- 
(4*\pp + \hp, 4*\pp + \sp);


\draw[line width = 0.5mm, color = violet] 
(-4*\pp + \hp, -4*\pp) -- 
(-4*\pp + \hp, -3*\pp) -- 
(-3*\pp + \hp, -3*\pp) -- 
(-3*\pp + \hp, -2*\pp) -- 
(-2*\pp + \hp, -2*\pp);


\draw[line width = 0.5mm, color = blue] 
(-4*\pp + \hp, -1*\pp - \sp) -- 
(-2*\pp + \hp + \sp, -1*\pp - \sp) -- 
(-2*\pp + \hp + \sp, 0*\pp - \sp) -- 
(0*\pp + \hp + \sp, 0*\pp - \sp) -- 
(0*\pp + \hp + \sp, 2*\pp - \sp) -- 
(2*\pp + \hp + \sp, 2*\pp - \sp) -- 
(2*\pp + \hp + \sp, 3*\pp - \sp) -- 
(3*\pp + \hp + \sp, 3*\pp - \sp) -- 
(3*\pp + \hp + \sp, 4*\pp - \sp) -- 
(4*\pp + \hp, 4*\pp - \sp);


\draw[line width = 0.5mm, color = brown] 
(-4*\pp + \hp, 1*\pp) -- 
(1*\pp + \hp, 1*\pp) -- 
(1*\pp + \hp, 2*\pp) -- 
(2*\pp + \hp, 2*\pp) -- 
(2*\pp + \hp, 3*\pp) -- 
(4*\pp + \hp, 3*\pp);

\node[] at (\hp, -3) {$\sfD_{3}^{(1)}(T)$ and $\sfD_{3}^{(2)}(T)$};

\node at (0 + 2*\hp, 0) {$
\begin{ytableau}
*(gray!70)   & *(gray!70)  & *(gray!70)  & *(gray!70)  & *(gray!70)  & *(gray!70)  & *(gray!70)  & *(gray!70) \\
*(gray!70)   & *(gray!70)  & *(gray!70)  & *(gray!70)  & *(gray!70)  & *(gray!70)  & *(gray!70)      \\
*(gray!70)   & *(gray!70)  & *(gray!70)  & *(gray!70)  & *(gray!70)  & *(gray!70)           \\
*(gray!70)   & *(gray!70)  & *(gray!70)  & *(gray!70)                     \\
*(gray!70)   & *(gray!70)  & *(gray!70)                        \\
\ast   & \ast                               \\
\ast   & \color{violet} \bullet                               \\
\color{violet} \bullet
\end{ytableau}
$};


\draw[line width = 0.5mm, color = red] 
(-4*\pp + 2*\hp, -1*\pp + \sp) -- 
(-3*\pp + 2*\hp - \sp, -1*\pp + \sp) -- 
(-3*\pp + 2*\hp - \sp, 0*\pp + \sp) -- 
(-1*\pp + 2*\hp - \sp, 0*\pp + \sp) -- 
(-1*\pp + 2*\hp - \sp, 1*\pp + \sp) -- 
(0*\pp + 2*\hp - \sp, 1*\pp + \sp) --
(0*\pp + 2*\hp - \sp, 2*\pp + \sp) -- 
(2*\pp + 2*\hp - \sp, 2*\pp + \sp) -- 
(2*\pp + 2*\hp - \sp, 3*\pp + \sp) -- 
(3*\pp + 2*\hp - \sp, 3*\pp + \sp) -- 
(3*\pp + 2*\hp - \sp, 4*\pp + \sp) -- 
(4*\pp + 2*\hp, 4*\pp + \sp);


\draw[line width = 0.5mm, color = violet] 
(-4*\pp + 2*\hp, -4*\pp) -- 
(-4*\pp + 2*\hp, -3*\pp) -- 
(-3*\pp + 2*\hp, -3*\pp) -- 
(-3*\pp + 2*\hp, -2*\pp) -- 
(-2*\pp + 2*\hp, -2*\pp);


\draw[line width = 0.5mm, color = blue] 
(-4*\pp + 2*\hp, -1*\pp - \sp) -- 
(-2*\pp + 2*\hp + \sp, -1*\pp - \sp) -- 
(-2*\pp + 2*\hp + \sp, 0*\pp - \sp) -- 
(0*\pp + 2*\hp + \sp, 0*\pp - \sp) -- 
(0*\pp + 2*\hp + \sp, 2*\pp - \sp) -- 
(2*\pp + 2*\hp + \sp, 2*\pp - \sp) -- 
(2*\pp + 2*\hp + \sp, 3*\pp - \sp) -- 
(3*\pp + 2*\hp + \sp, 3*\pp - \sp) -- 
(3*\pp + 2*\hp + \sp, 4*\pp - \sp) -- 
(4*\pp + 2*\hp, 4*\pp - \sp);


\draw[line width = 0.5mm, color = brown] 
(-4*\pp + 2*\hp, 1*\pp) -- 
(1*\pp + 2*\hp, 1*\pp) -- 
(1*\pp + 2*\hp, 2*\pp) -- 
(2*\pp + 2*\hp, 2*\pp) -- 
(2*\pp + 2*\hp, 3*\pp) -- 
(4*\pp + 2*\hp, 3*\pp);

\node[] at (2*\hp, -3) {$\sfD_{4}^{(1)}(T)$ and $\sfD_{4}^{(2)}(T)$};
\end{tikzpicture}
$}
\end{displaymath}
\caption{$\sfD_{u}^{(1)}(T)$ and $\sfD_{u}^{(2)}(T)$ for $u = 1,2,3,4$ in~\cref{Eg: sfp and sfD compute}}
\label{Fig: sfD compute}
\end{figure}

Now, we construct the desired tableau $\source(T)$ with the following algorithm.

\begin{algorithm}\label{Alg: construct source of T}
Let $T \in \IGLTm{\lambda}{m}$.
Set $e_0 = 0$ and $M_0 = 0$.
For $1 \le u \le |\calI(T)|$, let $e_u := |\sfD_u^{(1)}(T)| + 1$ and $M_u = \sum_{v = 0}^{u} e_v$.
\begin{enumerate}[label = {\it Step \arabic*.}]
\item Set $v := 1$.
\item Fill the boxes in $\sfD_v^{(1)}(T)$ by $M_{v-1} + 1, M_{v-1} + 2, \ldots, M_{v-1} + e_v - 1$ from left to right starting from the top.
\item Fill the boxes in $\sfD_v^{(2)}(T)$ by $M_v$.
\item If $v < |\calI(T)|$, then set $v := v + 1$ and go to {\it Step 2}. 
Otherwise, fill the remaining boxes by $M_{|\calI(T)|} + 1, M_{|\calI(T)|} + 2,\ldots, m$ from left to right starting from the top.
Set $\source(T)$ to be the resulting filling.
Return $\source(T)$ and terminate the algorithm.
\end{enumerate}
\end{algorithm}

\begin{example}
Revisit \cref{Eg: sfp and sfD compute}.
We see that
\[
\def\pp {0.512}
\def\hp {12}
\def\sp {0.035}
\scalebox{0.75}{$
\begin{tikzpicture}

\node at (-3 - \hp,0) {\fontsize{15.9}{15.9} $T \ \  = \quad $};
\node at (-\hp, 0) {$
\begin{ytableau}
1  & 6  & 10 & 14 & 22 & 24 & 26 & \color{blue} 27 \\
2  & 7  & 11 & 15 & 23 & 25 & \color{brown} 29    \\
3  & 8  & 12 & 16 & 28 & \color{brown} 29         \\
4  & 9  & 13 & \color{red} 17                      \\
5  &\color{red} 17 & \color{blue} 27              \\
18 & 20                                \\
19 & \color{violet} 21                             \\
\color{violet} 21                                     
\end{ytableau}
$};


\draw[line width = 0.5mm, color = red] 
(-4*\pp - \hp , -1*\pp + \sp) -- 
(-3*\pp - \hp  - \sp, -1*\pp + \sp) -- 
(-3*\pp - \hp  - \sp, 0*\pp + \sp) -- 
(-1*\pp - \hp  - \sp, 0*\pp + \sp) -- 
(-1*\pp - \hp  - \sp, 1*\pp + \sp) -- 
(0*\pp  - \hp - \sp, 1*\pp + \sp) --
(0*\pp  - \hp - \sp, 2*\pp + \sp) -- 
(2*\pp  - \hp - \sp, 2*\pp + \sp) -- 
(2*\pp  - \hp - \sp, 3*\pp + \sp) -- 
(3*\pp  - \hp - \sp, 3*\pp + \sp) -- 
(3*\pp  - \hp - \sp, 4*\pp + \sp) -- 
(4*\pp  - \hp, 4*\pp + \sp);


\draw[line width = 0.5mm, color = violet] 
(-4*\pp - \hp, -4*\pp) -- 
(-4*\pp - \hp, -3*\pp) -- 
(-3*\pp - \hp, -3*\pp) -- 
(-3*\pp - \hp, -2*\pp) -- 
(-2*\pp - \hp, -2*\pp);


\draw[line width = 0.5mm, color = blue] 
(-4*\pp - \hp, -1*\pp - \sp) -- 
(-2*\pp - \hp + \sp, -1*\pp - \sp) -- 
(-2*\pp - \hp + \sp, 0*\pp - \sp) -- 
(0*\pp  - \hp+ \sp, 0*\pp - \sp) -- 
(0*\pp  - \hp+ \sp, 2*\pp - \sp) -- 
(2*\pp  - \hp+ \sp, 2*\pp - \sp) -- 
(2*\pp  - \hp+ \sp, 3*\pp - \sp) -- 
(3*\pp  - \hp+ \sp, 3*\pp - \sp) -- 
(3*\pp  - \hp+ \sp, 4*\pp - \sp) -- 
(4*\pp  - \hp, 4*\pp - \sp);


\draw[line width = 0.5mm, color = brown] 
(-4*\pp- \hp , 1*\pp) -- 
(1*\pp - \hp, 1*\pp) -- 
(1*\pp - \hp, 2*\pp) -- 
(2*\pp - \hp, 2*\pp) -- 
(2*\pp - \hp, 3*\pp) -- 
(4*\pp - \hp, 3*\pp);

\node at (-7.5,0.5) {\small \cref{Alg: construct source of T}};
\draw[->] (-8.5,0) -- (-6.5,0);

\node at (-3.85,0) {\fontsize{15.9}{15.9} $\source(T) \  = \quad $};
\node at (2.5,0) {.};
\node at (0, 0) {$
\begin{ytableau}
1  & 2 & 3 & 4 & 5 & 6 & 7 & \color{blue} 23 \\
8  & 9 & 10& 11& 12& 13& \color{brown} 25      \\
14 & 15& 16& 17& 24  & \color{brown} 25           \\
18 & 19& 20& \color{red} 22                     \\
21 & \color{red} 22 & \color{blue} 23    \\
26   & 27                               \\
28   & \color{violet} 29                               \\
\color{violet} 29
\end{ytableau}
$};


\draw[line width = 0.5mm, color = red] 
(-4*\pp , -1*\pp + \sp) -- 
(-3*\pp  - \sp, -1*\pp + \sp) -- 
(-3*\pp  - \sp, 0*\pp + \sp) -- 
(-1*\pp  - \sp, 0*\pp + \sp) -- 
(-1*\pp  - \sp, 1*\pp + \sp) -- 
(0*\pp  - \sp, 1*\pp + \sp) --
(0*\pp  - \sp, 2*\pp + \sp) -- 
(2*\pp  - \sp, 2*\pp + \sp) -- 
(2*\pp  - \sp, 3*\pp + \sp) -- 
(3*\pp  - \sp, 3*\pp + \sp) -- 
(3*\pp  - \sp, 4*\pp + \sp) -- 
(4*\pp , 4*\pp + \sp);


\draw[line width = 0.5mm, color = violet] 
(-4*\pp , -4*\pp) -- 
(-4*\pp , -3*\pp) -- 
(-3*\pp , -3*\pp) -- 
(-3*\pp , -2*\pp) -- 
(-2*\pp , -2*\pp);


\draw[line width = 0.5mm, color = blue] 
(-4*\pp , -1*\pp - \sp) -- 
(-2*\pp  + \sp, -1*\pp - \sp) -- 
(-2*\pp  + \sp, 0*\pp - \sp) -- 
(0*\pp  + \sp, 0*\pp - \sp) -- 
(0*\pp  + \sp, 2*\pp - \sp) -- 
(2*\pp  + \sp, 2*\pp - \sp) -- 
(2*\pp  + \sp, 3*\pp - \sp) -- 
(3*\pp  + \sp, 3*\pp - \sp) -- 
(3*\pp  + \sp, 4*\pp - \sp) -- 
(4*\pp , 4*\pp - \sp);


\draw[line width = 0.5mm, color = brown] 
(-4*\pp , 1*\pp) -- 
(1*\pp , 1*\pp) -- 
(1*\pp , 2*\pp) -- 
(2*\pp , 2*\pp) -- 
(2*\pp , 3*\pp) -- 
(4*\pp , 3*\pp);

\end{tikzpicture}$}
\]
\end{example}

Let us collect some useful facts for $\source(T)$ which can be easily seen.
\begin{enumerate}[label = {\bf S\arabic*.}]
\item By \cref{Lem: source and sink cond}(1), for any $T \in \IGLTm{\lambda}{m}$, $\source(T)$ is a source tableau.
\item For any $T \in \IGLTm{\lambda}{m}$, $T \sim \source(T)$ by the construction of $\source(T)$.
\item By the construction of $\source(T)$, the set $\left\{ \left(\Gamma_i(T), T^{-1}(i) \right) \; \middle| \; i \in \calI(T) \right \}$ determines $\source(T)$.
In other words, if $T_1 \sim T_2$, then $\source(T_1) = \source(T_2)$.
\end{enumerate}

Combining the facts {\bf S1} and {\bf S2} shows the existence of source tableaux in $E$.
However, the above facts {\bf S1}, {\bf S2}, and {\bf S3} do not guarantee the uniqueness of source tableaux in $E$. 
To show the uniqueness, we need the lemma below.

\begin{lemma}\label{lem: tS(T)}
For any source tableau $T$, $\source(T) = T$.
\end{lemma}

\begin{proof}
Let $j_{0} = 0$ and $
\calI(T) = \{j_1 < j_2 < \cdots < j_{|\calI(T)|}\}$.
By \cref{lem: sfpT for source}, $\sfp_T(j_u) = u$ for all $1 \le u \le |\calI(T)|$, thus from the definition of $\sfD^{(2)}_u(T)$ we have that
\begin{align}\label{Eq: D(2)uT}
\sfD^{(2)}_u(T) = T^{-1}(j_u)  \quad \text{for all $1 \le u \le |\calI(T)|$.}
\end{align}

We claim that 
\begin{align}\label{eq: Du1 claim}
\sfD^{(1)}_{u}(T) = 
T^{-1}([j_{u-1} + 1, j_{u}-1]) \quad \text{for all $1 \le u \le |\calI(T)|$.}
\end{align}
First, let us prove the inclusion
$\sfD^{(1)}_{u}(T) \supseteq
T^{-1}([j_{u-1} + 1, j_{u}-1])$ for all $1 \le u \le |\calI(T)|$.
Take any $1 \le u \le |\calI(T)|$ and $i \in [j_{u-1} + 1, j_{u}-1]$. 
Let $B$ be the box filled with $i$ in $T$.
Recall that for any $1 \le v \le |\calI(T)|$, $\sfA_{v}$ is defined to be the subdiagram of $\tyd(\lambda)$ consisting of the boxes located above $\tGam^{(v)}$.
By the definition of $\sfD^{(1)}_{u}(T)$, the desired inclusion is obtained by proving that 
\[
B \in \sfA_u
\quad \text{and} \quad 
B \notin \sfA_v \quad \text{for all $1 \le v < u$.}
\]

Suppose for the sake of contradiction that $B \notin \sfA_u$.
Combining the fact that $T$ is an increasing tableau with the inequality $i < j_u$, we have that $B$ is strictly left of $\Bot_{j_u}(T)$.
This implies that $B$ is strictly below $\Bot_{j_u}(T)$.
By \cref{Lem: source and sink cond}(1), $i+1$ appears weakly below than $i$.
Again, by \cref{Lem: source and sink cond}(1), $i+2$ appears weakly below than $i+1$, so $i+2$ is weakly below than $i$.
Continuing this process, we see that $j_{u}$ appears weakly below than $i$, that is, $\Bot_{j_{u}}(T)$ is weakly below $B$.
This contradicts the above observation that $B$ is strictly below $\Bot_{j_u}(T)$.
Thus, $B \in \sfA_u$.

Suppose for the sake of contradiction that $B \in \sfA_v$ for some $1 \le v < u$.
Since $i > j_v$, $B$ cannot be placed weakly left of $\Top_{j_v}(T)$ while being above $\tGam_v$.
Therefore, $B$ is strictly right of $\Top_{j_v}(T)$.
In addition, by \cref{lem: sfpT for source}, $\sfp_T(j_v) = v \le u-1 = \sfp_T(j_{u-1})$, so the boxes strictly right of $\Top_{j_v}(T)$ while being above $\tGam_v$ are placed above $\tGam^{(u-1)}$.
Therefore, if we prove that $B$ is below $\tGam^{(u-1)}$, then we obtain a contradiction to the assumption that $B \in \sfA_v$.
Assume that $B$ is above $\tGam^{(u-1)}$.
Since $i > j_{u-1}$, $B$ is strictly above $\Top_{j_{u-1}}(T)$ by the construction of $\tGam^{(u-1)}$.
On the other hand, by using \cref{Lem: source and sink cond}(1) repeatedly, one can see that there exists at least one $j_{u-1}$ which appears weakly above $i$.
It follows that $\Top_{j_{u-1}}(T)$ is weakly above $B$.
This gives a contradiction to the previous observation, thus $B \notin \sfA_v$.

Next, let us prove the inclusion
$\sfD^{(1)}_{u}(T) \subseteq
T^{-1}([j_{u-1} + 1, j_{u}-1])$ for all $1 \le u \le |\calI(T)|$.
Equivalently, we claim that
\[
\sfD^{(1)}_{u}(T) \cap T^{-1}(i) = \emptyset \quad \text{for any $i \in [1,m] \setminus [j_{u-1} + 1, j_{u}-1]$.}
\]
To prove this, choose arbitrary $1 \le u \le |\calI(T)|$ and $i \in [1,m] \setminus [j_{u-1} + 1, j_{u}-1]$.
Note that if $i \in \calI(T)$, then $\sfD^{(1)}_{u}(T) \cap T^{-1}(i) = \emptyset$ because $\bigcup_{1\le v \le |\calI(T)|} \sfD^{(2)}_v(T) = T^{-1}(\calI(T))$ and $\sfD^{(1)}_{u}(T) \cap \sfD^{(2)}_{v}(T) = \emptyset$ for all $1 \le v \le |\calI(T)|$.
Therefore, we may assume that $i \notin \calI(T)$.
Take $u_0 \in \{1,2,\ldots, |\calI(T)|\} \setminus \{u\}$ such that $i \in [j_{u_0-1} +1, j_{u_0} - 1]$.
Since the method of proof for the case where $u_0 > u$ is similar with that for the case where $u_0 < u$, we only prove the latter case.

Suppose that $u_0 < u$.
Let $B$ be the box filled with $i$.
If $i$ appears below $\tGam_{j_{u_0}}$, then $B$ is strictly below $\Bot_{j_{u_0}}(T)$ since $i < j_{u_0}$.
On the other hand, for all $i \le j < j_{u_0}$, combining the fact $j \notin \calI(T)$ with \cref{Lem: source and sink cond}(1) yields that
$\Bot_{j+1}(T)$ is placed weakly below the unique box filled with $j$.
It follows that $\Bot_{j_{u_0}}(T)$ is weakly below $B$.
This contradicts the previous observation that $B$ is strictly below $\Bot_{j_{u_0}}(T)$.
This implies that $i$ must appear above $\tGam_{j_{u_0}}$.
If $i$ appears above $\tGam_{j_{u_0}}$, then $B \in \sfA_{u_0}$.
Thus, $\sfD^{(1)}_{u}(T) \cap T^{-1}(i) = \emptyset$ by \cref{Eq: sfD(i)}.

Now, combining \cref{Lem: source and sink cond}(1) with the equations \cref{Eq: D(2)uT} and \cref{eq: Du1 claim}, we have that $\sfD^{(1)}_{u}(T)$ is filled with $j_{u-1} + 1, j_{u-1} +2, \ldots, j_{u}-1$ from left to right starting from the top for each $1 \le u \le |\calI(T)|$.
Note that 
$$
T\Big(\tyd(\lambda) \setminus \bigcup_{1 \le u \le |\calI(T)|} \big( \sfD^{(1)}_{u}(T) \cup \sfD^{(2)}_{u}(T) \big) \Big) = \{j_{|\calI(T)|} + 1, j_{|\calI(T)|} + 2, \ldots, m\}.
$$
Again, by \cref{Lem: source and sink cond}(1), we have that $\tyd(\lambda) \setminus \bigcup_{1 \le u \le |\calI(T)|} (\sfD^{(1)}_{u}(T) \cup \sfD^{(2)}_{u}(T))$ is filled with $j_{|\calI(T)|} + 1, j_{|\calI(T)|} + 2, \ldots, m$ from left to right starting from the top.
Hence, by the construction of $\source(T)$, we have that $T = \source(T)$.
\end{proof}

Now, we prove the main theorem of this subsection.

\begin{theorem}\label{thm: uniqueness of source}
For each $E \in \calE_{\lambda;m}$, there exists a unique source tableau in $E$.
\end{theorem}

\begin{proof}
Recall that the existence is already shown by using {\bf S1} and {\bf S2}. 
For the uniqueness, suppose that $T_1$ and $T_2$ are source tableaux contained in $E$.
Combining {\bf S3} with \cref{lem: tS(T)} yields that
\[
T_1 = \source(T_1) = \source(T_2) = T_2.
\]
Hence, the source tableau in $E$ is unique.
\end{proof}

For each $E \in \calE_{\lambda;m}$, define
$$
T_E := \text{the unique source tableau contained in $E$.}
$$

\subsection{Existence and uniqueness of sink tableaux in $E$}
\label{subsec: sink tableau}

Similar to the previous subsection, we construct a tableau $\sink(T)$ and show that it is the unique sink tableau in $E$.
To do this, we need the following preparation.

Given two lattice points $P$ and $P'$ in the same column, we denote the vertical line from $P$ to $P'$ by $\mathtt{VL}(P,P')$.
For each $i \in \calI(T)$, we define a new lattice path $\hGam_i$ by extending $\Gamma_i$ with the following algorithm.

\begin{algorithm}\label{alg: hat Gamma}
Fix $i \in \calI(T)$.
\begin{enumerate}[label = {\it Step \arabic*.}]
\item For each $j \in \calI(T)$, set $\Gamma'_j$ to be the lattice path obtained by connecting the following three lattice paths:  
\[
\mathtt{VL}(\grco{r_\sft^{(i)} -1}{c_\sft^{(i)}}, \grco{0}{c_\sft^{(i)}}),
\ \ 
\Gamma_j(T),
\ \  \text{and} \ \ 
\mathtt{VL}(\grco{r_\sfb^{(i)}}{c_\sfb^{(i)} - 1}, \grco{R_\sfb^{(i)}}{c_\sfb^{(i)} - 1}),
\]
where $R_\sfb^{(i)}:=\max \{ r \mid (r,c_\sfb^{(i)} - 1) \in \tyd(\lambda) \}$.
\item Set $\grco{r_\sfb}{c_{\sfb}}$ to be the lattice point in $V(\Gamma'_i)$ satisfying that
\[
c_{\sfb} = \min\{ c \mid \grco{r}{c} \in V(\Gamma'_i) \}
\quad \text{and} \quad
r_{\sfb} = \min\{ r \mid \grco{r}{c_{\sfb}} \in V(\Gamma'_i) \}
.
\]
\item If there exists $j \in \calI(T)$ such that
\begin{align}\label{Eq: top cross condition2}
r', r'' > r_\sfb \quad \text{and} \quad c'< c_\sfb < c''
\quad \text{for some $\grco{r'}{c'}, \grco{r''}{c''} \in V(\Gamma'_j)$,}
\end{align}
then go to {\it Step 4}. Otherwise, go to {\it Step 5}.
\item Let $j_0 = \min\{j \mid \text{$\Gamma'_j$ satisfies \cref{Eq: top cross condition2}}\}$ and $r_{0} = \min \left\{ 
r \mid  \grco{r}{c_\sfb-1} \in V(\Gamma'_{j_0})
\right\}$.
Then, let $\Gamma$ be the lattice path satisfying that
\[
V(\Gamma) = V(\Gamma'_i) \setminus \{\grco{r}{c_\sfb-1} \mid r \ge r_0\} \cup \{ \grco{r}{c} \in V(\Gamma'_{j_{0}}) \mid r \ge r_{0} \ \text{and} \  c \le c_{\sfb} \}.
\]
Set $\Gamma'_i := \Gamma$. 
Go to {\it Step 2}.
\item Return $\hGam_i(T)=\Gamma'_i$ and terminate the algorithm.
\end{enumerate}
\end{algorithm}

If $T$ is clear in the context, we simply write the lattice path $\hGam_i(T)$ by $\hGam_i$ for $i \in \calI(T)$. 

\begin{example}\label{Eg: modified lattice path2}
Let us consider 
\[
T = \begin{array}{l}\scalebox{0.75}{$
\begin{ytableau}
1  & 2  & 3 & 4 & 5 & 6 & 7 & \color{blue} 23 \\
8  & 9  & 10 & 11 & 12 & 13 & \color{brown} 25    \\
14  & 15  & 16 & 17 & 24 & \color{brown} 25         \\
18  & 19  & 20 & \color{red} 22                      \\
21  &\color{red} 22 & \color{blue} 23              \\
26 & 27                                \\
28 & \color{violet} 29                             \\
\color{violet} 29                                     
\end{ytableau}$}
\end{array}.
\]
(In fact, $T$ is the source tableau in the equivalence class of the tableau given in \cref{Eg: lattice path}.)
By applying \cref{alg: hat Gamma} to each $i \in \calI(T)= \{ 22,23,25,29 \}$, one can see all $\hGam_i$'s as below.
\begin{displaymath}
\def\pp {0.512}
\def\hp {5.5}
\scalebox{0.75}{$
\begin{tikzpicture}
\node at (0 - \hp, 0) {$
\begin{ytableau}
1  & 2  & 3 & 4 & 5 & 6 & 7 & \color{blue} 23 \\
8  & 9  & 10 & 11 & 12 & 13 & \color{brown} 25    \\
14  & 15  & 16 & 17 & 24 & \color{brown} 25         \\
18  & 19  & 20 & \color{red} 22                      \\
21  &\color{red} 22 & \color{blue} 23              \\
26 & 27                                \\
28 & \color{violet} 29                             \\
\color{violet} 29                                   
\end{ytableau}
$};

\draw[line width = 0.5mm, color = red] 
(-4*\pp - \hp, -4*\pp) -- (-4*\pp - \hp, -3*\pp) -- (-3*\pp - \hp, -3*\pp) -- (-3*\pp - \hp, 0*\pp) -- (-1*\pp - \hp, 0*\pp) -- (-1*\pp - \hp, 1*\pp) -- (0*\pp - \hp, 1*\pp) -- (0*\pp - \hp, 4*\pp);

\node[] at (-\hp, -2.5) {$\hGam_{22}(T)$};

\node at (0,0) {$
\begin{ytableau}
1  & 2  & 3 & 4 & 5 & 6 & 7 & \color{blue} 23 \\
8  & 9  & 10 & 11 & 12 & 13 & \color{brown} 25    \\
14  & 15  & 16 & 17 & 24 & \color{brown} 25         \\
18  & 19  & 20 & \color{red} 22                      \\
21  &\color{red} 22 & \color{blue} 23              \\
26 & 27                                \\
28 & \color{violet} 29                             \\
\color{violet} 29                                     
\end{ytableau}
$};

\node at (0 + \hp,0) {$
\begin{ytableau}
1  & 2  & 3 & 4 & 5 & 6 & 7 & \color{blue} 23 \\
8  & 9  & 10 & 11 & 12 & 13 & \color{brown} 25    \\
14  & 15  & 16 & 17 & 24 & \color{brown} 25         \\
18  & 19  & 20 & \color{red} 22                      \\
21  &\color{red} 22 & \color{blue} 23              \\
26 & 27                                \\
28 & \color{violet} 29                             \\
\color{violet} 29                                   
\end{ytableau}
$};

\draw[line width = 0.5mm, color = red] 
(-2*\pp, -3*\pp) -- (-2*\pp, -1*\pp) -- (-2*\pp, 0*\pp) -- (0*\pp, 0*\pp) -- (0*\pp, 2*\pp) -- (2*\pp, 2*\pp) -- (2*\pp, 3*\pp) -- (3*\pp, 3*\pp) -- (3*\pp, 4*\pp) -- (4*\pp, 4*\pp);
\node[] at (0*\hp, -2.5) {$\hGam_{23}(T)$};

\node at (0 + 2*\hp, 0) {$
\begin{ytableau}
1  & 2  & 3 & 4 & 5 & 6 & 7 & \color{blue} 23 \\
8  & 9  & 10 & 11 & 12 & 13 & \color{brown} 25    \\
14  & 15  & 16 & 17 & 24 & \color{brown} 25         \\
18  & 19  & 20 & \color{red} 22                      \\
21  &\color{red} 22 & \color{blue} 23              \\
26 & 27                                \\
28 & \color{violet} 29                             \\
\color{violet} 29                                     
\end{ytableau}
$};

\draw[line width = 0.5mm, color = red] 
(1*\pp + \hp, 1*\pp) -- (1*\pp + \hp, 2*\pp) -- (2*\pp + \hp, 2*\pp) -- (2*\pp + \hp, 3*\pp) -- (3*\pp + \hp, 3*\pp) -- (3*\pp + \hp, 4*\pp);
\node[] at (\hp, -2.5) {$\hGam_{25}(T)$};

\draw[line width = 0.5mm, color = red] (-4*\pp + 2*\hp, -4*\pp) -- (-4*\pp + 2*\hp, -3*\pp) -- (-3*\pp + 2*\hp, -3*\pp) -- (-3*\pp + 2*\hp, -2*\pp) -- (-2*\pp + 2*\hp, -2*\pp) -- (-2*\pp + 2*\hp, 4*\pp);

\node[] at (2*\hp, -2.5) {$\hGam_{29}(T)$};

\end{tikzpicture}
$}
\end{displaymath}
\end{example}

For convenience, we introduce some terms related to $\hGam_i$'s.
For any $(r,c) \in \tyd(\lambda)$ and $i \in \calI(T)$, we say that \emph{$(r,c)$ is right of} $\hGam_i$ if there exists $0 \le c' < c$ such that $\grco{r-1}{c'}, \grco{r}{c'} \in V(\hGam_i)$.
Otherwise, we say that \emph{$(r,c)$ is left of} $\hGam_i$.
For each $i \in \calI(T)$, we call the path
$\mathtt{VL}(\grco{0}{c_\sft^{(i)}}, \grco{r_\sft^{(i)} -1}{c_\sft^{(i)}})$ the \emph{rightmost path of $\hGam_i$}.
Given $i,j \in \calI(T)$, if there exist $\grco{r'}{c'}, \grco{r''}{c''} \in V(\hGam_{j})$ such that $r',r'' < r^{(i)}_\sft$ and $c'< c_\sft^{(i)} < c''$, then we say that \emph{$\hGam_{j}$ crosses the rightmost path of $\hGam_{i}$}.

For each $i \in \calI(T)$, we will define a label $\sfq_i(T) \in \{1,2,\ldots, |\calI(T)| \}$ to enumerate the lattice paths $\{ \hGam_i \mid i \in \calI(T)\}$.
To do this, we first set $\sfq'_i \in \{1,2,\ldots, |\calI(T)|\}$ satisfying the following conditions:
Let $i,j \in \calI(T)$.
\begin{enumerate}[label = {\bf C\arabic*$\boldsymbol{'}$.}]
\item If $c_\sft^{(i)} < c_\sft^{(j)}$, then $\sfq'_i < \sfq'_j$.
\item If $c_\sft^{(i)} > c_\sft^{(j)}$, then $\sfq'_i > \sfq'_j$.
\item When $c_\sft^{(i)} = c_\sft^{(j)}$, consider the highest lattice point $q \in V(\hGam_i) \cap V(\hGam_j)$ such that neither $q + \grco{0}{-1}$ nor $q + \grco{1}{0}$ are contained in $V(\hGam_i) \cap V(\hGam_j)$.
If $q + \grco{0}{-1} \in V(\hGam_i)$, then $\sfq'_i < \sfq'_j$. Otherwise, $\sfq'_i > \sfq'_j$.
\end{enumerate}
We notice that $\{\sfq'_{i} \mid i\in \calI(T)\} = \{1,2,\ldots, |\calI(T)|\}$. By rearranging $\sfq_i'$'s with the following algorithm, we define a function $\sfq_{T} : \calI(T) \ra \{1,2,\ldots, |\calI(T)|\}$.

\begin{algorithm}\label{alg: tsfp2} 
For $i \in \calI(T)$, let $q_{i}:=\sfq'_i$, where $\sfq'_i$ is the index defined above.
\begin{enumerate}[label = {\it Step \arabic*.}]
\item Let $k = 1$.
\item Take $i_k$ and $i_{k+1}$ in $\calI(T)$ such that $q_{i_k} = k$ and $q_{i_{k+1}} = k+1$.
\item If $\hGam_{i_{k+1}}$ crosses the rightmost path of $\hGam_{i_{k}}$, then set $q_{i_k} := k + 1$ and $q_{i_{k+1}} := k$ and go to {\it Step 1}. Otherwise, go to {\it Step 4}.
\item If $k < |\calI(T)| - 1$, then set $k = k+1$ and go to {\it Step 2}. 
Otherwise, set $\sfq_T(i) := q_i$ for each $i\in \calI(T)$ and go to {\it Step 5}. 
\item Return $(\sfq_T(i))_{i \in \calI(T)}$ and terminate the algorithm.
\end{enumerate}
\end{algorithm}
By the construction of $\sfq_T$, it is clear that $\sfq_T$ is a bijection. 

\begin{lemma}\label{lem: sink order}
Given a sink tableau $T$, let us enumerate the elements of $\calI(T)$ in increasing order $j_1 < j_2 < \cdots < j_{|\calI(T)|}$.
Then, we have $\sfq_T(j_u) = u$ for all $1 \le u \le |\calI(T)|$.
\end{lemma}

\begin{proof}
We claim that $\sfq_T(j_u) < \sfq_T(j_{u+1})$ for any $1 \le u < |\calI(T)|$. Given $1 \le u < |\calI(T)|$, we have two cases
\[
c_\sft^{(j_u)} \le c_\sft^{(j_{u+1})} 
\quad \text{and} \quad
c_\sft^{(j_u)} > c_\sft^{(j_{u+1})}.
\]

In case where $c_\sft^{(j_u)} \le c_\sft^{(j_{u+1})}$, $\Top_{j_{u}}(T)$ is left of $\hGam_{j_{u+1}}$. 
This implies that $\sfq_{j_{u}}'$ and $\sfq_{j_{u+1}}'$, defined in {\bf C1$\boldsymbol{'}$}-{\bf C3$\boldsymbol{'}$}, satisfy the inequality $\sfq_{j_{u}}' < \sfq_{j_{u+1}}'$. By the construction of $\hGam_{j_{u}}$ and $\hGam_{j_{u+1}}$, $\hGam_{j_{u+1}}$ does not cross the rightmost path of $\hGam_{j_{u}}$. Thus, \cref{alg: tsfp2} does not allow that $\sfq_{j_{u}}'$ and $\sfq_{j_{u+1}}'$ are swapped in {\it Step~3}. This implies that $\sfq_T(j_u) < \sfq_T(j_{u+1})$.
 
In case where $c_\sft^{(j_u)} > c_\sft^{(j_{u+1})}$, we have that $\sfq_{j_{u}}' > \sfq_{j_{u+1}}'$. 
It follows from~\cref{Lem: source and sink cond}(2) that for each $j_u < i \le j_{u+1}$, there exists a box in $T^{-1}(i)$ which appears weakly right of $\Top_{j_u}(T)$, therefore $c_{\sfb}^{(j_{u})} \le c_{\sft}^{(j_{u+1})}$.
In addition, by the construction of $\hGam_{j_{u+1}}$, each box $(r,c)$ which appears right of $\hGam_{j_{u+1}}$ and satisfies $c \le c_{\sft}^{(j_{u+1})}$ is filled with an integer greater than $j_{u+1}$. 
This implies that $\Bot_{j_{u}}(T)$ is left of $\hGam_{j_{u+1}}$. 
Combining this with the assumption $c_\sft^{(j_u)} > c_\sft^{(j_{u+1})}$, we have that $\hGam_{j_{u}}$ crosses the rightmost path of $\hGam_{j_{u+1}}$.
Let
\begin{align*}
\calI_1 &:= \{i \in \calI(T) \mid i < j_u \text{ and } \sfq'_{j_{u+1}} < \sfq'_i < \sfq'_{j_u} \}, \\
\calI_2 &:= \{i \in \calI(T) \mid i > j_{u+1} \text{ and } \sfq'_{j_{u+1}} < \sfq'_i < \sfq'_{j_u} \},
\quad \text{and}\\
\calI_3 &:= \{i \in \calI(T) \mid \text{$\hGam_{j_{u+1}}$ crosses the rightmost path of $\hGam_{i}$} \}.
\end{align*}
Note that $\sfq'_{j_{u}} = \sfq'_{j_{u+1}} + |\calI_1| + |\calI_2| + 1$.
One can see that, when applying~\cref{alg: tsfp2}, we encounter the situation that
\[
k = \sfq'_{j_{u+1}} + |\calI_1| + |\calI_3|, 
\quad
i_k = {j_{u+1}},
\quad \text{and}\quad
i_{k+1} = {j_{u}}
\]
in {\it Step 2}.
In this situation, after applying {\it Step 3}, we have $q_{j_u} = k < k+1 = q_{j_{u+1}}$.
Since $\hGam_{j_{u+1}}$ cannot cross the rightmost path of $\hGam_{j_{u}}$, the relative order $q_{j_u} < q_{j_{u+1}}$ does not change until the algorithm terminates.
Thus, we have that $\sfq_T(j_u) < \sfq_T(j_{u+1})$.

Since we have shown that $\sfq_T(j_u) < \sfq_T(j_{u+1})$ for all $1 \le u < |\calI(T)|$, we immediately have that
$\sfq_T(j_u) = u$ for all $1 \le u \le |\calI(T)|$. 
\end{proof}

For convenience, we simply write the lattice path $\hGam_{\sfq^{-1}_{T}(u)}$ by $\hGam^{(u)}$ for $u \in [1,|\calI(T)|\,]$.
Given $u \in [1,|\calI(T)|\,]$, let $\sfhA_{u}$ be the subdiagram of $\tyd(\lambda)$ consisting of the boxes located left of $\hGam^{(u)}$. Then, we define
\begin{align}\label{Eq: sfhD(i)}
\sfhD_{u}^{(1)}(T) :=
\sfhA_{u} 
\setminus \Big(\bigcup_{1 \le v < u} (\sfhA_{v} \cup T^{-1}(\sfq_{T}^{-1}(v)))\Big) 
\quad \text{and} \quad 
\sfhD_{u}^{(2)}(T):= T^{-1}(\sfq_{T}^{-1}(u)).
\end{align}

\begin{example}\label{Eg: sfp and sfD compute2}
Let us revisit \cref{Eg: modified lattice path2}.
One can easily see that
\[
\sfq'_{22} = 2, \quad 
\sfq'_{23} = 4, \quad 
\sfq'_{25} = 3, \quad \text{and} \quad
\sfq'_{29} = 1.
\]
By applying~\cref{alg: tsfp2}, one can compute $\sfq_{T}(i)$'s as \cref{Table: sfq compute}, where $\cdot$'s in the third and fourth columns are used to omit unnecessary information.
Consequently, we have
\[
\sfq_{T}(22) = 1, \quad 
\sfq_{T}(23) = 3, \quad 
\sfq_{T}(25) = 4, \quad \text{and} \quad
\sfq_{T}(29) = 2.
\]
We draw $\sfhD_u^{(1)}(T)$ and $\sfhD_u^{(2)}(T)$ for $u = 1,2,3,4$ in~\cref{Fig: sfD compute2}. 
Here, asterisks and colored bullets are used to indicate the boxes in $\sfhD_u^{(1)}(T)$ and $\sfhD_u^{(2)}(T)$, respectively.
\end{example}

\begin{table}[t]
\renewcommand*\arraystretch{1.1}
\setlength{\tabcolsep}{9pt}
\centering
\begin{tabular}{c||c|c|c|c}
{\it Steps} & $k$ & $i_k$ & $i_{k+1}$ & $(q_{22}, q_{23}, q_{25}, q_{29})$  \\ \hline \hline
{\it Step 1} & 1 & $\cdot$ & $\cdot$ & $(2,4,3,1)$ \\ \hline
{\it Steps 2, 3} & 1 & 29 & 22 & $(1,4,3,2)$ \\ \hline
{\it Step 1} & 1 & $\cdot$ & $\cdot$ & $(1,4,3,2)$ \\ \hline
{\it Steps 2, 3} & 1 & 22 & 29 & $(1,4,3,2)$ \\ \hline
{\it Step 4} & 2 & $\cdot$ & $\cdot$ & $(1,4,3,2)$ \\ \hline
{\it Steps 2, 3} & 2 & 29 & 25 & $(1,4,3,2)$ \\ \hline
{\it Step 4} & 3 & $\cdot$ & $\cdot$ & $(1,4,3,2)$ \\ \hline
{\it Steps 2, 3} & 3 & 25 & 23 & $(1,3,4,2)$ \\ \hline
{\it Step 1} & 1 & $\cdot$ & $\cdot$ & $(1,3,4,2)$ \\ \hline
{\it Steps 2, 3} & 1 & 22 & 29 & $(1,3,4,2)$ \\ \hline
{\it Step 4} & 2 & $\cdot$ & $\cdot$ & $(1,3,4,2)$ \\ \hline
{\it Steps 2, 3} & 2 & 29 & 23 & $(1,3,4,2)$ \\ \hline
{\it Steps 4} & 3 & $\cdot$ & $\cdot$ & $(1,3,4,2)$ \\
\hline
{\it Steps 2, 3} & 3 & 23 & 25 & $(1,3,4,2)$ \\
\hline
{\it Steps 4, 5} & 3 & $\cdot$ & $\cdot$ & $(1,3,4,2)$
\end{tabular}
\medskip
\caption{The process of obtaining $\sfq_T(i)$'s in~\cref{Eg: sfp and sfD compute2}}
\label{Table: sfq compute}
\end{table}

\begin{figure}[t]
\begin{displaymath}
\def\pp {0.512}
\def\hp {5}
\def\sp {0.035}
\scalebox{0.75}{$
\begin{tikzpicture}
\node at (0 - \hp, 0) {$
\begin{ytableau}
\ast   & \ast  & \ast  & \ast  & ~  & ~  & ~  & ~ \\
\ast   & \ast  & \ast  & \ast  & ~  & ~  & ~      \\
\ast   & \ast  & \ast  & \ast  & ~  & ~           \\
\ast   & \ast  & \ast  & \color{red} \bullet   \\
\ast   & \color{red} \bullet  &  ~                         \\
\ast   & ~                               \\
\ast   & ~                               \\
~
\end{ytableau}
$};

\node at (0*\pp - \hp, 4*\pp + 0.5) {\color{red} \tiny $\hGam^{(22)}$};

\draw[line width = 0.5mm, color = red] 
(-4*\pp - \hp - \sp, -4*\pp) -- (-4*\pp - \hp - \sp, -3*\pp + \sp) -- (-3*\pp - \hp - \sp, -3*\pp + \sp) -- (-3*\pp - \hp - \sp, 0*\pp + \sp) -- (-1*\pp - \hp, 0*\pp + \sp) -- (-1*\pp - \hp, 1*\pp) -- (0*\pp - \hp - \sp, 1*\pp) -- (0*\pp - \hp - \sp, 4*\pp);

\node at (0*\pp - \hp - 1, 4*\pp + 0.5) {\color{violet} \tiny $\hGam^{(29)}$};

\draw[line width = 0.5mm, color = violet] (-4*\pp - \hp + \sp, -4*\pp - \sp) -- (-4*\pp - \hp + \sp, -3*\pp - \sp) -- (-3*\pp - \hp + \sp, -3*\pp - \sp) -- (-3*\pp - \hp + \sp, -2*\pp) -- (-2*\pp - \hp - \sp, -2*\pp) -- (-2*\pp - \hp - \sp, 4*\pp);

\node at (0*\pp - \hp + 2.3, 4*\pp + 0.5) {\color{blue} \tiny $\hGam^{(23)}$};

\draw[line width = 0.5mm, color = blue] 
(-2*\pp - \hp + \sp, -3*\pp) -- (-2*\pp - \hp + \sp, -1*\pp) -- (-2*\pp - \hp + \sp, 0*\pp - \sp) -- (0*\pp - \hp + \sp, 0*\pp - \sp) -- (0*\pp - \hp + \sp, 2*\pp - \sp) -- (2*\pp - \hp + \sp, 2*\pp - \sp) -- (2*\pp - \hp + \sp , 3*\pp - \sp) -- (3*\pp - \hp + \sp, 3*\pp - \sp) -- (3*\pp - \hp + \sp, 4*\pp - \sp) -- (4*\pp - \hp + \sp, 4*\pp - \sp);

\node at (0*\pp - \hp + 1.4, 4*\pp + 0.5) {\color{brown} \tiny $\hGam^{(25)}$};

\draw[line width = 0.5mm, color = brown] 
(1*\pp - \hp, 1*\pp) -- (1*\pp - \hp, 2*\pp + \sp) -- (2*\pp - \hp - \sp, 2*\pp + \sp) -- (2*\pp - \hp - \sp, 3*\pp + \sp) -- (3*\pp - \hp - \sp, 3*\pp + \sp) -- (3*\pp - \hp - \sp, 4*\pp);

\node[] at (-\hp, -3) {$\sfhD_{1}^{(1)}(T)$ and $\sfhD_{1}^{(2)}(T)$};

\node at (0, 0) {$
\begin{ytableau}
*(gray!70)   & *(gray!70)  & *(gray!70)  & *(gray!70)  & ~  & ~  & ~  & ~ \\
*(gray!70)   & *(gray!70)  & *(gray!70)  & *(gray!70)  & ~  & ~  & ~      \\
*(gray!70)   & *(gray!70)  & *(gray!70)  & *(gray!70)  & ~  & ~           \\
*(gray!70)   & *(gray!70)  & *(gray!70)  & *(gray!70)                     \\
*(gray!70)   & *(gray!70)  & ~                       \\
*(gray!70)   & \ast                               \\
*(gray!70)  & \color{violet} \bullet                               \\
\color{violet} \bullet
\end{ytableau}
$};


\draw[line width = 0.5mm, color = red] 
(-4*\pp - \sp, -4*\pp) -- (-4*\pp - \sp, -3*\pp + \sp) -- (-3*\pp - \sp, -3*\pp + \sp) -- (-3*\pp - \sp, 0*\pp + \sp) -- (-1*\pp, 0*\pp + \sp) -- (-1*\pp, 1*\pp) -- (0*\pp - \sp, 1*\pp) -- (0*\pp - \sp, 4*\pp);


\draw[line width = 0.5mm, color = violet] (-4*\pp + \sp, -4*\pp - \sp) -- (-4*\pp + \sp, -3*\pp - \sp) -- (-3*\pp + \sp, -3*\pp - \sp) -- (-3*\pp + \sp, -2*\pp) -- (-2*\pp - \sp, -2*\pp) -- (-2*\pp - \sp, 4*\pp);


\draw[line width = 0.5mm, color = blue] 
(-2*\pp+ \sp, -3*\pp) -- (-2*\pp + \sp, -1*\pp) -- (-2*\pp + \sp, 0*\pp - \sp) -- (0*\pp + \sp, 0*\pp - \sp) -- (0*\pp + \sp, 2*\pp - \sp) -- (2*\pp + \sp, 2*\pp - \sp) -- (2*\pp + \sp , 3*\pp - \sp) -- (3*\pp + \sp, 3*\pp - \sp) -- (3*\pp + \sp, 4*\pp - \sp) -- (4*\pp + \sp, 4*\pp - \sp);


\draw[line width = 0.5mm, color = brown] 
(1*\pp, 1*\pp) -- (1*\pp, 2*\pp + \sp) -- (2*\pp- \sp, 2*\pp + \sp) -- (2*\pp - \sp, 3*\pp + \sp) -- (3*\pp - \sp, 3*\pp + \sp) -- (3*\pp - \sp, 4*\pp);

\node[] at (0, -3) {$\sfhD_{2}^{(1)}(T)$ and $\sfhD_{2}^{(2)}(T)$};

\node at (0 + \hp, 0) {$
\begin{ytableau}
*(gray!70)   & *(gray!70)  & *(gray!70)  & *(gray!70)  & \ast  & \ast  & \ast  & \color{blue} \bullet \\
*(gray!70)   & *(gray!70)  & *(gray!70)  & *(gray!70)  & \ast  & \ast  & ~      \\
*(gray!70)   & *(gray!70)  & *(gray!70)  & *(gray!70)  & ~  & ~           \\
*(gray!70)   & *(gray!70)  & *(gray!70)  & *(gray!70)                     \\
*(gray!70)   & *(gray!70)  & \color{blue} \bullet                       \\
*(gray!70)   & *(gray!70)                               \\
*(gray!70)   & *(gray!70)                               \\
*(gray!70)
\end{ytableau}
$};


\draw[line width = 0.5mm, color = red] 
(-4*\pp + \hp - \sp, -4*\pp) -- (-4*\pp + \hp - \sp, -3*\pp + \sp) -- (-3*\pp + \hp - \sp, -3*\pp + \sp) -- (-3*\pp + \hp - \sp, 0*\pp + \sp) -- (-1*\pp + \hp, 0*\pp + \sp) -- (-1*\pp + \hp, 1*\pp) -- (0*\pp + \hp - \sp, 1*\pp) -- (0*\pp + \hp - \sp, 4*\pp);


\draw[line width = 0.5mm, color = violet] (-4*\pp + \hp + \sp, -4*\pp - \sp) -- (-4*\pp + \hp + \sp, -3*\pp - \sp) -- (-3*\pp + \hp + \sp, -3*\pp - \sp) -- (-3*\pp + \hp  + \sp, -2*\pp) -- (-2*\pp + \hp - \sp, -2*\pp) -- (-2*\pp + \hp - \sp, 4*\pp);


\draw[line width = 0.5mm, color = blue] 
(-2*\pp + \hp + \sp, -3*\pp) -- (-2*\pp + \hp + \sp, -1*\pp) -- (-2*\pp + \hp + \sp, 0*\pp - \sp) -- (0*\pp + \hp + \sp, 0*\pp - \sp) -- (0*\pp + \hp + \sp, 2*\pp - \sp) -- (2*\pp + \hp + \sp, 2*\pp - \sp) -- (2*\pp + \hp + \sp , 3*\pp - \sp) -- (3*\pp + \hp + \sp, 3*\pp - \sp) -- (3*\pp + \hp + \sp, 4*\pp - \sp) -- (4*\pp + \hp + \sp, 4*\pp - \sp);


\draw[line width = 0.5mm, color = brown] 
(1*\pp + \hp, 1*\pp) -- (1*\pp + \hp, 2*\pp + \sp) -- (2*\pp + \hp - \sp, 2*\pp + \sp) -- (2*\pp + \hp - \sp, 3*\pp + \sp) -- (3*\pp + \hp - \sp, 3*\pp + \sp) -- (3*\pp + \hp - \sp, 4*\pp);

\node[] at (\hp, -3) {$\sfhD_{3}^{(1)}(T)$ and $\sfhD_{3}^{(2)}(T)$};

\node at (0 + 2*\hp, 0) {$
\begin{ytableau}
*(gray!70)   & *(gray!70)  & *(gray!70)  & *(gray!70)  & *(gray!70)  & *(gray!70)  & *(gray!70)  & *(gray!70) \\
*(gray!70)   & *(gray!70)  & *(gray!70)  & *(gray!70)  & *(gray!70)  & *(gray!70)  & \color{brown} \bullet     \\
*(gray!70)   & *(gray!70)  & *(gray!70)  & *(gray!70)  & \ast  & \color{brown} \bullet           \\
*(gray!70)   & *(gray!70)  & *(gray!70)  & *(gray!70)                     \\
*(gray!70)   & *(gray!70)  & *(gray!70)                        \\
*(gray!70)   & *(gray!70)                               \\
*(gray!70)   & *(gray!70)                       \\
*(gray!70)
\end{ytableau}
$};


\draw[line width = 0.5mm, color = red] 
(-4*\pp + 2*\hp - \sp, -4*\pp) -- (-4*\pp + 2*\hp - \sp, -3*\pp + \sp) -- (-3*\pp + 2*\hp - \sp, -3*\pp + \sp) -- (-3*\pp + 2*\hp - \sp, 0*\pp + \sp) -- (-1*\pp + 2*\hp, 0*\pp + \sp) -- (-1*\pp + 2*\hp, 1*\pp) -- (0*\pp + 2*\hp - \sp, 1*\pp) -- (0*\pp + 2*\hp - \sp, 4*\pp);


\draw[line width = 0.5mm, color = violet] (-4*\pp + 2*\hp + \sp, -4*\pp - \sp) -- (-4*\pp + 2*\hp + \sp, -3*\pp - \sp) -- (-3*\pp + 2*\hp + \sp, -3*\pp - \sp) -- (-3*\pp + 2*\hp + \sp, -2*\pp) -- (-2*\pp + 2*\hp - \sp, -2*\pp) -- (-2*\pp + 2*\hp - \sp, 4*\pp);


\draw[line width = 0.5mm, color = blue] 
(-2*\pp + 2*\hp + \sp, -3*\pp) -- (-2*\pp + 2*\hp + \sp, -1*\pp) -- (-2*\pp + 2*\hp + \sp, 0*\pp - \sp) -- (0*\pp + 2*\hp + \sp, 0*\pp - \sp) -- (0*\pp + 2*\hp + \sp, 2*\pp - \sp) -- (2*\pp + 2*\hp + \sp, 2*\pp - \sp) -- (2*\pp + 2*\hp + \sp , 3*\pp - \sp) -- (3*\pp + 2*\hp + \sp, 3*\pp - \sp) -- (3*\pp + 2*\hp + \sp, 4*\pp - \sp) -- (4*\pp + 2*\hp + \sp, 4*\pp - \sp);


\draw[line width = 0.5mm, color = brown] 
(1*\pp + 2*\hp, 1*\pp) -- (1*\pp + 2*\hp, 2*\pp + \sp) -- (2*\pp + 2*\hp - \sp, 2*\pp + \sp) -- (2*\pp + 2*\hp - \sp, 3*\pp + \sp) -- (3*\pp + 2*\hp - \sp, 3*\pp + \sp) -- (3*\pp + 2*\hp - \sp, 4*\pp);

\node[] at (2*\hp, -3) {$\sfhD_{4}^{(1)}(T)$ and $\sfhD_{4}^{(2)}(T)$};
\end{tikzpicture}
$}
\end{displaymath}
\caption{$\sfhD_{u}^{(1)}(T)$ and $\sfhD_{u}^{(2)}(T)$ for $u = 1,2,3,4$ in~\cref{Eg: sfp and sfD compute2}}
\label{Fig: sfD compute2}
\end{figure}

Now, we construct the desired tableau $\sink(T)$ with the following algorithm.

\begin{algorithm}\label{Alg: construct sink of T}
Let $T \in \IGLTm{\lambda}{m}$.
Set $f_0 = 0$ and $N_0 = 0$.
Then, for $1 \le u \le |\calI(T)|$, let $f_u := |\sfhD_u^{(1)}(T)| + 1$ and $N_u = \sum_{v = 0}^{u} f_v$.
\begin{enumerate}[label = {\it Step \arabic*.}]
\item Set $v := 1$.
\item Fill the boxes in $\sfhD_v^{(1)}(T)$ by $N_{v-1} + 1, N_{v-1} + 2, \ldots, N_{v-1} + f_v - 1$ from top to bottom starting from the left.
\item Fill the boxes in $\sfhD_v^{(2)}(T)$ by $N_v$.
\item If $v < |\calI(T)|$, then set $v := v + 1$ and go to {\it Step 2}. 
Otherwise, fill the remaining boxes by $N_{|\calI(T)|} + 1, N_{|\calI(T)|} + 2,\ldots, m$ from top to bottom starting from the left.
Then, define $\sink(T)$ to be the resulting filling and terminate the algorithm.
\end{enumerate}
\end{algorithm}

\begin{example}
Revisit \cref{Eg: sfp and sfD compute2}.
We see that
\[
\def\pp {0.512}
\def\hp {12}
\def\sp {0.035}
\scalebox{0.75}{$
\begin{tikzpicture}

\node at (-3 - \hp,0) {\fontsize{15.9}{15.9} $T \ \  = \quad $};
\node at (-\hp, 0) {$
\begin{ytableau}
1  & 2  & 3 & 4 & 5 & 6 & 7 & \color{blue} 23 \\
8  & 9  & 10 & 11 & 12 & 13 & \color{brown} 25    \\
14  & 15  & 16 & 17 & 24 & \color{brown} 25         \\
18  & 19  & 20 & \color{red} 22                      \\
21  &\color{red} 22 & \color{blue} 23              \\
26 & 27                                \\
28 & \color{violet} 29                             \\
\color{violet} 29                                     
\end{ytableau}
$};

\draw[line width = 0.5mm, color = red] 
(-4*\pp - \hp - \sp, -4*\pp) -- (-4*\pp - \hp - \sp, -3*\pp + \sp) -- (-3*\pp - \hp - \sp, -3*\pp + \sp) -- (-3*\pp - \hp - \sp, 0*\pp + \sp) -- (-1*\pp - \hp, 0*\pp + \sp) -- (-1*\pp - \hp, 1*\pp) -- (0*\pp - \hp - \sp, 1*\pp) -- (0*\pp - \hp - \sp, 4*\pp);

\draw[line width = 0.5mm, color = violet] (-4*\pp - \hp + \sp, -4*\pp - \sp) -- (-4*\pp - \hp + \sp, -3*\pp - \sp) -- (-3*\pp - \hp + \sp, -3*\pp - \sp) -- (-3*\pp - \hp + \sp, -2*\pp) -- (-2*\pp - \hp - \sp, -2*\pp) -- (-2*\pp - \hp - \sp, 4*\pp);

\draw[line width = 0.5mm, color = blue] 
(-2*\pp - \hp + \sp, -3*\pp) -- (-2*\pp - \hp + \sp, -1*\pp) -- (-2*\pp - \hp + \sp, 0*\pp - \sp) -- (0*\pp - \hp + \sp, 0*\pp - \sp) -- (0*\pp - \hp + \sp, 2*\pp - \sp) -- (2*\pp - \hp + \sp, 2*\pp - \sp) -- (2*\pp - \hp + \sp , 3*\pp - \sp) -- (3*\pp - \hp + \sp, 3*\pp - \sp) -- (3*\pp - \hp + \sp, 4*\pp - \sp) -- (4*\pp - \hp + \sp, 4*\pp - \sp);

\draw[line width = 0.5mm, color = brown] 
(1*\pp - \hp, 1*\pp) -- (1*\pp - \hp, 2*\pp + \sp) -- (2*\pp - \hp - \sp, 2*\pp + \sp) -- (2*\pp - \hp - \sp, 3*\pp + \sp) -- (3*\pp - \hp - \sp, 3*\pp + \sp) -- (3*\pp - \hp - \sp, 4*\pp);

\node at (-7.5,0.5) {\small \cref{Alg: construct sink of T}};
\draw[->] (-8.5,0) -- (-6.5,0);

\node at (-3.85,0) {\fontsize{15.9}{15.9} $\sink(T) \  = \quad $};
\node at (2.5,0) {.};
\node at (0, 0) {$
\begin{ytableau}
1  & 8  & 12 & 16 & 22 & 24 & 26 & \color{blue} 27 \\
2  & 9  & 13 & 17 & 23 & 25 & \color{brown} 29    \\
3  & 10  & 14 & 18 & 28 & \color{brown} 29         \\
4  & 11  & 15 & \color{red} 19                      \\
5  &\color{red} 19 & \color{blue} 27              \\
6 & 20                                \\
7 & \color{violet} 21                             \\
\color{violet} 21                                     
\end{ytableau}
$};

\draw[line width = 0.5mm, color = red] 
(-4*\pp - \sp, -4*\pp) -- (-4*\pp - \sp, -3*\pp + \sp) -- (-3*\pp - \sp, -3*\pp + \sp) -- (-3*\pp - \sp, 0*\pp + \sp) -- (-1*\pp, 0*\pp + \sp) -- (-1*\pp, 1*\pp) -- (0*\pp - \sp, 1*\pp) -- (0*\pp - \sp, 4*\pp);

\draw[line width = 0.5mm, color = violet] (-4*\pp + \sp, -4*\pp - \sp) -- (-4*\pp + \sp, -3*\pp - \sp) -- (-3*\pp + \sp, -3*\pp - \sp) -- (-3*\pp + \sp, -2*\pp) -- (-2*\pp - \sp, -2*\pp) -- (-2*\pp - \sp, 4*\pp);

\draw[line width = 0.5mm, color = blue] 
(-2*\pp + \sp, -3*\pp) -- (-2*\pp + \sp, -1*\pp) -- (-2*\pp + \sp, 0*\pp - \sp) -- (0*\pp + \sp, 0*\pp - \sp) -- (0*\pp + \sp, 2*\pp - \sp) -- (2*\pp + \sp, 2*\pp - \sp) -- (2*\pp + \sp , 3*\pp - \sp) -- (3*\pp + \sp, 3*\pp - \sp) -- (3*\pp + \sp, 4*\pp - \sp) -- (4*\pp + \sp, 4*\pp - \sp);

\draw[line width = 0.5mm, color = brown] 
(1*\pp, 1*\pp) -- (1*\pp, 2*\pp + \sp) -- (2*\pp - \sp, 2*\pp + \sp) -- (2*\pp - \sp, 3*\pp + \sp) -- (3*\pp - \sp, 3*\pp + \sp) -- (3*\pp - \sp, 4*\pp);
\end{tikzpicture}$}
\]
\end{example}

Let us collect some useful facts for $\sink(T)$ which can be easily seen.
\begin{enumerate}[label = {\bf S\arabic*$\boldsymbol{'}$.}]
\item By \cref{Lem: source and sink cond}(2), for any $T \in \IGLTm{\lambda}{m}$, $\sink(T)$ is a sink tableau.
\item For any $T \in \IGLTm{\lambda}{m}$, $T \sim \sink(T)$ by the construction of $\sink(T)$.
\item By the construction of $\sink(T)$, the set $\left\{ \left(\Gamma_i(T), T^{-1}(i) \right) \; \middle| \; i \in \calI(T) \right \}$ determines $\sink(T)$. 
In other words, if $T_1 \sim T_2$, then $\sink(T_1) = \sink(T_2)$.
\end{enumerate}

Combining the facts {\bf S1$\boldsymbol{'}$} and {\bf S2$\boldsymbol{'}$} shows the existence of sink tableaux in $E$.
However, the above facts {\bf S1$\boldsymbol{'}$}, {\bf S2$\boldsymbol{'}$}, and {\bf S3$\boldsymbol{'}$} do not guarantee the uniqueness of sink tableaux in $E$. 
To show the uniqueness, we need the lemma below.

\begin{lemma}\label{lem: tS(T)2}
For any sink tableau $T$, $\sink(T) = T$.
\end{lemma}

\begin{proof}
Let $j_{0}=0$ and $
\calI(T) = \{j_1 < j_2 < \cdots < j_{|\calI(T)|}\}$. By \cref{lem: sink order}, $\sfq_T(j_u) = u$ for all $1 \le u \le |\calI(T)|$, thus from the definition of $\sfhD_u^{(2)}(T)$ we have that
\begin{align}\label{Eq: D(2)uT2}
\sfhD^{(2)}_u(T) = T^{-1}(j_u)  \quad \text{for all $1 \le u \le |\calI(T)|$.}
\end{align}

We claim that 
\begin{align}\label{eq: Du1 claim2}
\sfhD^{(1)}_{u}(T) = 
T^{-1}([j_{u-1} + 1, j_{u}-1]) \quad \text{for all $1 \le u \le |\calI(T)|$.}
\end{align}
First, let us prove the inclusion
$\sfhD^{(1)}_{u}(T) \supseteq
T^{-1}([j_{u-1} + 1, j_{u}-1])$ for all $1 \le u \le |\calI(T)|$.
Take any $1 \le u \le |\calI(T)|$ and $i \in [j_{u-1} + 1, j_{u}-1]$.
Let $B$ be the box filled with $i$ in $T$.
Recall that for any $1 \le v \le |\calI(T)|$, $\sfhA_{v}$ is defined to be the subdiagram of $\tyd(\lambda)$ consisting of the boxes located left of $\hGam^{(v)}$.
By the definition of $\sfhD^{(1)}_{u}(T)$, the desired inclusion is obtained by proving that 
\[
B \in \sfhA_u
\quad \text{and} \quad 
B \notin \sfhA_v \quad \text{for all $1 \le v < u$.}
\]

Suppose for the sake of contradiction that $B \notin \sfhA_u$.
Combining the fact that $T$ is an increasing tableau with the inequality $i < j_u$, we have that $B$ is strictly above $\Top_{j_{u}}(T)$.
This implies that $B$ is strictly right of $\Top_{j_{u}}(T)$.
By \cref{Lem: source and sink cond}(2), $i+1$ is weakly right of $i$.
Again, by \cref{Lem: source and sink cond}(2), $i+2$ appears weakly right of $i+1$, so $i+2$ is weakly right of $i$.
Continuing this process, we see that $j_{u}$ appears weakly right of $i$, that is, $\Top_{j_{u}}(T)$ is weakly right of $B$.
This contradicts the above observation that $B$ is strictly right of $\Top_{j_{u}}(T)$.
Thus, $B \in \sfhA_u$.

Suppose for the sake of contradiction that $B \in \sfhA_v$ for some $1 \le v < u$.
Since $i > j_v$, $B$ cannot be placed weakly above of $\Bot_{j_v}(T)$ while being left of $\hGam_v$.
Therefore, $B$ is strictly below of $\Bot_{j_v}(T)$.
In addition, by \cref{lem: sink order}, $\sfq_T(j_v) = v \le u-1 = \sfq_T(j_{u-1})$, so the boxes strictly below of $\Bot_{j_v}(T)$ while being left of $\hGam_v$ are placed left of $\hGam^{(u-1)}$.
Therefore, if we prove $B$ is right of $\hGam^{(u-1)}$, then we obtain a contradiction to the assumption that $B \in \sfhA_v$. 
Suppose that $B$ is left of $\hGam^{(u-1)}$.
Since $i>j_{u-1}$, $B$ is strictly left of $\Bot_{j_{u-1}}(T)$.
On the other hand, by using \cref{Lem: source and sink cond}(2) repeatedly, one can see that there exists at least one $j_{u-1}$ which appears weakly left of $i$.
It follows that $\Bot_{j_{u-1}}(T)$ is weakly left of $B$.
This gives a contradiction to the previous observation, thus $B \notin \sfhA_v$.

Next, let us prove the inclusion
$\sfhD^{(1)}_{u}(T) \subseteq
T^{-1}([j_{u-1} + 1, j_{u}-1])$ for all $1 \le u \le |\calI(T)|$.
Equivalently, we claim that
\[
\sfhD^{(1)}_{u}(T) \cap T^{-1}(i) = \emptyset \quad \text{for any $i \in [1,m] \setminus [j_{u-1} + 1, j_{u}-1]$.}
\]
To prove this, we let $1 \le u \le |\calI(T)|$ and $i \in [1,m] \setminus [j_{u-1} + 1, j_{u}-1]$.
Since $\bigcup_{1\le v \le |\calI(T)|} \sfhD^{(2)}_v(T) = T^{-1}(\calI(T))$ and $\sfhD^{(1)}_{u}(T) \cap \sfhD^{(2)}_{v}(T) = \emptyset$ for all $1 \le v \le |\calI(T)|$, $\sfhD^{(1)}_{u}(T) \cap T^{-1}(i) = \emptyset$ if $i \in \calI(T)$.
Therefore, we may assume that $i \notin \calI(T)$.
Take $u_0 \in \{1,2,\ldots, |\calI(T)|\} \setminus \{u\}$ such that $i \in [j_{u_0-1} +1, j_{u_0} - 1]$.
Since the method of proof for the case where $u_0 > u$ is similar with that for the case where $u_0 < u$, we only prove the latter case. 

Suppose that $u_0 < u$.
Let $B$ be the box filled with $i$.
If $i$ appears right of $\hGam_{j_{u_0}}$, then $B$ is strictly right of $\Top_{j_{u_0}}(T)$ since $i < j_{u_0}$.
On the other hand, for all $i \le j < j_{u_0}$, combining the fact $j \notin \calI(T)$ with \cref{Lem: source and sink cond}(2) yields that
$\Top_{j+1}(T)$ is placed weakly right of the unique box filled with $j$.
It follows that $\Top_{j_{u_0}}(T)$ is weakly right of $B$.
This contradicts the previous observation that $B$ is strictly right of $\Top_{j_{u_0}}(T)$.
Therefore, $i$ must appear left of $\hGam_{j_{u_0}}$.
It follows that $B \in \sfhA_{u_0}$.
Thus, $\sfhD^{(1)}_{u}(T) \cap T^{-1}(i) = \emptyset$ by \cref{Eq: sfhD(i)}.

Now, combining \cref{Lem: source and sink cond}(2) with the equations \cref{Eq: D(2)uT2} and \cref{eq: Du1 claim2}, we have that $\sfhD^{(1)}_{u}(T)$ is filled with $j_{u-1} + 1, j_{u-1} +2, \ldots, j_{u}-1$ from top to bottom starting from the left for each $1 \le u \le |\calI(T)|$.
Note that 
$$
T\Big(\tyd(\lambda) \setminus \bigcup_{1 \le u \le |\calI(T)|} \big( \sfhD^{(1)}_{u}(T) \cup \sfhD^{(2)}_{u}(T) \big) \Big) = \{j_{|\calI(T)|} + 1, j_{|\calI(T)|} + 2, \ldots, m\}
$$
and that $\tyd(\lambda) \setminus \bigcup_{1 \le u \le |\calI(T)|} (\sfhD^{(1)}_{u}(T) \cup \sfhD^{(2)}_{u}(T))$ is filled with $j_{|\calI(T)|} + 1, j_{|\calI(T)|} + 2, \ldots, m$ from top to bottom starting from the left.
Hence, by the construction of $\sink(T)$, we have that $T = \sink(T)$.
\end{proof}

Now, we prove the main theorem of this subsection.

\begin{theorem}\label{thm: uniqueness of sink}
For each $E \in \calE_{\lambda;m}$, there exists a unique sink tableau in $E$.
\end{theorem}

\begin{proof}
Recall that the existence is already shown by using {\bf S1$\boldsymbol{'}$} and {\bf S2$\boldsymbol{'}$}.
For the uniqueness, suppose that $T_1$ and $T_2$ are sink tableaux contained in $E$.
Combining {\bf S3$\boldsymbol{'}$} with \cref{lem: tS(T)2} yields that
\[
T_1 = \sink(T_1) = \sink(T_2) = T_2.
\]
Hence, the source tableau in $E$ is unique.
\end{proof}

For each $E \in \calE_{\lambda;m}$, define
$$
T'_E := \text{the unique sink tableau contained in $E$.}
$$

\section{A weak Bruhat interval module description of $\bfG_{E}$}
\label{Sec: weak bruhat interval module}

The purpose of this section is to prove that the $H_m(0)$-module $\bfG_E$ is equipped with the structure of weak Bruhat interval module.
In \cref{Subsec: poset structure}, we define a relation $\preceq_{E}$ on $E$ and assign a permutation  $\sfread(T) \in \SG_m$ to each $T \in E$.
Then, we show that $(E,\preceq_{E})$ is a poset which is isomorphic to $([\sfread(T_E), \sfread(T'_E)]_L, \preceq_L)$.
In \cref{Subsec: WBIM module}, by extending this isomorphism, we prove that $\bfG_E$ is isomorphic to $\sfB(\sfread(T_E), \sfread(T'_E))$ as $H_{m}(0)$-modules.

Hereafter, we let $\Des(T_E) = \{d_1< d_2 < \cdots < d_k\}$, $d_0 := 0$, and $d_{k+1} := m$.

\subsection{A poset structure on $E$}
\label{Subsec: poset structure}

To begin with, we introduce necessary notation and definitions.
Define a relation $\preceq_{E}$ on $E$ by 
\[
T_1  \preceq_E  T_2 \quad \text{if $\pi_\sigma \cdot T_1 = T_2$ for some $\sigma \in \SG_m$}.
\]
For each $1 \le j \le k+1$, let 
$$
\tH_j := T_E^{-1}([d_{j-1}+1, d_{j}]).
$$
By considering \cref{Lem: source and sink cond}, one can easily show the following properties:
\begin{enumerate}[label = {\rm (P\arabic*)}]
\item 
For each $1 \le j \le k+1$, the set $\tH_j$ is a horizontal strip, a set of boxes which contains at most one box in each column of $\tyd(\lambda)$.
\item
For any $B \in \tH_{j}$, if $B' \in \tH_{j}$ appears to the right of $B$, then $B'$ is placed weakly above $B$.
\item 
For any $B \in \tH_{j}$, if $B' \in \tH_{j}$ is the leftmost box among the boxes in $\tH_{j}$ placed strictly right of $B$, then $T_E(B) = T_E(B')$ or $B' = B + (0,1)$, that is, $B'$ is placed immediately right of $B$.
\end{enumerate}

\begin{example}\label{Eg: tH_j}
When
\[
T_{E} = \begin{array}{l}
\begin{ytableau}
1 & 2 & 3 & 4 & 5 & 11 & 12 & 14 & 15 \\
6 & 7 & 8 & 10 & 11 & 14 \\
9 & 10 \\
13 & 14
\end{ytableau}
\end{array},
\]
we have that $\Des(T_{E}) = \{5,8,12\}$. In this case, $\tH_j~(1 \le j \le 4)$ are given as follows: 
\[
\begin{array}{c}
\begin{ytableau}
*(red!20) \tH_1 & *(red!20) \tH_1 & *(red!20) \tH_1 & *(red!20) \tH_1 & *(red!20) \tH_1 & *(gray!20) \tH_3 & *(gray!20) \tH_3  & *(magenta!20) \tH_4 & *(magenta!20) \tH_4\\
*(blue!20) \tH_2 & *(blue!20) \tH_2 & *(blue!20) \tH_2 & *(gray!20) \tH_3 & *(gray!20) \tH_3 & *(magenta!20) \tH_4 \\
*(gray!20) \tH_3 & *(gray!20) \tH_3 \\
*(magenta!20) \tH_4 & *(magenta!20) \tH_4
\end{ytableau}
\end{array}.
\]
\end{example}

For each $1\le j \le k+1$, let $\rmw^{(j)}(T)$ be the word obtained by reading the entries of $T$ contained in $\tH_j$ from right to left. 
Note that if an integer $i$ appears multiple times in $\rmw^{(j)}(T)$, then the integer $i$'s are placed consecutively.
We define $\overline{\rmw}^{(j)}(T)$ as the word obtained from $\rmw^{(j)}(T)$ by erasing all $i$'s except one $i$ for each $i$ that appears in $\rmw^{(j)}(T)$.

\begin{definition}\label{def: standard reading}
For $T \in E$, \emph{the standardized reading word $\sfread(T)$ of $T$} is defined to be the word $\overline{\rmw}^{(1)}(T) \overline{\rmw}^{(2)}(T) \cdots \overline{\rmw}^{(k+1)}(T)$ obtained by concatenating $\overline{\rmw}^{(j)}(T)$ for $1 \le j \le k+1$.
\end{definition}

Hereafter, we will identify $\sfread(T)$ with the permutation in $\SG_m$ written in one-line notation.

\begin{example}\label{eg: reading}
We revisit \cref{Eg: tH_j}. 
For each $j=1,2,3,4$, $\rmw^{(j)}(T_{E})$ and $\overline{\rmw}^{(j)}(T_{E})$ are obtained as in \cref{tab: reading}. 
Therefore,
\[
\sfread(T_{E}) = 5 \ 4 \ 3 \ 2 \ 1 \ 8 \ 7 \ 6 \ 12 \ 11 \ 10 \ 9 \ 15 \ 14 \ 13 \in \SG_{15}. 
\]
\end{example}

\begin{table}[t]
$\begin{array}{c|c|c|c|c}
j & 1 & 2 & 3 & 4 \\ \hline
\rmw^{(j)}(T_{E}) 
& 5 \ 4 \ 3 \ 2 \ 1 & 8 \ 7 \ 6 & 12 \ 11 \ 11 \ 10 \ 10 \ 9 & 15 \ 14 \ 14 \ 14 \ 13 \\ \hline
\overline{\rmw}^{(j)}(T_{E})
& 5 \ 4 \ 3 \ 2 \ 1 & 8 \ 7 \ 6 & 12 \ 11 \ 10 \ 9 & 15 \ 14 \ 13 
\end{array}
$
\medskip
\caption{$\rmw^{(j)}(T_{E})$ and $\overline{\rmw}^{(j)}(T_{E})$ in \cref{eg: reading}}
\label{tab: reading}
\end{table}

\begin{lemma}\label{lem: read and order}
Let $T_1, T_2 \in E$.
\begin{enumerate}[label = {\rm (\arabic*)}]
\item Suppose that $i$ is a non-attacking descent of $T_1$.
Then, 
$$
\sfread(\pi_i \cdot T_1) = s_i \ \sfread(T_1)
\quad \text{and} \quad
\sfread(T_1) \prec_L \sfread(\pi_i \cdot T_1).
$$
\item 
If $T_1 \preceq_E  T_2$, then $\sfread(T_1) \preceq_L \sfread(T_2)$ and 
\[
\pi_{\sfread(T_2)\sfread(T_1)^{-1}} \cdot T_1 =
s_{i_p} \cdot \  \cdots \ \cdot s_{i_2} \cdot s_{i_1} \cdot T_1 = T_2,
\]
where $s_{i_p} \cdots s_{i_2} s_{i_1}$ is a reduced expression for $\sfread(T_2)\sfread(T_1)^{-1}$.
\end{enumerate}
\end{lemma}

\begin{proof}
(1) By \cref{def: standard reading}, $s_i \ \sfread(T_1) = \sfread(\pi_i \cdot T_1)$ is obvious.
We will prove that $\sfread(T_1) \prec_L \sfread(\pi_i \cdot T_1)$.

Let $j^{(i)}$ and $j^{(i+1)}$ be the unique indices in $\{1,2,\ldots, k+1\}$ such that $i \in T_1(\tH_{j^{(i)}})$ and $i+1 \in T_1(\tH_{j^{(i+1)}})$, respectively.
We claim that 
\begin{align}\label{eq: claim}
j^{(i)} < j^{(i+1)}.
\end{align}
Since $i$ is a non-attacking descent of $T_1$, we have
\begin{align}\label{eq: r i and r i+1}
r_{\sfb}^{(i)}(T_1) < r_{\sft}^{(i+1)}(T_1).
\end{align}

If $j^{(i)} = j^{(i+1)}$, then the properties (P1)--(P3) written in the first paragraph of this subsection imply that $r_{\sft}^{(i)}(T_1) \ge r_{\sfb}^{(i+1)}(T_1)$.
This contradicts \cref{eq: r i and r i+1}, thus $j^{(i)} \neq j^{(i+1)}$.

Suppose that $j^{(i)} > j^{(i+1)}$.
By \cref{eq: r i and r i+1},
every box in $T_1^{-1}(i+1)$ is placed strictly below $\Bot_i(T_1)$.
In addition, if there exist 
$B \in T_1^{-1}(i+1)$ such that $B$ is placed weakly below and weakly right of $\Bot_i(T_1)$, then $T_E(\Bot_i(T_1)) < T_E(B)$ which implies $j^{(i)} \le j^{(i+1)}$.
This contradicts the assumption $j^{(i)} > j^{(i+1)}$, so  every box in $T_1^{-1}(i+1)$ is placed strictly below and strictly left of $\Bot_i(T_1)$.

Assume that there exist $B_1,B_2 \in \tH_{j^{(i+1)}}$ such that $T_1(B_1) = T_1(B_2)$ and $\row(B_1) \le \row(\Bot_i(T_1)) \le \row (B_2)$.
Since every box in $T_1^{-1}(i+1)$ is placed strictly below and strictly left of $\Bot_i(T_1)$, we have $T_1(B_1) > i+1$.
It follows that $\Gamma_{T_1(B_1)}(T_1)$ passes below 
$\Bot_i(T_1)$.
However, since $j^{(i)} > j^{(i+1)}$, $x > y$ for all $x \in T_E(\tH_{j^{(i)}})$ and $y \in T_E(\tH_{j^{(i+1)}})$, so
$\Gamma_{T_E(B_1)}(T_E)$ passes above $\Bot_i(T_1)$.
This implies that $T_1 \not\sim T_E$, which gives a contradiction.
Thus, we have 
\[
\row(\Bot_i(T_1))< \row(B) \quad \text{for all $B \in \tH_{j^{(i+1)}}$.}
\]

Let $B_L^{(i)}$ be the leftmost box in $\tH_{j^{(i)}}$ and $B_R^{(i+1)}$ the  rightmost box in $\tH_{j^{(i+1)}}$.
Suppose that there exist $B_1, B_2 \in \tH_{j^{(i)}}$ such that $T_1(B_1) = T_1(B_2)$ and $\row(B_1) \le \row(B_R^{(i+1)}) \le \row(B_2)$.
By the above observation, $B_1$ is strictly left of every box in $T_1^{-1}(i)$, so $T_1(B_1) < i < i+1 \le T_1(B_R^{(i+1)})$.
It follows that $\Gamma_{T_1(B_1)}(T_1)$ passes above $B_R^{(i+1)}$.
However, since $j^{(i)} > j^{(i+1)}$, we have $T_E(B_1) > T_E(B_R^{(i+1)})$ and so $\Gamma_{T_E(B_1)}(T_E)$ passes below $B_R^{(i+1)}$.
This implies that $T_1 \not\sim T_E$, which gives a contradiction.
Thus, we have
\begin{align}\label{eq: row of BL BR}
\row(B_L^{(i)})< \row(B_R^{(i+1)}).
\end{align}
In addition, if $\col(B_L^{(i)}) \le \col(B_R^{(i+1)})$, then $T_E(B_L^{(i)}) < T_E(B_R^{(i+1)})$, which implies that $j^{(i)} < j^{(i+1)}$.
Thus, we have 
\begin{align}\label{eq: col of BL BR}
\col(B_R^{(i+1)}) < \col(B_L^{(i)}).
\end{align}

By considering the construction of $\source(T_1)$ in \cref{subsec: source tableau} together with the inequalities $j^{(i+1)} < j^{(i)}$, \cref{eq: row of BL BR}, and \cref{eq: col of BL BR}, we deduce that there exists $p \in \calI(T_E)$ such that $\Gamma_p(T_E)$ passes right of $B_R^{(i+1)}$ and left of $B_L^{(i)}$.
Since $T_1 \in E$, this implies that 
\[
i+1 \le T_1(B_R^{(i+1)}) < T_1(T_E^{-1}(p)) < T_1(B_L^{(i)}) \le i
\]
which gives a contradiction.
Therefore, we obtain \cref{eq: claim}.

Now, combining the definition of $\sfread$ with the claim above yields that $i$ is placed to left of $i+1$ in $\sfread(T_1)$.
Thus, $\sfread(T_1) \prec_L s_i \, \sfread(T_1) = \sfread(\pi_i \cdot T_1)$.

(2) 
Let $\sigma \in \SG_m$ satisfying $\pi_\sigma \cdot T_1 = T_2$.
Take a reduced expression $s_{j_q} \cdots s_{j_2} s_{j_1}$ for $\sigma$.
Note that for all $1 \le r \le q$, $j_r$ cannot be an attacking descent of $\pi_{j_{r-1}} \cdots \pi_{j_{2}} \pi_{j_{1}} \cdot T_1$, otherwise $\pi_\sigma \cdot T_1 = 0$.
That is, $j_r$ is a non-descent or non-attacking descent of $\pi_{j_{r-1}} \cdots \pi_{j_{2}} \pi_{j_{1}} \cdot T_1$ for all $1 \le r \le q$.
Enumerate the indices $r \in [1,q]$ such that $j_r$ is a non-attacking descent of $\pi_{j_{r-1}} \cdots \pi_{j_{2}} \pi_{j_{1}} \cdot T_1$ by
\[
a_1 < a_2 < \cdots < a_p
\]
in increasing order.
By (1), for all $1 \le r \le p$,
\begin{align*}
\sfread(\pi_{a_{r-1}} \cdots \pi_{a_2} \pi_{a_1} \cdot T_1)  
\preceq_L
s_{a_r} \sfread(\pi_{a_{r-1}} \cdots \pi_{a_2} \pi_{a_1} \cdot T_1) 
= \sfread(\pi_{a_{r}} \cdots \pi_{a_2} \pi_{a_1} \cdot T_1).
\end{align*}
This implies that $\sfread(T_1) \preceq_L \sfread(T_2)$ and $s_{a_p} \cdots s_{a_2} s_{a_1}$ is a reduced expression of $\sfread(T_2) \sfread(T_1)^{-1}$. 
By the definition of $a_r$'s, we have that $\pi_{a_p} \cdots \pi_{a_2} \pi_{a_1} \cdot T_1 = \pi_\sigma \cdot T_1$, and hence
\begin{equation*}
\pi_{\sfread(T_2) \sfread(T_1)^{-1}} \cdot T_1 =
\pi_{a_p} \cdots \pi_{a_2} \pi_{a_1} \cdot T_1 = \pi_\sigma \cdot T_1 = T_2.
\qedhere
\end{equation*}
\end{proof}

\begin{theorem}\label{thm: poset isomorphism}
Let $E \in \calE_{\lambda;m}$.
\begin{enumerate}[label = {\rm (\arabic*)}]
\item For any $T \in E$, $T_E \preceq_E T \preceq_E T'_E$.
\item 
$(E,\preceq_{E})$ is a poset which is isomorphic to $([\sfread(T_{E}), \sfread (T'_{E})]_{L},\preceq_{L})$.
\end{enumerate}
\end{theorem}

\begin{proof}
(1)
Let $T \in E$.
Let us prove $T_E \preceq_E T$.
If $T$ is a source tableau, then $T = T_E$ by \cref{thm: uniqueness of source}, thus $T_E \preceq_E T$.
Otherwise, by \cref{Lem: source and sink cond}(1), 
we can choose $i_1 \notin \Des(T)$ such that $\Top_{i_1}(T)$ is strictly below $\Bot_{i_1+1}(T)$.
One can easily see that $i_1$ is non-attacking descent of $s_{i_1} \cdot T$ and so $\pi_{i_1} \cdot (s_{i_1} \cdot T) = T$.
By \cref{thm: preserving}, we have $s_{i_1} \cdot T \in E$. In addition,
$\sfread(s_{i_1} \cdot T) \prec_L \sfread(T)$ by \cref{lem: read and order}(1).
If $s_{i_1} \cdot T$ is a source tableau, then $s_{i_1} \cdot T = T_E$ by \cref{thm: uniqueness of source}, thus $T_E \preceq_E T$.
Otherwise, following a similar procedure as described above, we can choose $i_2 \notin \Des(s_{i_1} \cdot T)$ such that 
\[
\pi_{i_2} \cdot (s_{i_2} s_{i_1} \cdot T) = s_{i_1} \cdot T, \quad s_{i_1} \cdot T \in E, \quad \text{and} \quad \sfread(s_{i_2} s_{i_1} \cdot T) \prec_L \sfread(s_{i_1} \cdot T).
\]
Continuing this process, we have $s_{i_p} \cdots  s_{i_2} s_{i_1} \cdot T = T_E$ for some $p \in \Z_{\ge 0}$ since $E$ is a finite set and the inequalities 
$\sfread(s_{i_p} \cdots s_{i_2} s_{i_1} \cdot T) \prec_L \cdots \prec_L \sfread(s_{i_1} \cdot T) \prec_L \sfread(T)$ ensure that the tableaux $T, s_{i_1} \cdot T, \ldots,  s_{i_p} \cdots s_{i_2} s_{i_1} \cdot T$ are all distinct.
By the choice of $i_1,i_2, \ldots, i_p$, we have $(\pi_{i_1} \pi_{i_2} \cdots \pi_{i_k}) \cdot T_E = T$, thus $T_E \preceq_E T$.

Next, let us prove that $T \preceq_E T'_E$.
If $T$ is a sink tableau, then $T = T'_E$ by \cref{thm: uniqueness of sink}.
Otherwise, by \cref{Lem: source and sink cond}(2), we can choose a non-attacking descent $i_1$ of $T$.
By \cref{lem: read and order}(1), $\sfread(T) \prec_L \sfread(\pi_{i_1} \cdot T)$.
If $\pi_{i_1} \cdot T$ is a sink tableau, then $T \preceq_E T'_E$.
Otherwise, by \cref{Lem: source and sink cond}(2), we can choose a non-attacking descent $i_2$ of $\pi_{i_1} \cdot T$.
Again, by \cref{lem: read and order}(1), $\sfread(\pi_{i_1} \cdot T) \prec_L \sfread(\pi_{i_2} \pi_{i_1} \cdot T)$.
Continuing this process, we have $\pi_{i_{p'}} \cdots  \pi_{i_2} \pi_{i_1} \cdot T = T'_E$ for some $p' \in \Z_{\ge 0}$ since $E$ is a finite set and the inequalities 
$\sfread(T) \prec_L \sfread(\pi_{i_1} \cdot T) \prec_L \cdots \prec_L
\sfread(s_{i_p} \cdots s_{i_2} s_{i_1} \cdot T)$ ensure that the tableaux $T, \pi_{i_1} \cdot T, \ldots,  \pi_{i_p} \cdots \pi_{i_2} \pi_{i_1} \cdot T$ are all distinct.
Therefore, $T \preceq_E T'_E$.

(2) Let $f: E \ra [\sfread(T_{E}), \sfread (T'_{E})]_{L}$ be a map sending $T$ to $\sfread(T)$.
Combining \cref{lem: read and order}(2) with \cref{thm: poset isomorphism}(1) yields that $\sfread(T) \in [\sfread(T_{E}), \sfread (T'_{E})]_{L}$ for all $T \in E$.
Thus, $f$ is well defined.
The injectivity of $f$ immediately follows from the definition of $\sfread$.

Let us prove the surjectivity of $f$.
Take any $\gamma \in [\sfread(T_{E}), \sfread (T'_{E})]_{L}$
and let 
$$
T = \pi_{\gamma \, \sfread (T_{E})^{-1}} \cdot T_E.
$$
We claim that $\gamma = \sfread(T)$.
Let us prove our claim by using mathematical induction on $\ell(\gamma \, \sfread(T_{E})^{-1})$.
When $\ell(\gamma \, \sfread(T_{E})^{-1}) = 0$, we have $\gamma = \sfread(T_{E})$ and $T = T_E$.
Thus, $\gamma = \sfread(T)$.
Suppose that $\ell(\gamma \, \sfread(T_{E})^{-1}) > 0$ and the claim holds for all $\omega \in [\sfread(T_{E}), \sfread (T'_{E})]_{L}$ with $\ell(\omega) < \ell(\gamma)$.
Take any reduced expression $s_{i_p} \cdots s_{i_2} s_{i_1}$ for $\gamma \, \sfread(T_{E})^{-1}$.
Let 
$$
\gamma' = s_{i_p} \gamma 
\quad \text{and} \quad 
T' = \pi_{\gamma' \, \sfread(T_{E})^{-1}} \cdot T_E.
$$
One can easily see that $\gamma' \in [\sfread(T_{E}), \sfread (T'_{E})]_{L}$ and $\ell(\gamma') < \ell(\gamma)$.
By the induction hypothesis, we have $\sfread(T') = \gamma'$.
Since $\sfread(T') = \gamma' \neq \gamma = \sfread(T)$, we have $T \neq T'$.
In addition, since $\pi_{i_p} \cdot T' = T$, the integer $i_p$ is a non-attacking descent of $T'$.
Thus,
\[
\sfread(T) 
= s_{i_p} \, \sfread(T') 
= s_{i_p} \gamma' 
= \gamma.
\]
Here, the first equality follows from \cref{lem: read and order}(1). 

For the remaining part of the proof, we claim that for any $T_1, T_2 \in E$,
\[
T_1 \preceq_E T_2
\quad \text{if and only if} \quad
\sfread(T_1) \preceq_L \sfread(T_2).
\]
The ``only if'' part was proved in \cref{lem: read and order}(2).
To prove ``if'' part, suppose that $\sfread(T_1) \preceq_L \sfread(T_2)$.
Since $T_E \preceq_E T_1$, \cref{lem: read and order}(2) implies that $\sfread(T_E) \preceq_L \sfread(T_1)$.
Therefore, there exists a reduced expression $s_{i_p} \cdots s_{i_2} s_{i_1}$ for $\sfread(T_2) \sfread(T_E)^{-1}$ such that
$s_{i_{p'}} \cdots s_{i_2} s_{i_1} = \sfread(T_1) \sfread(T_E)^{-1}$ for some $1 \le p' \le p$.
Again, by \cref{lem: read and order}(2),
\[
\pi_{i_{p'}} \cdots \pi_{i_2} \pi_{i_1} \cdot T_E = T_1
\quad \text{and} \quad
\pi_{i_p} \cdots \pi_{i_{p'+1}} \pi_{i_{p'}} \cdots  \pi_{i_2} \pi_{i_1} \cdot T_E = T_2.
\]
This implies that $\pi_{i_p} \cdots \pi_{i_{p'+2}} \pi_{i_{p'+1}} \cdot T_1 = T_2$, thus $T_1 \preceq_E T_2$.
\end{proof}

We conclude this subsection by proving the equality in \cref{eq: char of bfG}.

\begin{proposition}\label{Prop: characteristic image of G}
For any $\lambda \vdash n$ and $1 \le m \le n$, $\ch([\bfG_{\lambda;m}]) = \GSm{\lambda}{m}$. 
Consequently, 
$$
\sum_{1 \le m \le n} \ch([\bfG_{\lambda;m}]) = \GS{\lambda}.
$$
\end{proposition}

\begin{proof}
Enumerate the equivalence classes in $\calE_{\lambda;m}$ by $E_1, E_2, \ldots, E_p$.
For each $1 \le i \le p$, let $(E_i, \preceq_{E_i}^t)$ be a linear extension of the poset $(E_i, \preceq_{E_i})$.
We define a total order $\preceq$ on $\IGLTm{\lambda}{m}$ by
\[
T_1 \preceq T_2
\quad \text{if $i_1 < i_2$ or ($i_1 = i_2$ and $T_1 \preceq_{E_{i_1}}^t T_2$)},
\]
where $1 \le i_1, i_2 \le p$ such that $T_1 \in E_{i_1}$ and $T_2 \in E_{i_2}$.
Enumerate the elements in $\IGLTm{\lambda}{m}$ 
by $T_1 \preceq T_2 \preceq \cdots \preceq T_{|\IGLTm{\lambda}{m}|}$.
Let
\[
M_0 := \{0\}
\quad \text{and} \quad
M_j := \C
\{T_i \mid 1 \le i \le j \}
\quad \text{for $1 \le j \le |\IGLTm{\lambda}{m}|$.}
\]
Then
\[
\{0\} = M_0 \subseteq M_1 \subseteq \cdots \subseteq M_{|\IGLTm{\lambda}{m}|} = \bfG_{\lambda;m}
\]
is a filtration of $\bfG_{\lambda;m}$ and $M_j / M_{j-1}$ is one-dimensional for each $1 \le j \le |\IGLTm{\lambda}{m}|$. 
Therefore,
$\ch([\bfG_{\lambda;m}]) = \sum_{j=1}^{|\IGLTm{\lambda}{m}|} \ch([M_j / M_{j-1}])
= \sum_{T \in \IGLTm{\lambda}{m}} F_{\comp(T)}
= U_{\lambda;m}$.
\end{proof}

\subsection{A weak Bruhat interval module structure on $\mathbf{G}_E$}
\label{Subsec: WBIM module}

The purpose of this subsection is to prove that $\bfG_E \cong 
\sfB(\sfread(T_E), \sfread(T'_E))$ as $H_m(0)$-modules.

\begin{lemma}\label{lem: Descent cond for T}{\rm (cf. \cite[Lemma 5.5]{20CKNO2})}
Let $T \in E$ and $1 \le i \le m-1$. 
Suppose that $i \in T(\tH_{j_1})$ and $i+1 \in T(\tH_{j_2})$ for some $1 \le j_1, j_2 \le k+1$.
\begin{enumerate}[label = {\rm (\arabic*)}]
\item 
$i \in \Des(T)$ if and only if $j_1 < j_2$.
\item $i \in \Des(T)$ if and only if $i \notin \Des_L(\sfread(T))$.
\end{enumerate}
\end{lemma}


\begin{proof}
If $j_2 \le j_1$, then $i+1$ appears to the left of $i$ in $\sfread(T)$. 
Otherwise, $i+1$ appears to the right of $i$ in $\sfread(T)$. 
It follows that $i \notin \Des_L(\sfread(T))$ if and only if $j_1 < j_2$. Therefore, it suffices to show that (1) holds.

To prove the ``only if'' part of (1), assume that $i$ is a descent of $T$.
By \cref{thm: poset isomorphism}(1), we can take a permutation $\sigma \in \SG_m$ such that
\begin{align*}
T = \pi_{\sigma} \cdot T_E = s_{p_{l}} \cdot \  \cdots \  \cdot s_{p_2} \cdot s_{p_1} \cdot T_E,
\end{align*}
where $l := \ell(\sigma)$ and $s_{p_{l}} \cdots s_{p_2} s_{p_1}$ is a reduced expression for $\sigma$.

We use the induction on $l$ to show the assertion.
In case where $l = 0$, we have $T = T_E$. 
Since $i$ is a descent of $T_E$ and $i \in T_E(\tH_{j_1})$, we have $i+1 \in T_E(\tH_{j_1 + 1})$ and so $j_1 < j_1 + 1 = j_2$.
Suppose that $l = 1$.
Let $\sigma = s_p$ for some $1 \le p \le m-1$.
Then, $T$ and $T_E$ are same except for $p$ and $p+1$.
In addition, $p$ cannot be a descent of $T$.
Therefore, we have only to check the cases where $i = p-1, p+1$.
Assume that $i = p-1$.
Since $T = s_{i+1} \cdot T_E$,
the assumption $i+1 \in T(\tH_{j_2})$ implies that
$i+2 \in T_E(\tH_{j_2})$.
Moreover, since $i+1$ is a descent in $T_E$, $i+1$ is contained in $T_E(\tH_{j_2 - 1})$ by the definition of $\tH_{j}$'s.
Since $i$ is contained in $T_E(\tH_{j_1})$ and $i+1$ is contained in $T_E(\tH_{j_2-1})$, we have $j_1 \le j_2 -1$ and thus $j_1 < j_2$.
The case $i = p+1$ can be proved in the same manner as above.

For the induction step, we assume that $l \ge 2$ and the assertion is true for any $U \in E$ such that $U = \pi_\gamma \cdot T_E$ for some $\gamma \in \SG_m$ satisfying $\ell(\gamma) < l$.
Let $T' = \pi_{\sigma'} \cdot T_E$ with $\sigma' = s_{p_l} \sigma = s_{p_{l-1}} \cdots s_{p_2} s_{p_1}$. 
Note that $p_l$ is not a descent of $T$ and $T$ is identical to $T'$ except for $p_l$ and $p_{l}+1$.
Therefore, it is enough to check the cases where $i = p_l-1, p_l+1$.
Since the case $i = p_l + 1$ can be proved in the same manner as in the case $i = p_l -1$, we only deal with the  case $i = p_l -1$.

Suppose that $i = p_l - 1$ is a descent of $T$.
Let $q \in [1, k+1]$ be the index satisfying $p_l + 1 \in T(\tH_q)$.
Then
\[
p_l - 1 \in T'(\tH_{j_1}), \quad 
p_l \in T'(\tH_{q}), \quad \text{and} \quad
p_l + 1 \in T'(\tH_{j_2}).  
\]
Since $p_l$ is a descent of $T'$, the induction hypothesis implies that $q < j_2$.
If $p_l - 1$ is a descent of $T'$, then the induction hypothesis again implies that $j_1 < q$, thus $j_1 < j_2$.
For the remaining case, assume that $p_l-1$ is not a descent of $T'$.
Then, we have
$$
r_\sfb^{(p_l)}(T') \le
r_\sft^{(p_l-1)}(T').
$$
If $r_\sfb^{(p_l)}(T') = r_\sft^{(p_l-1)}(T')$, then $j_1 \le q$ by the definition of $\tH_j$'s.
Since $q < j_2$, we have $j_1 < j_2$.
Assume that $r_\sfb^{(p_l)}(T') < r_\sft^{(p_l-1)}(T')$.
Then $T'' := s_{p_l - 1} \cdot T'$ is contained in $E$ and $\pi_{p_l - 1} \cdot T'' = T'$.
Note that 
\[
p_l - 1 \in T''(\tH_{q}), \quad 
p_l \in T''(\tH_{j_1}), \quad \text{and} \quad
p_l + 1 \in T''(\tH_{j_2}).
\]
Combining the assumption that $p_l-1$ is a descent of $T$ with the equality $T = \pi_{p_l} \cdot T'$ yields that $r_\sft^{(p_l-1)}(T') < r_\sfb^{(p_l+1)}(T')$.
Since $T'' = s_{p_l - 1} \cdot T'$, we have that
\[
r_\sft^{(p_l)}(T'') <
r_\sfb^{(p_l+1)}(T'').
\]
Therefore, $p_l$ is a descent of $T''$.
Since $T'' \preceq_E T$, by
\cref{thm: poset isomorphism}(2), there exists $\sigma'' \in \SG_m$ such that $\pi_{\sigma''} \cdot T_E = T''$ and $\ell(\sigma'') < l$. Hence, by the induction hypothesis, we have that $j_1 < j_2$.

Next, we prove the ``if'' part of (1).
Suppose that $j_{1}<j_{2}$, but $i$ is not a descent of $T$.
If $\Top_i(T)$ and $\Bot_{i+1}(T)$ are in the same row in $T$, then $\Bot_{i+1}(T)$ lies to the immediate right of $\Top_i(T)$.
This, together with the properties (P1)--(P3), implies that $\Top_i(T)$ is the rightmost box of $\tH_{j_1}$ and $\Bot_{i+1}(T)$ is the leftmost box of $\tH_{j_2}$.
Note that for each $1 \le t \le k$, the rightmost box of $\tH_t$ is located strictly above the leftmost box of $\tH_{t+1}$.
Therefore, there exists $j_1 < t_0 < j_2$ such that the leftmost box in $\tH_{t_0}$ is located strictly below $\Top_i(T)$ and the rightmost box in $\tH_{t_0}$ is located strictly above $\Bot_{i+1}(T)$.
By the definition of $\tH_{t_0}$, there exists $p \in T_E(\tH_{t_0})$ such that $p \in \calI(T_E)$ and 
\[
r^{(p)}_{\sft}(T_E) < r^{(i)}_{\sft}(T) = r^{(i+1)}_{\sfb}(T) < r^{(p)}_{\sfb}(T_E).
\]
Since $\Top_i(T)$ (resp. $\Bot_{i+1}(T)$)  is contained in $\tH_{j_1}$ (resp. $\tH_{j_2}$) and $j_1 < t_0 < j_2$, we have $T_E(\Top_i(T)) < p < T_E(\Bot_{i+1}(T))$.
It follows that the lattice path $\tGam_{p}(T_E)$ passes through  $\mathtt{VL}(\grco{r_\sft^{(i)}-1}{c_\sft^{(i)}}, \grco{r_\sft^{(i)}}{c_\sft^{(i)}})$. 
In addition, since $T$ and $T_E$ are in the same class $E$, there exists $p' \in \calI(T)$ such that $\tGam_{p'}(T) = \tGam_{p}(T_E)$.
By the definition of $\tGam_{p'}(T)$, it follows that $i < p' < i+1$, which is absurd.
Thus, $\Bot_{i+1}(T)$ is strictly above $\Top_i(T)$.
Let $T' := s_i \cdot T$.
It is clear that $T'$ is an increasing gapless tableau, $i \in \Des(T')$, $i \in T'(\tH_{j_2})$, and $i+1 \in T'(\tH_{j_1})$.
Thus, by the ``only if'' part, we have $j_2 < j_1$, which contradicts the assumption $j_1< j_2$.
Hence, we conclude that $i$ is a descent of $T$.
\end{proof}

Now, we are ready to prove the main theorem of this subsection.

\begin{theorem}\label{thm: WBIM for G_E}
For any $E \in \calE_{\lambda;m}$, as $H_m(0)$-modules,
\[
\bfG_E \cong 
\sfB(\sfread(T_E), \sfread(T'_E)).
\]
\end{theorem}

\begin{proof}
Let $\widetilde{f}: \bfG_E \ra \sfB(\sfread(T_E), \sfread(T'_E))$ be the linear map sending $T \mapsto \sfread(T)$ for all $T \in E$.
By \cref{thm: poset isomorphism}(2), $\widetilde{f}$ is a bijection, so it suffices to show that
\[
\widetilde{f}(\pi_i \cdot T) = \pi_i \cdot \widetilde{f}(T)
\quad \text{for any $1 \le i \le m-1$ and $T \in E$.}
\]

Take any $1 \le i \le m-1$ and $T \in E$.
By \cref{lem: Descent cond for T}(2), if $i \notin \Des(T)$, then $i \in \Des_L(\sfread(T))$, thus
\[
\widetilde{f}(\pi_i \cdot T) = \widetilde{f}(T) = \sfread(T) = \pi_i \cdot \sfread(T) = \pi_i \cdot \widetilde{f}(T).
\]
Suppose that $i$ is a non-attacking descent of $T$.
Then \cref{lem: read and order}(1) implies that
$\sfread(T) \preceq_L \sfread(\pi_i \cdot T) 
= s_i \  \sfread(T)$, thus
\[
\widetilde{f}(\pi_i \cdot T) 
= \sfread(\pi_i \cdot T) 
= s_i \  \sfread(T)
= \pi_i\cdot \sfread(T) 
= \pi_i \cdot \widetilde{f}(T).
\]
Suppose that $i$ is an attacking descent of $T$.
By \cref{lem: Descent cond for T}(2), we have $i \notin \Des_L(\sfread(T))$ and so
\[
\pi_i \cdot \sfread(T) = 0
\quad \text{or} \quad
\pi_i \cdot \sfread(T) = s_i \  \sfread(T) \in [\sfread(T_{E}), \sfread (T'_{E})]_{L}.
\]
Assume that $\pi_i \cdot \sfread(T) = s_i \  \sfread(T)$.
Then by \cref{thm: poset isomorphism}(2), there exists $T' \in E$ such that $\sfread(T') = s_i \  \sfread(T)$.
By the definition of $\sfread(T)$, we have $T' = s_i \cdot T$.
However, from the assumption that $i$ is an attacking descent of $T$, one can easily see that $s_i \cdot T$ is not in $E$, which gives a contradiction.
It follows that $\pi_i \cdot \sfread(T) = 0$.
On the other hand, since $i$ is an attacking descent of $T$,  $\widetilde{f}(\pi_i \cdot T) = 0$, thus
$\widetilde{f}(\pi_i \cdot T) = 0 = \pi_i \cdot \widetilde{f}(T)$.
\end{proof}

\section{The projective cover of $\bfG_E$}\label{Sec: proj cover}

The purpose of this section is to find the projective cover of $\bfG_E$.
To begin with, we introduce the necessary terminology.

Let $A,B$ be finitely generated $H_m(0)$-modules.
A surjective $H_m(0)$-module homomorphism $f:A\to B$ is called an \emph{essential epimorphism} if an $H_m(0)$-module homomorphism $g: X\to A$ is surjective  
whenever $f \circ g:X\to B$ is surjective.
A \emph{projective cover} of $A$ is an essential epimorphism $f:P\to A$ with $P$ a projective $H_m(0)$-module.
The following lemma is useful when determining whether a surjective $H_m(0)$-module homomorphism is an essential epimorphism.

\begin{lemma}{\rm (\cite[Proposition 3.6]{95ARS})}\label{lem: ess epi}
The following are equivalent for an epimorphism $f:A \ra B$, where $A$ and $B$ are finitely generated modules over a left artin ring.
\begin{enumerate}[label = {\rm (\alph*)}]
\item $f$ is an essential epimorphism.
\item $\ker(f) \subset \rad(A)$.
\end{enumerate}
\end{lemma}

According to \cref{lem: ess epi}, knowing membership conditions of $\ker(f)$ and $\rad(A)$ can help determine whether $f$ is an essential epimorphism or not. 
To study a membership condition for $\rad(A)$, let us collect some definitions and notation.

Let $\bal$ be a generalized composition.
For any $\scrT \in \SRT(\bal)$, let $\itread(\scrT)$ be the word obtained from $\scrT$ by reading the entries from left to right starting with the bottom row.
Combining \cite[Theorem 3.3 and Proposition 5.1]{16Huang} with \cite[Theorem 1(3)]{22JKLO}, one can see that the $\C$-linear map 
\begin{align}\label{eq: Phi map}
\sfLread: \bfP_\bal \ra \sfB(w_0(\bal_\bullet^\rmc), w_0 w_0(\bal_\odot)), \quad 
\scrT \mapsto \itread(\scrT) \quad \text{for $\scrT \in \SRT(\bal)$}
\end{align}
is an $H_m(0)$-module isomorphism.
In addition, if we define a partial order $\preceq_{\SRT(\bal)}$ on $\SRT(\bal)$ by 
\[
\scrT_1  \preceq_{\SRT(\bal)}  \scrT_2 \quad \text{if $\pi_\sigma \cdot \scrT_1 = \scrT_2$ for some $\sigma \in \SG_m$},
\]
then the map
\begin{align*}
\sfLread': (\SRT(\bal), \preceq_{\SRT(\bal)}) \ra ([w_0(\bal_\bullet^\rmc), w_0 w_0(\bal_\odot)]_L, \preceq_L),
\ 
\scrT \mapsto \Phi(\scrT) \  \text{for $\scrT \in \SRT(\bal)$}
\end{align*}
is an order isomorphism.
For the definitions of $\bal_\bullet$ and $\bal_\odot$, see \cref{eq: bal bullet and odot}.

Let $\scrT_{\bal}^{\bullet}$ and 
$\scrT_{\bal}^{\odot}$ be the unique standard ribbon tableaux of shape $\bal$ satisfying 
\begin{align*}
\itread(\scrT_{\bal}^{\bullet}) = w_0(\bal_\bullet^\rmc)
\quad \text{and} \quad 
\itread(\scrT_{\bal}^{\odot})  = w_0(\bal_\odot^\rmc),
\end{align*}
respectively.
Considering the isomorphism $\Phi$ defined in \cref{eq: Phi map}, one can see that
\[
\itread(\scrT_{\bal}^{\bullet}) = \itread(\scrT_{\bal_\bullet})
\quad \text{and} \quad 
\itread(\scrT_{\bal}^{\odot})  = \itread(\scrT_{\bal_\odot}).
\]
For the definitions of $\scrT_{\bal_\bullet}$ and $\scrT_{\bal_\odot}$, see the last paragraph in \cref{subsec: Proj module}.
For instance, if $\bal = (2,1) \RS (1,1)$, then $w_0(\bal_\bullet^\rmc) = 21345$ and $w_0(\bal_\odot^\rmc) = 21435$, therefore
$$
\scrT_{\bal}^{\bullet} = 
\begin{array}{l}
\begin{ytableau}
\none & \none & 4 & 5 \\
1 & 3 \\
2
\end{ytableau}
\end{array} 
\quad \text{and} \quad
\scrT_{\bal}^{\odot} = 
\begin{array}{l}
\begin{ytableau}
\none & \none & 3 & 5 \\
1 & 4 \\
2
\end{ytableau}
\end{array}.
$$
It follows from the definitions of $\scrT_{\bal}^{\bullet}, \scrT_{\bal}^{\odot}$ that $\scrT_{\bal}^{\bullet} \preceq_{\SRT(\bal)} \scrT_{\bal}^{\odot}$.
Thus, we can define 
$$
[\scrT_{\bal}^{\bullet}, \scrT_{\bal}^{\odot}] := \{\scrT \in \SRT(\bal) \mid 
\scrT_{\bal}^{\bullet} \preceq_{\SRT(\bal)} \scrT \preceq_{\SRT(\bal)} \scrT_{\bal}^{\odot}\}.
$$

Under a special assumption on $\bal$, Choi, Kim, Nam, and Oh \cite[Lemma 5.9]{20CKNO2} provided a necessary condition for standard ribbon tableaux of shape $\bal$ to be included in the radical of $\bfP_\bal$.
It motivates us to state the following lemma which plays a key role in proving the main theorem of this section.
To avoid excessive overlap with their proof, we omit some part of the proof in the following lemma.

\begin{lemma}\label{Lem: calT in Rad}
Let $\bal$ be a generalized composition of $m$ and $\scrT$ be a standard ribbon tableau of shape $\bal$. 
If $\scrT \not\preceq_{\SRT(\bal)} \scrT_{\bal}^{\odot}$, then $\scrT \in \rad(\bfP_{\bal})$.
\end{lemma}

\begin{proof}
Given $\scrT \in \SRT(\bal)$, 
let $s_{i_l} \cdots s_{i_2} s_{i_1}$ be a reduced expression for the permutation $\itread(\scrT) {\itread}(\scrT_{\bal}^{\bullet})^{-1}$ in $\SG_m$.
Considering the equality $\itread(\scrT_{\bal}^{\bullet}) = w_0(\bal_\bullet^\rmc)$ together with the order isomorphism $\sfLread'$, 
we have that $\scrT = \pi_{\itread(\scrT) {\itread}(\scrT_{\bal}^{\bullet})^{-1}} \cdot \scrT_{\bal}^{\bullet}$.
By following the proof of \cite[Lemma 5.9]{20CKNO2}, one can see that 
\begin{align}\label{eq: rad membership}
\scrT \in \rad(\bfP_{\bal})
\quad
\text{
if $i_j \in \Des(\scrT_{\bal}^{\odot})$ for some $1 \le j \le l$,} 
\end{align}
where $\Des(\scrT) := \{i \in [1,n-1] \mid \text{$i$ appears weakly below $i+1$ in $\scrT$}  \}$.

On the other hand, the equality $\itread(\scrT_{\bal}^{\odot})  = w_0(\bal_\odot^\rmc)$ implies that 
$\itread(\scrT) \preceq_{L} \itread(\scrT_{\bal}^{\odot})$
if and only if $i_j \in \set(\bal_\odot^\rmc)$ for all $1\le j \le l$.
From the definition of $\scrT_{\bal}^{\odot}$, it follows that $\Des(\scrT_{\bal}^{\odot}) = \set(\bal_\odot)$.
Putting these together, we have that
if $\itread(\scrT) \not\preceq_{L} \itread(\scrT_{\bal}^{\odot})$, then $i_j \in \Des(\scrT_{\bal}^{\odot})$ for some $1 \le j \le l$.
Combining this observation with \cref{eq: rad membership} yields that if $\itread(\scrT) \not\preceq_{L} \itread(\scrT_{\bal}^{\odot})$, then $\scrT \in \rad(\bfP_{\bal})$.
Now, considering the order isomorphism $\sfLread|_{\SRT(\bal)}^{-1}$, we see that if $\scrT \not\preceq_{\SRT(\bal)} \scrT_{\bal}^{\odot}$, then $\scrT \in \rad(\bfP_{\bal})$.
\end{proof}

Next, let us introduce the notation needed to describe the projective cover of $\bfG_E$.
Recall that $\Des(T_E) = \{d_1< d_2 < \cdots < d_k\}$, $d_0 = 0$, and $d_{k+1} = m$.
Let $\bal^{(1)} := (d_1)$.
For $1 < j \le k+1$, define
\begin{align}\label{Def: bal^(j+1)}
\bal^{(j)} :=
\begin{cases}
\bal^{(j-1)} \cdot (d_{j} - d_{j-1}) & 
\text{if $\Bot_{d_{j-2}+1}(T_E)$ is weakly left of $\Top_{d_{j}}(T_E)$,} \\[1ex]
\bal^{(j-1)} \RS (d_{j} - d_{j-1}) &
\text{if $\Bot_{d_{j-2}+1}(T_E)$ is strictly right of $\Top_{d_{j}}(T_E)$.}
\end{cases}
\end{align}
For the definition of $\bal^{(j-1)} \cdot (d_{j} - d_{j-1})$, see \cref{eq: bal odot beta}.
Let 
\begin{equation}\label{eq: generalized composition}
\bal_E := \bal^{(k+1)}.
\end{equation}
Given $\scrT \in \SRT(\bal_E)$, we define $T_{\scrT}$ to be the filling of $\tyd(\lambda)$ whose boxes in each $\tH_j~(1 \le j \le k+1)$ are filled with the entries of the $j$th column of $\scrT$ in the following manner:
\begin{enumerate}[label = {\rm (\roman*)}]
\item Let $\epsilon_1 < \epsilon_2 < \cdots < \epsilon_l$ be the entries of the $j$th column of $\scrT$ and let $C_1, C_2, \ldots, C_d$ be the connected components of $\tH_j$ such that $C_i$ is left of $C_{i+1}$ for $1 \le i < d$.
\item Let $c_0 := 1$ and let $c_i := \sum_{p=1}^{i} |C_p| - i + 1$ for $1 \le i \le d$.
\item For each $1 \le i \le d$, fill $C_i$ with $\epsilon_{c_{i-1}}, \epsilon_{c_{i-1} + 1}, \ldots, \epsilon_{c_{i}}$ from left to right.
\end{enumerate}
For later use, we notice that $s_i \cdot T_{\scrT} = T_{s_i \cdot \scrT}$ for any $i \in [1, m-1]$ and $\scrT \in \SRT(\bal_E)$.

Now, we define a $\C$-linear map $\eta: \bfP_{\bal_E} \ra \bfG_E$ by
\[
\scrT \mapsto \begin{cases}
T_{\scrT} & \text{if $T_{\scrT}$ is contained in $E$}, \\
0 & \text{otherwise}
\end{cases}
\]
for $\scrT \in \SRT(\bal_E)$ and extending it by linearity.

\begin{example}\label{Eg: eta map}
(1)
Let 
$$
T_E = \begin{array}{c}
\begin{ytableau}
1 & 2 & 3 & 4\\
2 & 3 \\
5 \\
6
\end{ytableau}
\end{array}.
$$
Note that $\Des(T_E) = \{1,2,4,5\}$.
The following figure illustrates $\tH_j$:
\[
\begin{array}{c}
\begin{ytableau}
*(red!20) \tH_1 & *(blue!20) \tH_2 & *(gray!20) \tH_3 & *(gray!20) \tH_3 \\
*(blue!20) \tH_2 & *(gray!20) \tH_3 \\
*(magenta!20) \tH_4\\
*(cyan!20) \tH_5
\end{ytableau}
\end{array}
\]

On the other hand, one can see that $\bal_E = (1,1,2) \RS (1,1)$.
Let
\[
\scrT_0 = \begin{array}{c}
\begin{ytableau}
\none & \none & \none & 5 & 6\\
\none & \none  & 3 \\
1 & 2 & 4
\end{ytableau}
\end{array},
\qquad
\scrT_1 = \begin{array}{c}
\begin{ytableau}
\none & \none & \none & 3 & 4\\
\none & \none  & 5 \\
1 & 2 & 6
\end{ytableau}
\end{array},
\qquad
\scrT_2 = \begin{array}{c}
\begin{ytableau}
\none & \none & \none & 2 & 4\\
\none & \none  & 3 \\
1 & 5 & 6
\end{ytableau}
\end{array}.
\]
Then, we have that
$$
T_{\scrT_0} = \begin{array}{c}
\begin{ytableau}
*(red!20) 1 & *(blue!20) 2 & *(gray!20) 3 & *(gray!20) 4 \\
*(blue!20) 2 & *(gray!20) 3 \\
*(magenta!20) 5\\
*(cyan!20) 6
\end{ytableau}
\end{array},
\qquad
T_{\scrT_1} = \begin{array}{c}
\begin{ytableau}
*(red!20) 1 & *(blue!20) 2 & *(gray!20) 5 & *(gray!20) 6 \\
*(blue!20) 2 & *(gray!20) 5 \\
*(magenta!20) 3 \\
*(cyan!20) 4
\end{ytableau}
\end{array}, 
\quad \text{and} \quad 
T_{\scrT_2} = \begin{array}{c}
\begin{ytableau}
*(red!20) 1 & *(blue!20) 5 & *(gray!20) 3 & *(gray!20) 6 \\
*(blue!20) 5 & *(gray!20) 3 \\
*(magenta!20) 2 \\
*(cyan!20) 4
\end{ytableau}
\end{array}.
$$
For instance, one can obtain $T_{\scrT_2}$ with the datum in \cref{Table: T_scrT_2}.
Since $T_{\scrT_0} = T_E$ and $T_{\scrT_1} = T$ are contained in $E$, $\eta(\scrT_0) = T_E$ and $\eta(\scrT_1) = T$.
On the other hand, since $T_{\scrT_2}$ is not an increasing tableau, it is not contained in $E$ and thus $\eta(\scrT_2) = 0$.

\begin{table}[t]
\[
\arraycolsep = 0.5em 
\begin{array}{c|c|c|c|c}
j  &  \epsilon_1,\epsilon_2,\ldots, \epsilon_l  &  C_1, C_2, \ldots, C_d  &  c_0,c_1,\ldots,c_d  & T_{\scrT_2}(\tH_j)
\\ \hline \hline 
1 & 1 & \{(1,1)\} & 1,1 & 
\begin{array}{l}
\scalebox{0.5}{$
\begin{ytableau}
*(red!20) 1 & *(blue!20)  & *(gray!20)  & *(gray!20)  \\
*(blue!20)  & *(gray!20)  \\
*(magenta!20)  \\
*(cyan!20) 
\end{ytableau}$}
\end{array} 
\\ \hline 
2 & 5 & \{(2,1)\}, \{(1,2)\} & 1,1,1 & 
\begin{array}{l}
\scalebox{0.5}{$
\begin{ytableau}
*(red!20)  & *(blue!20) 5 & *(gray!20)  & *(gray!20)  \\
*(blue!20) 5 & *(gray!20)  \\
*(magenta!20)  \\
*(cyan!20) 
\end{ytableau}$} 
\end{array}
\\ \hline 
3 & 3,6 & \{(2,2)\}, \{(1,3), (1,4)\} & 1,1,2 & \begin{array}{l}
\scalebox{0.5}{$
\begin{ytableau}
*(red!20)  & *(blue!20)  & *(gray!20) 3 & *(gray!20) 6 \\
*(blue!20)  & *(gray!20) 3  \\
*(magenta!20)  \\
*(cyan!20) 
\end{ytableau}$} 
\end{array} 
\\ \hline 
4 & 2 & \{(3,1)\} & 1,1 & 
\begin{array}{l}
\scalebox{0.5}{$
\begin{ytableau}
*(red!20)  & *(blue!20)  & *(gray!20)  & *(gray!20)  \\
*(blue!20)  & *(gray!20)  \\
*(magenta!20) 2  \\
*(cyan!20) 
\end{ytableau}$} 
\end{array}
\\ \hline 
5 & 4 & \{(4,1)\} & 1,1 & 
\begin{array}{l}
\scalebox{0.5}{$
\begin{ytableau}
*(red!20) & *(blue!20)  & *(gray!20)  & *(gray!20)  \\
*(blue!20)  & *(gray!20)  \\
*(magenta!20)  \\
*(cyan!20) 4
\end{ytableau}$}
\end{array}
\end{array}
\]
\smallskip
\caption{Datum for constructing $T_{\scrT_2}$ in \cref{Eg: eta map}(1)}
\label{Table: T_scrT_2}
\end{table}

(2) Letting
$$
T_E = \begin{array}{c}
\begin{ytableau}
1 & 2 & 4 & 5\\
2 & 3 \\
4
\end{ytableau}
\end{array},
$$
we have that
\[
\bal_E = (1,1,1,2) 
\quad \text{and} \quad
\begin{array}{c}
\begin{ytableau}
*(red!20) \tH_1 & *(blue!20) \tH_2 & *(magenta!20) \tH_4 & *(magenta!20) \tH_4 \\
*(blue!20) \tH_2 & *(gray!20) \tH_3 \\
*(magenta!20) \tH_4
\end{ytableau}
\end{array}.
\]
When $\scrT = \begin{array}{c}
\begin{ytableau}
\none & \none & \none & 3 \\
1 & 2 & 4 & 5 
\end{ytableau}
\end{array}$,
one sees that
\[
T_\scrT = \begin{array}{c}
\begin{ytableau}
*(red!20) 1 & *(blue!20) 2 & *(magenta!20) 3 & *(magenta!20) 5 \\
*(blue!20) 2 & *(gray!20) 4 \\
*(magenta!20) 3
\end{ytableau}
\end{array}.
\]
The filling $T_\scrT$ is an increasing gapless tableau, but $\eta(\scrT) = 0$ since $T_\scrT$ is not contained in $E$.
\end{example}

As a first step to prove that the map $\eta: \bfP_{\bal_E} \ra \bfG_E$ is a projective cover of $\bfG_E$, we  show that $\eta: \bfP_{\bal_E} \ra \bfG_E$ is a surjective $H_{m}(0)$-module homomorphism. 

\begin{lemma}\label{lem: surj linear map}
The map $\eta: \bfP_{\bal_E} \ra \bfG_E$ is a surjective $H_m(0)$-module homomorphism.
\end{lemma}

\begin{proof}
For each $T \in E$, let $\scrT_T$ be the filling of $\trd(\bal_E)$ whose $j$th column is filled with the elements of $T(\tH_j)$ so that they are increasing from top to bottom for all $1 \le j \le k+1$.
Then it follows straightforwardly from the definitions of $\sfread$, and $\itread$ that
$\sfread(T) = \itread(\scrT_T)$ for all $T \in E$. 
By the definition of $\bal_E$, we have $\scrT_{\bal_E} = \scrT_{T_E}$.
For the definition of $\scrT_{\bal_E}$, see \cref{subsec: Proj module}.
In addition, from \cref{thm: WBIM for G_E} and \cref{eq: Phi map}, we have the $H_m(0)$-module isomorphisms
\begin{align*}
\widetilde{f}: \bfG_E \ra \sfB(\sfread(T_E), \sfread(T'_E)), \quad
T \mapsto \sfread(T) \quad \text{for $T \in E$}
\end{align*}
and
\begin{align*}
\sfLread: \bfP_\bal \ra \sfB(w_0(\bal_\bullet^\rmc), w_0 w_0(\bal_\odot)), \quad 
\scrT \mapsto \itread(\scrT) \quad \text{for $\scrT \in \SRT(\bal)$.}
\end{align*}
Putting these altogether, we see that 
\[
w_0(\bal_\bullet^\rmc) = \itread(\scrT_{\bal_E}) = \itread(\scrT_{T_E}) = \sfread(T_E).
\]
If we prove $\sfread(T'_E) \preceq_L w_0 w_0(\bal_\odot)$, then the linear map 
\begin{align*}
\pr : \sfB(w_0(\bal_\bullet^\rmc), w_0 w_0(\bal_\odot)) \ra \sfB(\sfread(T_E), \sfread(T'_E)), 
\ 
\gamma \mapsto \begin{cases}
\gamma & \text{if $\gamma \in [\sfread(T_E), \sfread(T'_E)]_L$,} \\
0 & \text{otherwise}
\end{cases}
\end{align*}
is a surjective $H_n(0)$-module homomorphism by \cref{eq: definition of pr}.
In addition, we have
\[
\eta = \widetilde{f}^{-1} \circ \pr \circ \Phi,
\]
which implies that $\eta$ is a surjective $H_m(0)$-module homomorphism.
Thus, it suffices to show that 
\begin{align}\label{eq: ineq read and w0}
\sfread(T'_E) \preceq_L w_0 w_0(\bal_\odot).
\end{align}

In order to prove \cref{eq: ineq read and w0}, we first show that $\scrT_T \in \SRT(\bal_E)$ for all $T \in E$. 
Choose a $T \in E$ and an arbitrary $1 < j \le k+1$.
Let
\begin{align*}
x & = \text{the entry at the uppermost box in the $(j-1)$st column of $\scrT_T$ and} \\
y & = \text{the entry at the lowermost box in the $j$th column of $\scrT_T$}.
\end{align*}
It suffices to show that $x < y$ only in the case where the $(j-1)$st column and the $j$th column are connected.
Note that the $(j-1)$st column and the $j$th column are connected if and only if $\Bot_{d_{j-2}+1}(T_E)$ is weakly left of $\Top_{d_{j}}(T_E)$. 
Assume that $\Bot_{d_{j-2}+1}(T_E)$ is weakly left of $\Top_{d_{j}}(T_E)$.
If $\Top_{d_{j}}(T_E)$ is weakly below $\Bot_{d_{j-2}+1}(T_E)$, then $x < y$ since $T$ is an increasing tableau.
Suppose $\Top_{d_{j}}(T_E)$ is strictly above $\Bot_{d_{j-2}+1}(T_E)$.
For the sake of contradiction assume that  $x > y$.
Take $\sigma \in \SG_m$ satisfying that $T = \pi_\sigma \cdot T_E$ and a reduced expression $s_{i_p} \cdots s_{i_2} s_{i_1}$ for $\sigma$.
Since $x$ appears at $\Bot_{d_{j-2}+1}(T_E)$ and $y$ appears at $\Top_{d_{j}}(T_E)$ in $T$, there exists $1 \le r \le p$ such that 
\begin{align*}
T' (\Bot_{d_{j-2}+1}(T_E)) < T'(\Top_{d_{j}}(T_E))
\quad \text{and} \quad
\pi_{i_r} \cdot T' (\Bot_{d_{j-2}+1}(T_E)) > \pi_{i_r} \cdot T'(\Top_{d_{j}}(T_E)).
\end{align*}
Here, $T' = \pi_{i_{r-1}} \cdots \pi_{i_{2}} \pi_{i_{1}} \cdot T_E$.
Since the $\pi_{i_r}$-action swap $i_r$ and $i_r + 1$, the following hold:
\begin{align}\label{eq: T' and i_r}
\begin{aligned}
& T'(\Bot_{d_{j-2}+1}(T_E)) = i_r, \quad 
T'(\Top_{d_{j}}(T_E)) = i_{r} + 1,
\quad \text{and} \\
& \text{$i_r$ is a non-attacking descent in $T'$.}
\end{aligned}
\end{align}
However, since $\Bot_{d_{j-2}+1}(T_E)$ is below $\Top_{d_{j}}(T_E)$ in $T'$, the descent $i_r$ of $T'$ cannot be non-attacking.
Hence, \cref{eq: T' and i_r} cannot occur, which shows that $x < y$.

Since $\scrT_{T'_E} \in \SRT(\bal_E)$, we have
\[
\sfread(T'_E) = \itread(\scrT_{T'_E}) \preceq_L  w_0 w_0(\bal_\odot)
\]
as desired.
\end{proof}

We are ready to prove the main theorem of this section.

\begin{theorem}\label{thm: proj cover}
For any $E \in \calE_{\lambda;m}$, $\eta: \bfP_{\bal_E} \ra \bfG_E$ is a projective cover of $\bfG_E$.
\end{theorem}

\begin{proof}
By \cref{Lem: calT in Rad}, we have 
$\C (\SRT(\bal_E) \setminus [\scrT_{\bal_E}^{\bullet}, \scrT_{\bal_E}^{\odot}]) \subseteq \rad(\bfP_{\bal_E})$.
Therefore, by \cref{lem: ess epi}, it suffices to show that 
\begin{align}\label{eq: ker in SRT}
\ker(\eta) \subseteq \C (\SRT(\bal_E) \setminus [\scrT_{\bal_E}^{\bullet}, \scrT_{\bal_E}^{\odot}]).
\end{align}

First, we claim that $T_{\scrT_{\bal_E}^{\odot}}$ is an increasing tableau.
Take any boxes $B_1, B_2 \in \tyd(\lambda)$ with $B_1 \neq B_2$ and $B_2$ is positioned weakly southeast of $B_1$, that is, $B_2 \in B_1 + (\Z_{\ge 0} \times \Z_{\ge 0} \setminus (0,0))$.
We need to show that $T_{\scrT_{\bal_E}^{\odot}}(B_1) < T_{\scrT_{\bal_E}^{\odot}}(B_2)$.
To prove it, we collect necessary notation.

Let $j_{B_1}, j_{B_2} \in [1,k+1]$ such that $B_1 \in \tH_{j_{B_1}}$ and $B_2 \in \tH_{j_{B_2}}$.
For $1 \le j \le k+1$, let $\tC_j$ be the set of the boxes in the $j$th column of $\trd(\bal_E)$ from left to right.
Then, we have $T_{\scrT_{\bal_E}^{\odot}}(\tH_{j_{B_1}}) = \scrT_{\bal_E}^{\odot}(\tC_{j_{B_1}})$ and $T_{\scrT_{\bal_E}^{\odot}}(\tH_{j_{B_2}}) = \scrT_{\bal_E}^{\odot}(\tC_{j_{B_2}})$.
Letting $(\bal_E)_{\odot} = (a_1,a_2,\ldots, a_p)$, set
$$
A_0 := 0 
\quad \text{and} \quad
A_{r} := a_1 + \cdots + a_r \quad \text{for $1 \le r \le p$.}
$$
By the definition of $\scrT^{\odot}_{\bal_E}$, for each $1 \le r \le p$, there exist $u_r, u'_r \in[1,k+1]$ such that $\bigcup_{u_r \le j \le u'_r} \scrT^{\odot}_{\bal_E}(\tC_j) = [A_{r-1} + 1, A_{r}]$.
Take $r_{B_1}$ and $r_{B_2}$ in $[1,p]$ satisfying that $\tC_{j_{B_1}} \subseteq [A_{r_{B_1}-1} + 1, A_{r_{B_1}}]$ and $\tC_{j_{B_2}} \subseteq [A_{r_{B_2}-1} + 1, A_{r_{B_2}}]$, respectively.

Considering the assumption that $B_2$ is positioned weakly southeast of $B_1$ together with the definitions of $\tH_{j_{B_1}}$ and $\tH_{j_{B_2}}$, we see that $j_{B_1} \le j_{B_2}$, thus $r_{B_1} \le r_{B_2}$.
If $r_{B_1} < r_{B_2}$, then we have $\scrT^{\odot}_{\bal_E}(B_1) < \scrT^{\odot}_{\bal_E}(B_2)$ by the definition of $r_{B_1}$ and $r_{B_2}$.
For the remaining case, suppose that $r_{B_1} = r_{B_2}$.
Note that $\tC_j$ and $\tC_{j+1}$ are disconnected for all $j \in [u_{r_{B_1}}, u'_{r_{B_1}}-1]$.
This implies that $\Bot_{d_{j-2}+1}(T_E)$ is strictly right of $\Top_{d_{j}}(T_E)$, equivalently, the leftmost box in $\tH_{j-1}$ is strictly right of the rightmost box of $\tH_{j}$.
Combining this observation with the definitions of $\tH_{j-1}$ and $\tH_{j}$ yields that the leftmost box in $\tH_{j-1}$ is strictly above the rightmost box of $\tH_{j}$.
Therefore, if $j_{B_1} < j_{B_2}$, then $B_2 \notin B_1 + (\Z_{\ge 0} \times \Z_{\ge 0} \setminus (0,0))$.
If $j_{B_1} = j_{B_2}$, then combining $B_2 \in B_1 + (\Z_{\ge 0} \times \Z_{\ge 0} \setminus (0,0))$ with (P2) written in the first paragraph of \cref{Subsec: poset structure} yields that $B_1$ and $B_2$ are in the same row.
It follows that $B_2$ is strictly east of $B_1$, thus $T_{\scrT_{\bal_E}^{\odot}}(B_1) < T_{\scrT_{\bal_E}^{\odot}}(B_2)$.

Next, we claim that $T_{\scrT_{\bal_E}^{\odot}} \in E$, equivalently, $T_{\scrT_{\bal_E}^{\odot}} \sim T_E$.
Take any $i \in \calI(T_E)$.
By the construction of $T_{\scrT_{\bal_E}^{\odot}}$, $|\calI(T_{\scrT_{\bal_E}^{\odot}})| = |\calI(T_E)|$ and there exists $i' \in [1,m]$ such that $T_{\scrT_{\bal_E}^{\odot}}^{-1} (i') = T_E^{-1}(i)$.
Let us show that $\Gamma_{i'}(T_{\scrT_{\bal_E}^{\odot}}) = \Gamma_i(T_E)$. 
Let
\[
\mathrm{Rect}_i := 
\left\{
(r,c) \in \tyd(\lambda) \setminus T_E^{-1}(i) \mid 
r^{(i)}_{\sft} \le r \le r^{(i)}_{\sfb} 
\quad \text{and} \quad 
c^{(i)}_{\sfb} \le c \le c^{(i)}_{\sft} 
\right\}.
\]
Choose any $B \in \mathrm{Rect}_i$ and let $j_{1}, j_{2} \in [1,k+1]$ with
$$
B \in \tH_{j_{1}}
\quad \text{and} \quad
T_E^{-1}(i) \subset \tH_{j_{2}}.
$$
Set $x =T_{E}(B)$ and $x' = T_{\scrT_{\bal_E}^{\odot}}(B)$.
In order to prove $\Gamma_{i'}(T_{\scrT_{\bal_E}^{\odot}}) = \Gamma_i(T_E)$, it suffices to show that $x < i$ if and only if $x' < i'$.
We omit the proof of the ``if'' part since it can be proved in the same manner as the ``only if'' part.

To prove the ``only if'' part, suppose that $x < i$, but $x' \ge i'$.
Since $B \notin T_E^{-1}(i)$ and $x'=T_{\scrT_{\bal_E}^{\odot}}(B)$, we have $x' > i'$.
Putting \cref{Lem: source and sink cond}, the inequality $x < i$, and $B \in \mathrm{Rect}_i$ together, one can derive that $j_1 < j_2$.
By the definition of $\scrT_{\bal_E}^{\odot}$, given $l_1, l_2 \in [1, k+1]$ with $l_1 < l_2$, if there exists $a \in \scrT_{\bal_E}^{\odot}(\tC_{l_1})$ and $b \in \scrT_{\bal_E}^{\odot}(\tC_{l_2})$ such that $a > b$, then $\tC_p$ is disconnected to $\tC_{p+1}$ for all $l_1 \le p < l_2$.
Because $x' \in T_{\scrT_{\bal_E}^{\odot}}(\tH_{j_{1}})$ and $i' \in T_{\scrT_{\bal_E}^{\odot}}(\tH_{j_{2}})$, this property, together with the inequalities $x' > i'$ and $j_1 < j_2$, implies that $\tC_p$ is disconnected to $\tC_{p+1}$ for all $j_{1} \le p < j_{2}$.
By \cref{Def: bal^(j+1)}, $\Bot_{d_{p-1}+1}(T_E)$ is strictly right of $\Top_{d_{p+1}}(T_E)$  for all $j_{1} \le p < j_{2}$.
Since $\Bot_{d_{p-1}+1}(T_E)$ is the leftmost box of $\tH_{p}$ and $\Top_{d_{p+1}}(T_E)$ is the rightmost box of $\tH_{p+1}$ for all $j_{1} \le p < j_{2}$, every box in $\tH_{j_1}$ is placed strictly right of each box in $\tH_{j_2}$.
Since $B \in \tH_{j_1}$ and $T_E^{-1}(i) \subset \tH_{j_2}$, this is a  contradiction to the choice of $B \in \mathrm{Rect}_i$.
Thus, if $x < i$, $x'$ must be less than $i'$ as desired.

Now, we have $\eta(\scrT_{\bal_E}^{\odot}) = T_{\scrT_{\bal_E}^{\odot}} \in E$.
Combining this with \cref{lem: surj linear map}, we see that $\eta(\scrT) =  T_\scrT \in E$ for all $\scrT \in [\scrT_{\bal_E}^{\bullet}, \scrT_{\bal_E}^{\odot}]$.
In addition, the definition of $\eta$ together with \cref{lem: surj linear map} imply that
$$
|\{\scrT \in \SRT(\bal_E) \mid T_\scrT \in E\}| 
= \dim(\mathrm{Im}(\eta)) 
= |\SRT(\bal_E)| - \dim (\ker(\eta)).
$$
It follows that the set $\{\scrT \in \SRT(\bal_E) \mid T_\scrT \notin E\}$ is a basis for $\ker(\eta)$.
Putting these together yields the inclusion \cref{eq: ker in SRT}.
\end{proof}

\begin{remark}
\cref{Lem: calT in Rad}
provides a method for finding a projective cover of weak Bruhat interval modules of the form $\sfB(w_0(\alpha), \rho)$, where $\alpha \models n$ and $\rho \in \SG_n$ with $w_0(\alpha) \preceq_L \rho$. 
This approach was recently applied to find a projective cover of poset modules associated with regular Schur labeled skew shape posets. 
For further details, refer to \cite[Section 5]{24KLO}.

Independently, Bardwell and Searles \cite{24BD} have introduced a type-independent method for finding projective covers of various modules of the $0$-Hecke algebras of finite Coxeter type. 
For more information, see \cite[Theorem 4.2]{24BD}.
\end{remark}

\section{Further avenues}\label{Sec: Further avenues}
(1) Pechenik and Yong \cite{17PY2} studied a theory of genomic tableaux parallel to the theory for increasing gapless tableaux developed by Thomas and Yong \cite{09TY}.
In order to relate genomic tableaux with increasing gapless tableaux, Pechenik and Yong \cite{17PY2} introduced a map, called the $K$-standardization, sending a genomic tableau to an increasing gapless tableau.
From the viewpoint of this correspondence, increasing gapless tableaux play a similar role to standard Young tableaux.  
However, to the best of the authors' knowledge, while the relationship between standard Young tableaux and permutations is well understood, the relationship between increasing gapless tableaux and permutations is not well studied.
Our standardized reading $\sfread$ can be helpful to study the relationship because it maps increasing gapless tableaux to permutations.
For this reason, it would be interesting to investigate the standardized reading $\sfread$ and the set $\{\sfread(T) \mid T \in E\}$ for each $E \in \calE_{\lambda;m}$.
\medskip

(2) For $\alpha \models n$, Tewari and van Willigenburg~\cite{15TW} introduced $H_n(0)$-modules $\bfS_\alpha$ such that $\ch([\bfS_\alpha])$ is equal to the \emph{quasisymmetric Schur function $\calS_\alpha$}.
As a generalization of $\calS_\alpha$, in \cite{19TW}, they also introduced $H_n(0)$-modules $\bfS_\alpha^{\sigma}$ for all $\sigma \in \SG_{\ell(\alpha)}$.
Then, they decomposed $\bfS^\sigma_\alpha$ into a direct sum of cyclic submodules $\bfS^\sigma_{\alpha,E}$.
In the case where $\sigma = \id$, K\"onig~\cite{19Konig} proved that all $\bfS^\id_{\alpha,E}$ are indecomposable.
Later, Choi, Kim, Nam, and Oh~\cite{20CKNO} showed K\"onig's method still works for all $\sigma \in \SG_{\ell(\alpha)}$ under suitable adjustments, proving that all $\bfS^\sigma_{\alpha,E}$ are also indecomposable.

In this paper, we give a direct sum decomposition $\bfG_{\lambda;m} = \bigoplus_{E \in \calE_{\lambda;m}} \bfG_E$. 
A natural question that arises is whether or not $\bfG_E$ is indecomposable for all $E \in \calE_{\lambda;m}$.
However, there exists $E \in \calE_{\lambda;m}$ such that $\bfG_E$ is not indecomposable.
For example, the set
\[
E = \left\{T_1 = \begin{array}{l}
\begin{ytableau}
1 & 2 & 3 \\
2 & 3 \\
4
\end{ytableau}
\end{array}, ~
T_2 =
\begin{array}{l}
\begin{ytableau}
1 & 2 & 4 \\
2 & 4 \\
3
\end{ytableau}
\end{array}
\right\}
\]
is an equivalence class in $\calE_{(3,2,1);4}$ such that $\bfG_E = \C(T_1 - T_2) \oplus \C T_2$.
It would be interesting to characterize for which $E \in \calE_{\lambda;m}$ the $H_m(0)$-module $\bfG_E$ is indecomposable.
\medskip

(3) In \cite[Example 6.7]{17PY2}, Pechenik and Yong pointed out that $U_\lambda$ is not Schur-positive for some partition $\lambda$.
Later, Pechenik \cite{20Pechenik} provided a signed Schur expansion for $U_\lambda$.
Therein, it was also proved that $U_\lambda$ is Schur-positive for all partitions $\lambda$ with $\ell(\lambda) = 2$.
To be precise, for all $\lambda = (\lambda_1, \lambda_2) \vdash n$,
\begin{align}\label{eq: schur expansion for genomic}
U_\lambda = \sum_{l_\lambda \le m \le n} \  \sum_{\mu \in \mathsf{Par}(\lambda; m)} s_{\mu},
\end{align}
where $l_\lambda := \max \{\lambda_1, \lambda_2 +1\}$, $\mathsf{Par}(\lambda; n) := \{(\lambda_1, \lambda_2)\}$, and
\begin{align*}
\mathsf{Par}(\lambda; m) & := 
\begin{cases}
\left\{(\lambda_1 - k_m, \lambda_1 - k_m, 1^{k_m})\right\} & \text{if $\lambda_1 = \lambda_2$,}\\
\left\{(\lambda_1 - k_m, \lambda_2 - k_m, 1^{k_m}), (\lambda_1 - k_m, \lambda_2 - k_m + 1, 1^{k_m-1}) \right\} & \text{if $\lambda_1 > \lambda_2$}
\end{cases}
\end{align*}
for all $l \le m < n$.
Here, $k_m := n-m$ and $s_\mu := 0$ if $\mu$ is not a partition.

In this paper, we have constructed the $H_m(0)$-module $\bfG_{\lambda;m}$ and show that $\ch([\bfG_{\lambda;m}]) = \GSm{\lambda}{m}$. 
On the other hand, for each $\alpha \models m$, Searles \cite{19Searles} introduced the $H_m(0)$-module $X_\alpha$ such that $\ch([X_\alpha]) = \mathsf{ES}_\alpha$,
where $\mathsf{ES}_\alpha$ is the \emph{extended Schur function} introduced in \cite{19AS}.
It was shown in \cite[Proposition 6.15]{19AS} that $\mathsf{ES}_\lambda = s_\lambda$ for all $\lambda \vdash m$. 
The study of the representation theoretic interpretation for \cref{eq: schur expansion for genomic} will be pursued in the near future by using the $0$-Hecke modules $\bfG_{\lambda;m}$ $(l_\lambda \le m \le n)$ and $X_\mu$ $(\mu \in \mathsf{Par}(\lambda;m))$.
In this direction, we leave the following conjecture.

\begin{conjecture}\label{conj: repn interpretation}
Let $\lambda$ be a partition with $\ell(\lambda) \le 2$.
For each $l_\lambda \le m \le n$, there exists a partition $\{\calE_\mu \mid \mu \in \mathsf{Par}(\lambda; m)\}$ of $\calE_{\lambda;m}$ satisfying the following two conditions:
\begin{enumerate}[label = {\rm (\arabic*)}]
\item For each $\mu \in \mathsf{Par}(\lambda; m)$, $\sum_{E \in \calE_\mu} \ch([\bfG_E]) = s_{\mu}$.
\item For each $\mu \in \mathsf{Par}(\lambda; m)$, there exist a total order $\prec_\mu$ on $\calE_\mu = \{E_1 \prec_\mu E_2 \prec_\mu \cdots \prec_\mu E_{|\calE_\mu|}\}$ and a filtration 
\[
M_0 = \{0\} \subseteq M_1 \subseteq M_2 \subseteq \cdots \subseteq M_{|\calE_\mu|} = X_\mu
\]
of $H_m(0)$-modules such that $\bfG_{E_i} \cong M_{i}/M_{i-1}$ for all $1 \le i \le |\calE_\mu|$.
\end{enumerate}
\end{conjecture}

Let $\lambda = (\lambda_1, \lambda_2) \vdash n$.
In case where $\lambda_1 = \lambda_2$, Pechenik defined a descent preserving bijection from $\IGLT(\lambda)$ to $\bigcup_{l_\lambda \le m \le n} \bigcup_{\mu \in \mathsf{Par}(\lambda; m)} \mathrm{SYT}(\mu)$ in \cite[Proof of Proposition 2.1]{14Pechenik}.
Here, $\mathrm{SYT}(\mu)$ is the set of standard Young tableaux of shape $\mu$.
In order to prove \cref{eq: schur expansion for genomic}, Pechenik used this bijection and a similar bijection for the case where $\lambda_1 > \lambda_2$.
For details, see \cite[Proof of Proposition 4.3]{20Pechenik}.
Using these bijections, we observe that \cref{conj: repn interpretation} is true when $|\lambda| \le 8$. Further, in this case, we also observe that the bijection induces an $H_m(0)$-module isomorphism from $\bfG_{E_i}$ to $M_{i}/M_{i-1}$ for all $1 \le i \le |\calE_\mu|$.

\section*{Acknowledgments.}
The authors are deeply grateful to the anonymous referees for their meticulous reading of the manuscript and their invaluable advice.

\bibliographystyle{abbrv}
\bibliography{references}

\end{document}